\documentclass[11pt,a4paper]{amsart}

\pdfoutput=1

\usepackage{etex}
\usepackage[english]{babel}
\usepackage{amsmath}
\usepackage{amssymb, amscd, amsthm, bbold}
\usepackage{mathrsfs}
\usepackage{placeins}
\usepackage{enumitem}
\usepackage[all]{xy}
\usepackage{tikz}
\usepackage{subfigure}

\usetikzlibrary{decorations.pathreplacing}
\usepackage{booktabs}

\usepackage{url}

\usepackage[pdftex]{hyperref}

\def\dim{{\operatorname{dim}}}
\def\id{{\operatorname{id}}}

\def \eps{\varepsilon}
\def\1{{\mathbb{1}}}
\def\C{{\mathbb C}}
\def\R{{\mathbb R}}
\def\Z{{\mathbb Z}}
\def\Q{{\mathbb Q}}
\def\N{{\mathbb N}}

\def\cC{{\mathscr C}}
\def\cD{{\mathscr D}}
\def\Com{{\mathscr Com}}
\def\O{{\mathscr O}}

\def\k{{\mathbb K}}
\def\des{{\mathcal S}^{-1}}
\def\dest{\widetilde{{\mathcal S}}^{-1}}
\def\D{{\mathcal D}}
\def\Sph{{\mathcal Sph}}
\newcommand{\OO}{\mathcal{O}}
\newcommand{\OP}{\mathcal{P}}

\newcommand{\OR}{\mathcal{R}}

\newcommand{\e}{\mathcal{E}}

\newcommand{\Ch}{\operatorname{Ch}}

\newcommand{\Nat}{\operatorname{Nat}}

\newcommand{\oc}[2]{[\begin{subarray}{c} #1 \\ #2 \end{subarray}]}
\newcommand{\bC}{\overline{C}}

\newcommand{\bD}{\overline{D}}
\newcommand{\bNat}{\overline{\Nat}}

\newcommand{\SD}{\mathcal{SD}}
\newcommand{\lD}{l\mathcal{D}^{>0}}
\newcommand{\lDa}{l\mathcal{D}}
\newcommand{\lDc}{l\mathcal{D}^{cst}}
\newcommand{\ilD}{il\mathcal{D}^{>0}}
\newcommand{\ilDa}{il\mathcal{D}}
\newcommand{\plD}{pl\mathcal{D}^{>0}}
\newcommand{\plDa}{pl\mathcal{D}}
\newcommand{\plDc}{pl\mathcal{D}^{cst}}

\newcommand{\iplDa}{ipl\mathcal{D}}

\newcommand{\tiplDa}{\widetilde{ipl\mathcal{D}}_\Com}
\newcommand{\tplDa}{\widetilde{pl\mathcal{D}}_\Com}

\numberwithin{equation}{section}
\newtheorem{satz}{Satz}[section]
\newtheorem{Theorem}[satz]{Theorem}
\newtheorem{Cor}[satz]{Corollary}
\newtheorem{Lemma}[satz]{Lemma}
\newtheorem{Prop}[satz]{Proposition}
\newtheorem{Conjecture}[satz]{Conjecture}
\theoremstyle{definition}
\newtheorem{Remark}[satz]{Remark}
\newtheorem{Example}[satz]{Example}
\newtheorem{Def}[satz]{Definition}

\newtheorem{Th}{Theorem}
\newtheorem{Co}[Th]{Corollary}

\title[Operations on $C_*(A,A)$ for commutative Frobenius algebras]{Natural operations on the Hochschild complex of commutative Frobenius algebras via the complex of looped diagrams}

\address{Angela Klamt, Department of Mathematical Sciences, 
         University of Copenhagen, 
         Universitetsparken 5,
       2100 Copenhagen, 
         Denmark}
\email{angela@math.ku.dk}\date{\today}
\author{Angela Klamt}

\begin{document}

\begin{abstract}
We define a dg-category of looped diagrams which we use to construct operations on the Hochschild complex of commutative Frobenius dg-algebras. We show that we recover the operations known for symmetric Frobenius dg-algebras constructed using Sullivan chord diagrams as well as all formal operations for commutative algebras (including Loday's lambda operations) and prove that there is a chain level version of a suspended Cactus operad inside the complex of looped diagrams. This recovers the suspended BV algebra structure on the Hochschild homology of commutative Frobenius algebras defined by Abbaspour and proves that it comes from an action on the Hochschild chains.
\end{abstract}
\maketitle
\section*{Introduction}
To a dg-algebra $A$ one associates the Hochschild chain complex $C_*(A,A)$. Operations of the form
\[C_*(A,A)^{\otimes n_1} \otimes A^{\otimes m_1} \to C_*(A,A)^{\otimes n_2} \otimes A^{\otimes m_2}\]
have been investigated by many authors. We are interested in those operations which exist for all algebras of a certain class, more concretely all algebras over a given operad or PROP. In \cite{wahl12}, the complex of so-called formal operations is introduced, a more computable complex approximating the complex of all natural operations for a given class of ($A_\infty$--)algebras. In this paper we build a combinatorial complex mapping to the complex of formal operations for the case of commutative Frobenius dg-algebras, defining in particular a large family of operations for commutative Frobenius dg-algebras.

Before describing the complex of operations, we recall some of the operations known so far, which we want to be covered by our new complex. 
 
One of the main motivations for investigating operations on Hochschild homology is given by string topology. String topology started in 1999 when Chas and Sullivan in \cite{chas99} gave a construction of a product $H_*(LM) \otimes H_*(LM) \to H_{*-d}(LM)$ for $M$ a closed oriented manifold of dimension $d$ and $LM$ the free loop space on $M$, that makes $H_*(LM)$ into a BV-algebra. Afterward, more operations were discovered and in \cite{godi07} the structure of an open-closed HCFT was exhibited on the pair $(H_*(M), H_*(LM))$, yielding a whole family of operations
\[H_*(LM)^{\otimes n_2} \otimes H_*(M)^{\otimes m_2} \to H_*(LM)^{\otimes n_1} \otimes  H_*(M)^{\otimes m_1}\]
parametrized by the moduli space of Riemann surfaces.

Taking coefficients in a field and $M$ to be a 1-connected closed oriented manifold, we have an isomorphism \[HH_*(C^{-*}(M), C^{-*}(M)) \cong H^{-*}(LM).\] Moreover, if $M$ is a formal manifold, i.e. if we have a weak equivalence $C^*(M) \simeq H^*(M)$ preserving the multiplication, we obtain an isomorphism $HH_*(C^{-*}(M), C^{-*}(M)) \cong HH_*(H^{-*}(M), H^{-*}(M))$. Thus, under these conditions, operations
\[HH_*(H^{-*}(M), H^{-*}(M))^{\otimes n_1} \otimes H^{-*}(M)^{\otimes m_1} \to HH_*(H^{-*}(M), H^{-*}(M))^{\otimes n_2} \otimes H^{-*}(M)^{\otimes m_2}\]
are equivalent to operations
\[H^{-*}(LM)^{\otimes n_1} \otimes H^{-*}(M)^{\otimes m_1} \to H^{-*}(LM)^{\otimes n_2} \otimes  H^{-*}(M)^{\otimes m_2},\]
which are dual to the operations we ask for in string topology. On the other hand, $H^{-*}(M)$ is a commutative Frobenius algebra, thus constructing operations on the Hochschild homology of commutative Frobenius algebras gives us (dual) string operations. This correspondence can be applied even more generally. Working with a field of characteristic zero and taking the deRham complex $\Omega^\bullet(M)$ instead of singular cochains, in \cite{lamb07} Lambrechts and Stanley prove that there is a commutative differential graded Poincar\'{e}  duality  algebra $A$  weakly equivalent to $\Omega^\bullet(M)$. A Poincar\'{e}  duality algebra is a graded version of  a commutative Frobenius algebra. Hence, the Hochschild complex is isomorphic to $HH_*(A^{-*}, A^{-*})$ and string operations on $H^{-*}(LM)$ correspond to operations on the Hochschild homology of $HH_*(A^{-*}, A^{-*})$.

In \cite{gore09} Goresky and Hingston investigate a product on the relative cohomology $H^*(LM, M)$ that is an operation which is not part of the HCFT mentioned above. We define a product on the Hochschild homology of commutative Frobenius dg-algebras and show that it is part of a shifted BV-structure. Such a product also occurs in \cite[Section 7]{abba13} and \cite[Section 6]{abba13b}. Simultaneously with the aforementioned paper we conjecture that this product is the operation corresponding to the Goresky-Hingston product under the above isomorphism (see Conjecture \ref{conj:GH}).

Since commutative Frobenius algebras are in particular symmetric, we want our complex to recover all operations known for symmetric Frobenius algebras. In \cite{trad06} Tradler and Zeinalian show that a certain chain complex of Sullivan chord diagrams acts on the Hochschild cochain complex of a symmetric Frobenius algebra (a dual construction on the Hochschild chains was done by Wahl and Westerland in \cite{wahl11}). In \cite[Theorem 3.8]{wahl12} this complex is shown to give all formal operations for symmetric Frobenius algebras up to a split quasi-isomorphism.

On the other hand, every commutative Frobenius dg-algebra is of course a differential graded commutative algebra. In \cite{kla13} we give a description of the homology of all formal operations for differential graded commutative algebras in terms of Loday's shuffle operations (defined in \cite{loda89}) and the Connes' boundary operator. Well-known operations which are covered in this complex are Loday's $\lambda$--operations and the shuffle product $C_*(A,A) \otimes C_*(A,A) \to C_*(A,A)$. 

The chain complex of operations on commutative Frobenius dg-algebras constructed in this paper recovers all the operations just mentioned: The shifted BV-structure, the operations coming from Sullivan diagrams and the more classical operations on the Hochschild chains of commutative algebras. 
In addition, this complex provides a large class of other non-trivial  operations and can be used to compute relations between the previously known operations. 

\medskip

We now present our results in more detail. The main new object introduced in this paper is what we call a \emph{looped diagram}. A looped diagram of type $\oc{n_2}{(m_1+m_2, n_1)}$ is a pair $(\Gamma,\cC)$ where $\Gamma$ can be described as an equivalence class of one-dimensional  cell complexes (a ``commutative Sullivan diagram'') built from $n_2$ circles by attaching chords with $n_1+m_1+m_2$ marked points on the circles and $\cC$ is a collection of $n_1$ loops on $\Gamma$ starting at the marked points labeled $1$ to $n_1$. An example of a $\oc{1}{(2, 2)}$--looped diagram is given in Figure \ref{fig:sullivan1}.

\begin{figure}[!ht]

 \center
\input{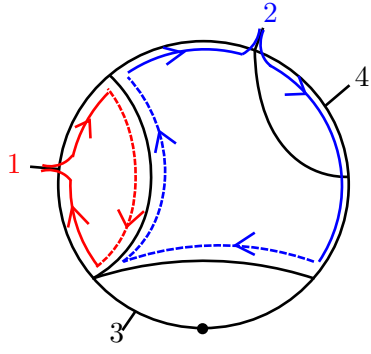}

	\caption[A looped Sullivan diagram]{A $\oc{1}{(2, 2)}$--looped diagram}
		\label{fig:sullivan1}
\end{figure}
The set of looped diagrams forms a multi-simplicial set with the boundary maps given by identifying neighbored marked points on the circles and taking the induced loops. The corresponding reduced chain complex defines $\lDa(\oc{n_1}{m_1},\oc{n_2}{m_2})$, the chain complex of looped diagrams.
Inside $\lDa(\oc{n_1}{m_1},\oc{n_2}{m_2})$, we have a subcomplex $\lDa_+(\oc{n_1}{m_1},\oc{n_2}{m_2})$ of diagrams with an open boundary condition which we use to construct operations on commutative cocommutative open Frobenius dg-algebras (commutative Frobenius algebras without a counit). In both complexes we can compose elements, i.e. we in fact construct dg-categories $\lDa$ and $\lDa_+$ with objects $\N \times \N$ and morphism spaces $\lDa(\oc{n_1}{m_1},\oc{n_2}{m_2})$ and $\lDa_+(\oc{n_1}{m_1},\oc{n_2}{m_2})$ and show:
%
\begin{Th}[{see Theorem \ref{Th:buildop}}]\label{Th:intro1}
For any commutative Frobenius dg-algebra $A$  there is a map of chain complexes
\[C_*(A,A)^{\otimes n_1} \otimes A^{\otimes m_1} \otimes \lDa(\oc{n_1}{m_1},\oc{n_2}{m_2}) \to C_*(A,A)^{\otimes n_2} \otimes A^{\otimes m_2}\]
natural in $A$ and commuting with the composition in $\lDa$. 

For any commutative, cocommutative open Frobenius dg-algebra $A$, we have a chain map
\[C_*(A,A)^{\otimes n_1} \otimes A^{\otimes m_1} \otimes \lDa_+(\oc{n_1}{m_1},\oc{n_2}{m_2}) \to C_*(A,A)^{\otimes n_2} \otimes A^{\otimes m_2}\]
natural in $A$ and preserving the composition of $\lDa$.
Moreover, all these operations are formal operations in the sense of  \cite[Section 2]{wahl12}.
\end{Th}
Moving on, we enlarge  $\lDa(\oc{n_1}{m_1},\oc{n_2}{m_2})$ to a complex $\ilDa(\oc{n_1}{m_1},\oc{n_2}{m_2})$, where we take products over specific types of diagrams. In this complex not all elements are composable, but we can show:
\begin{Th}[{see Theorem \ref{Th:buildop2}}] \label{Th:intro2}
For $A$ a commutative Frobenius dg-algebra there is a map of chain complexes
\[C_*(A,A)^{\otimes n_1} \otimes A^{\otimes m_1} \otimes \ilDa(\oc{n_1}{m_1},\oc{n_2}{m_2}) \to C_*(A,A)^{\otimes n_2} \otimes A^{\otimes m_2}\]
natural in $A$ and commuting with the composition of composable elements in $\ilDa$. Again, for $A$ a commutative, cocommutative open Frobenius dg-algebra, we have a chain map
\[C_*(A,A)^{\otimes n_1} \otimes A^{\otimes m_1} \otimes \ilDa_+(\oc{n_1}{m_1},\oc{n_2}{m_2}) \to C_*(A,A)^{\otimes n_2} \otimes A^{\otimes m_2}\]
natural in $A$.
\end{Th}

We first explain how we recover the operations known from symmetric Frobenius algebras. We denote the complexes of Sullivan diagrams by $\SD(\oc{n_1}{m_1},\oc{n_2}{m_2})$ (see Section \ref{sec:graphs} for a definition). As already mentioned, Sullivan diagrams define natural operations for symmetric Frobenius algebras, i.e. for $A$ a symmetric Frobenius dg-algebra there is an action
\[C_*(A,A)^{\otimes n_1} \otimes A^{\otimes m_1} \otimes \SD(\oc{n_1}{m_1},\oc{n_2}{m_2}) \to C_*(A,A)^{\otimes n_2} \otimes A^{\otimes m_2}\]
natural in $A$. There is a canonical way to build a looped diagram out of a Sullivan diagram $\Gamma$ and thus there is a dg-map $K:\SD(\oc{n_1}{m_1},\oc{n_2}{m_2}) \to \lDa(\oc{n_1}{m_1},\oc{n_2}{m_2})$ commuting with the composition of diagrams. 
In Proposition \ref{prop:commsym} we show that both actions (the one of $\SD$ and the one of $\lDa$) are compatible, i.e. that given a commutative Frobenius dg-algebra $A$ the diagram
\[\xymatrix{ C_*(A,A)^{\otimes n_1} \otimes A^{\otimes m_1} \otimes \SD(\oc{n_1}{m_1},\oc{n_2}{m_2})  \ar[d]_-{\id^{\otimes n_1} \otimes \id^{\otimes m_1} \otimes K} \ar[rd]& & \\C_*(A,A)^{\otimes n_1} \otimes A^{\otimes m_1} \otimes \lDa(\oc{n_1}{m_1},\oc{n_2}{m_2})  \ar[r]&C_*(A,A)^{\otimes n_2} \otimes A^{\otimes m_2}  }\]
commutes.

Second, restricting to those diagrams which only have leaves and no chords glued in, we define a subcomplex $\iplDa_{\Com}(\oc{n_1}{m_1}, \oc{n_2}{m_2})$ of $\ilDa_+(\oc{n_1}{m_1}, \oc{n_2}{m_2})$. This subcomplex defines operations for commutative algebras and contains Loday's lambda operations. 
More precisely, by \cite[Theorem 3.4]{kla13} it gives all formal operations of differential graded commutative algebras up to quasi-isomorphism. So in particular the complex of looped diagrams includes (up to quasi-isomorphism) all formal operations on the Hochschild chains of differential graded commutative algebras.

Last, we define another subcomplex $\plD_{cact}(n_1, n_2) \subset \lDa_+(\oc{n_1}{0}, \oc{n_2}{0})$ with $\plD_{cact}(n_1, n_2)$ the PROP coming from an operad $\plD_{cact}(n, 1)$. We show that a topological version of this operad is homeomorphic to a topologically desuspended Cacti operad and hence deduce:
\begin{Th}[{see Theorem \ref{Th:shiftbv}}]\label{Th:C}
The complex $\plD_{cact}(n, 1)$ is a chain model for the twisted operadic desuspended BV-operad,
i.e. \[H_*(\plD_{cact}(-,1)) \cong \widetilde{s}^{-1}BV\]
as graded operads.
\end{Th}
Here $\widetilde{s}^{-1}$ denotes a desuspension with twisted sign. 
The sign twist comes from the fact that we actually work with topological operads and suspend topologically by smashing with the sphere operad (see Definition \ref{def:susp}).
As a corollary of the above theorem we can deduce:
\begin{Co}[{see Corollary \ref{Cor:action}}]\label{Cor:D}
There is a desuspended BV-algebra structure on the Hochschild homology of a commutative cocommutative open Frobenius dg-algebra (in particular on the Hochschild homology of a commutative Frobenius dg-algebra) which comes from an action of a chain model of the suspended Cacti operad on the Hochschild chains.
\end{Co}

\medskip

The paper is organized as follows: The combinatorics used in the paper are given in Section \ref{sec:combi}. We start with recalling the definitions of black and white graphs and Sullivan diagrams in Section \ref{sec:graphs}. In Section \ref{sec:comsul} we define looped diagrams and show that $\lDa$ and $\lDa_+$ are well-defined dg-categories. The two following sections \ref{sec:nonconst} and \ref{sec:type} are very technical and not needed for the actual construction of operations (they will be used to construct the commutative operations and the action of the Cacti operad). We suggest the reader to skip them and come back later, if needed. More precisely, in Section \ref{sec:nonconst} we prove that the subcomplex of diagrams with a constant loop is a split subcomplex and investigate how the composition looks like on the split complement of non-constant diagrams. In Section \ref{sec:type} we give a finer subdivision of $\lD$ on the level of vector spaces and afterward take the product over all the subspaces to get the complexes $\ilDa(\oc{n_1}{m_1}, \oc{n_2}{m_2})$ and $\ilDa_+(\oc{n_1}{m_1}, \oc{n_2}{m_2})$. Composition is not well-defined on these complexes, but we give some subcomplexes for which composition with every other element is well-defined.
Section \ref{sec:natcom} deals with the formal operations on the Hochschild chains of commutative Frobenius dg-algebras. We start by recalling the definition of Frobenius algebras in Section \ref{sec:comfrob} and the definition and results on formal operations in Section \ref{sec:formal}. In Section \ref{sec:build} we explain how to build formal operations out of looped diagrams, i.e. prove Theorem \ref{Th:intro1} and Theorem \ref{Th:intro2}. In Section \ref{sec:con} we investigate the connection to the operations on symmetric Frobenius algebras stated above. Finally, in Section \ref{sec:exop} we show how the shuffle product, the Chas-Sullivan coproduct, the BV-operator and the shifted commutative product defined in \cite[Section 7]{abba13} and \cite[Section 6]{abba13b} look like in terms of looped diagrams and use the techniques of looped diagrams to prove a relation between the new product and the BV-operator.
In Section \ref{sec:com} we define the subcomplexes of graphs giving the operations of commutative algebras and recall \cite[Theorem 3.4]{kla13} in terms of these diagrams. In Section \ref{sec:cacti}  we define the complex $\plD_{cact}(n_1, n_2)$, prove Theorem \ref{Th:C} and thus obtain the action of a desuspended cacti operad on the Hochschild chains of a commutative Frobenius dg-algebra (see Corollary \ref{Cor:D}).

In Appendix \ref{sec:overview} we have listed all complexes defined in the paper together with a short explanation and reference. We hope that this is helpful to keep track of the definitions throughout the paper.

\subsection*{Acknowledgements}
I would like  to thank my advisor Nathalie Wahl for many helpful discussions, questions and comments. I would also like to thank Daniela Egas Santander for fruitful discussions.
The author was supported by the Danish National Research Foundation through the Centre for Symmetry and Deformation (DNRF92).

\tableofcontents
\subsection*{Conventions}
If not specified otherwise we work in the category $\Ch$ of chain complexes over $\Z$. We use the usual sign convention on the tensor product, i.e. the differential $d_{V \otimes W}$ on $V \otimes W$ is given by $d_{V \otimes W}(v \otimes w)=d_V(v) \otimes w + (-1)^{|v|}v \otimes d_W (w)$.

A \emph{dg-category} $\e$ is a category enriched over chain complexes, i.e. the morphism sets are chain complexes. We use composition from the right, i.e. we require the composition maps $\e(m,n) \otimes \e(n,p) \to \e(m,p)$ to be chain maps. A \emph{dg-functor} is an enriched functor $\Phi: \e \to \Ch$, so the structure maps
$\Phi(m) \otimes \e(m,n) \to \Phi(n)$
are chain maps.

Given a chain complex $A$ we denote by $A[k]$ the shifted complex with $(A[k])_n=A_{n-k}$. Throughout the paper the natural numbers are always assumed to include zero.

\section{Definitions of graph complexes}\label{sec:combi}
In this section we define the chain complex of looped diagrams and its subcomplex of positive diagrams. The complexes are an extension of a quotient of the chain complex of Sullivan diagrams which we first recall. We mainly follow \cite[Section 2]{wahl11}. 
\subsection{Graphs}\label{sec:graphs}

A \emph{graph} is a tuple $(V,H,s,i)$ with $V$ the vertices, $H$ the half-edges, $s:H \to V$ the source map and $i:H \to H$ an involution. A half-edge is a \emph{leaf} if it is a fixed point under $i$. A \emph{fat graph} is a graph with a cyclic ordering of the half-edges at the vertices. The cyclic orderings define \emph{boundary cycles} on the graph which correspond to the boundary cycles of the surface one gets by thickening the graph (for more details see \cite[Section 2.1]{wahl11}).

In the graphical representation the half-edges are glued to the vertices using $s$ and to each other using $i$. 

An \emph{orientation} of a graph is a unit vector in $\operatorname{det}(\R(V \amalg H))$. Note that any odd-valent fat graph has a canonical orientation and that an orientation of a fat graph with even-valent vertices is given by an ordering of the (even-valent) vertices together with a choice of a start half-edge $h_1^i$ for each (even-valent) vertex (changing the position of the odd-valent vertices in the ordering or changing the choice of their start half-edge does not change the orientation). Denoting the half-edges belonging to a vertex $v_i$ by $h_j^i$ starting from the start half-edge and following the cyclic ordering, the canonical orientation is given by
\[v_1 \wedge h_1^1 \wedge \cdots h_{n_1}^1 \wedge v_2 \wedge \ldots \wedge h_{n_r}^r.\]

A \emph{black and white graph} is an oriented fat graph where we label the vertices black or white and allow the white vertices to have any positive valence, whereas the black vertices are requested to have valence at least three. The white vertices are ordered and each white vertex is equipped with a choice of a start half-edge. A \emph{$\oc{p}{m}$--graph} is a black and white graph with $p$ white vertices and $m$ labeled leaves, quotiening out the equivalence relation of forgetting unlabeled leaves which are not the start half-edge of a white vertex. In the graphical representation we mark the start half-edges by black blocks. An example of a $\oc{3}{2}$--graph is given in Figure \ref{fig:BW2}. A special example of a $\oc{1}{n}$--graph is the graph $l_n$ which will play a crucial role in defining operations. This graph is given by attaching $n$ leaves to the white vertex and labeling them starting from the start half-edge (for an example see Figure \ref{fig:ln}). In general we omit the label $v_1$ in the pictures if there is only one white vertex.

 \begin{figure}[ht!]
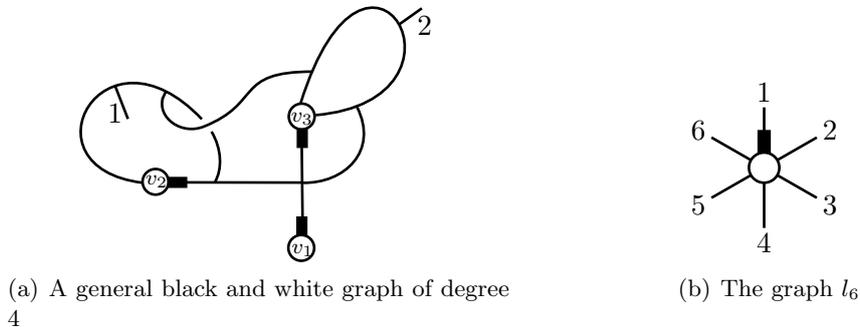

 
 \vspace{-1cm}
     \begin{center}
        \subfigure
        [A general black and white graph of degree $4$]
        {%
            \label{fig:BW2}
            \input{pic_BW2}
      
        }%
              \hspace{2cm}
        \subfigure
       [The graph $l_6$]
        {%
           \label{fig:ln}
           \input{pic_ln}
           
        }
    \end{center}
   \caption[Examples of black and white graphs]{
        A $\oc{3}{2}$--graph and the $\oc{1}{6}$--graph $l_6$}
\end{figure}

The \emph{degree} of a black vertex of valence $v_b$ is given by $v_b-3$ whereas the degree of a white vertex $v_w$ is defined as $v_w-1$. The degree of a black and white graph is the sum of the degrees over all its vertices.

The differential of a black and white graph is given by the sum of all graphs obtained by blowing up all vertices of degree at least $1$ in all possible ways, i.e. splitting the set of half-edges attached to the vertex into two subsets (respecting the cyclic ordering and with at least $2$ elements in each if the vertex was black) and adding an edge in between these. For more details on the differential and examples see \cite[Section 2.5]{wahl11}.

The chain complex of \emph{$\oc{p}{n}$--Sullivan diagrams} is defined as a quotient of the above complex of $\oc{p}{n}$--graphs by the subcomplex spanned by the graphs with at least one black vertex of valence at least $4$ and the boundaries of these graphs. Hence an element in this complex is an equivalence class of graphs with all its black vertices of valence exactly $3$ and the equivalence relation is generated by the relation shown in Figure \ref{fig:equiv}.

\begin{figure}[!ht]

 \center
\input{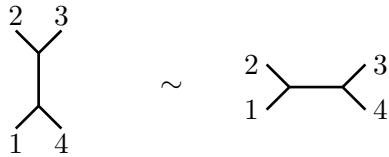}

	\caption{The equivalence relation on Sullivan diagrams}
		\label{fig:equiv}
\end{figure}

In \cite[Theorem 2.7]{wahl11} it was shown that this complex is isomorphic to the complex of \emph{Cyclic Sullivan chord diagrams} defined in \cite[Def. 2.1]{trad06}.

The dg-category $\SD$ is defined to have pairs of natural numbers $\oc{n}{m}$ as objects and morphism complexes $\SD(\oc{n_1}{m_1}, \oc{n_2}{m_2}) \subset \oc{n_1}{n_1+m_1+m_2}$--Sullivan diagrams the subcomplex of the graphs with the first $n_1$ leaves being sole labeled leaves in their boundary cycle. The composition is defined in \cite[Section 2.8]{wahl11}. 

We want to define looped diagrams as an enlargement of a quotient of these.

\subsection{Commutative Sullivan diagrams and looped diagrams}\label{sec:comsul}
\begin{Def}
A \emph{$\oc{p}{m}$--commutative Sullivan diagram} is an equivalence class of $\oc{p}{m}$--Sullivan diagrams by forgetting the ordering at the black vertices, i.e. the chain complex $\oc{p}{m}-CSD$ of $\oc{p}{m}$--commutative Sullivan diagrams is the quotient of $\oc{p}{m}$--graphs by 
\begin{itemize}
\item graphs with black vertices of valence at least 4,
\item the boundaries of such graphs and
\item reordering of the half-edges at black vertices.
\end{itemize}
\end{Def}
An example of two equivalent $\oc{1}{1}$--commutative Sullivan diagrams is given in Figure \ref{fig:equivsullivan}.
\begin{figure}[!ht]
\vspace{-1cm}

 \center
\input{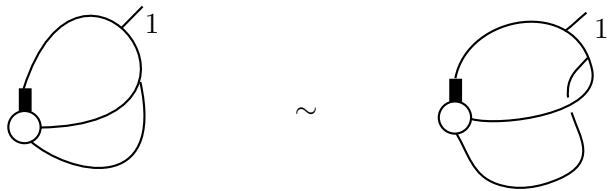}
\caption{Two equivalent commutative Sullivan diagrams}
	\label{fig:equivsullivan}
\end{figure}
\begin{Remark}
Black and white graphs were defined with an orientation. As mentioned above, any trivalent graph has a canonical orientation and so does every white vertex (starting from the start half edge). Thus every Sullivan diagram has a canonical orientation. The relation we divide out is the commutativity relation together with the canonical orientation of the two graphs. From now on, we always work with the canonical orientation. However, the orientation was also used to define the sign of the differential and the sign of the composition. We will spell out these sign explicitly, too.
\end{Remark}

In order to be able to define a composition, we need an additional structure analogous to the one we lost going from Sullivan diagrams to commutative Sullivan diagrams. We define a larger category of looped diagrams which includes all Sullivan diagrams without free boundary as a subcategory. 

Given a white vertex in a black and white graph with $k$ half-edges attached to it, the half-edges cut the circle around the white vertex into $k$ parts. Given a labeling of the white vertices $v_1, \cdots, v_p$ we label the segments at the vertex $v_i$ by $s^i_{1}, \cdots, s^i_{k_i}$ following the ordering of the half-edges at $v_i$. These segments inherit a canonical orientation from the ordering at the white vertex and hence we can talk about their start and end. By $-s_j^i$, we mean the segment with the opposite orientation, i.e. start and end got interchanged. An example of the labeling of segments for one white vertex is given in Figure \ref{fig:labelloop}. An \emph{arc component} of a black and white graph is a set of half-edges and black vertices which is path-connected with the paths in the graph not passing through a white vertex (i.e. in pictures, a connected component of the graph after ``deleting'' the white vertices). For example, the commutative Sullivan diagram in Figure \ref{fig:equivsullivan} has one arc component, whereas the underlying diagram of Figure \ref{fig:labelloop} has two arc components (the one with the labeled leaf and the one without). 
\begin{figure}[ht!]
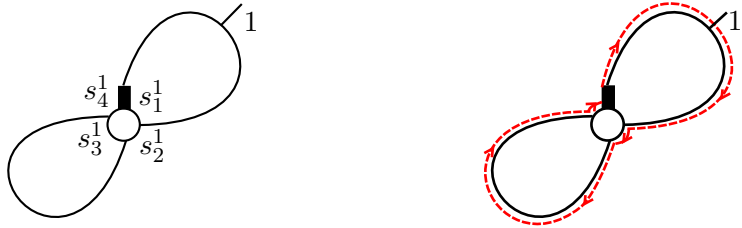

\vspace{-1cm}
     \begin{center}
        \subfigure
        [The segments around one white vertex]
        {%
            \label{fig:labelloop}
            \input{pic_segments}
        }%
        \subfigure
       [The loop $\gamma=\{s^1_2, s^1_4\}$]
        {%
           \label{fig:labelloop2}
           \input{pic_looped1}
        }
    \end{center}
   \caption{
        A commutative Sullivan diagram with segments and a loop starting from leaf $1$
     }
   \label{fig:subfigures}
\end{figure}
\begin{Def}
A \emph{loop in a $\oc{p}{m}$--commutative Sullivan diagram} from an arc component $a$ to itself is an ordered set of oriented segments of the boundary circles of the white vertices $\{\eps_1 s^{i_1}_{t_1}, \cdots, \eps_r s^{i_r}_{t_r}\}$ for $r \geq 0$,  $1 \leq i_1, \cdots, i_r \leq p$ and $\eps_i \in \{-1, 1\}$ such that the following conditions hold:
\begin{enumerate}
\item $\eps_1 s^{i_1}_{t_1}$ starts at the arc component $a$.
\item $\eps_r s^{i_r}_{t_r}$ ends at the arc component $a$.
\item $\eps_{w+1} s^{i_{w+1}}_{t_{w+1}}$ starts at the arc component at which $\eps_w s^{i_w}_{t_w}$ ends (which can be at a different white vertex!).
\item $\eps_w s^{i_w}_{t_w} \neq -\eps_{w+1} s^{i_{w+1}}_{t_{w+1}}$.
\end{enumerate}
The loop is called \emph{constant} if the set of segments is empty, i.e. if $r=0$.

A loop from a leaf $k$ to itself is a loop starting at the arc component the leaf belongs to.

The composition $\gamma_1 \ast \gamma_2$ of loops $\gamma_1$ and $\gamma_2$ both starting at the same leaf $k$ is the concatenation of these two loops (see Figure \ref{fig:juxt}).

A loop $\gamma$ starting at the leaf $k$ is called \emph{irreducible} if it cannot be written as the composition of two non-trivial loops (i.e. the loop does not return to the arc component of $k$ before it finishes).


A \emph{positively oriented} loop in a $\oc{p}{m}$--commutative Sullivan diagram is a loop such that all orientations of the boundary segments are positive (i.e. $\eps_w=1$ for all $w$). 

\end{Def}

We draw a loop from a leaf by starting at the leaf and marking the segments of the white vertex (with orientation). To keep track of their ordering, we also draw the loop through arc components (dotted) even though this is not part of the data (i.e. changing the way we walk through an arc component does not change the loop). This way, a loop in a diagram is really represented by a loop in the picture. An example is given in Figure \ref{fig:labelloop2}.

\begin{figure}[ht!]
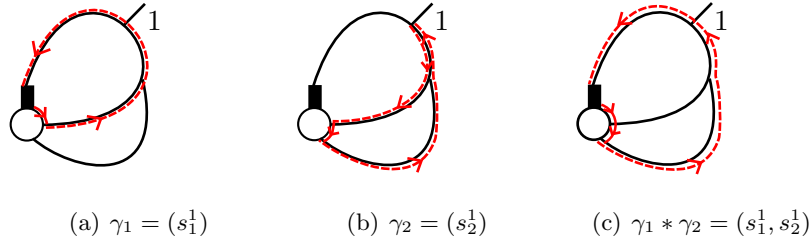

\vspace{-1cm}  \begin{center}
 \subfigure[center]
        [$\gamma_1=(s^1_1)$]
        {%
          
            \input{pic_gamma1}
        }%
\subfigure
        [$\gamma_2=(s^1_2)$]
        {
            
            \input{pic_gamma2}

        } \nobreak
         \subfigure
        [$\gamma_1 \ast \gamma_2=(s^1_1, s^1_2)$]
        {%
           
            \input{pic_juxt2}
        }%
    \end{center}
   \caption{
        Juxtaposition of loops
     }
   \label{fig:juxt}
\end{figure}

\begin{Def}
A \emph{$\oc{p}{(m,n)}$--looped diagram} is a $\oc{p}{n+m}$--commutative Sullivan diagram with $n$ loops such that the $i$-th loop starts at the $i$-th labeled leaf.
\end{Def}
An example of a $\oc{1}{(0,2)}$--looped diagram and an example of a $\oc{3}{(0,2)}$--looped diagram are given in Figure \ref{fig:loopeddiagrams}. To indicate which path starts at which leaf we color the label of the leaf with the same color as the corresponding loop.
\begin{figure}[ht!]
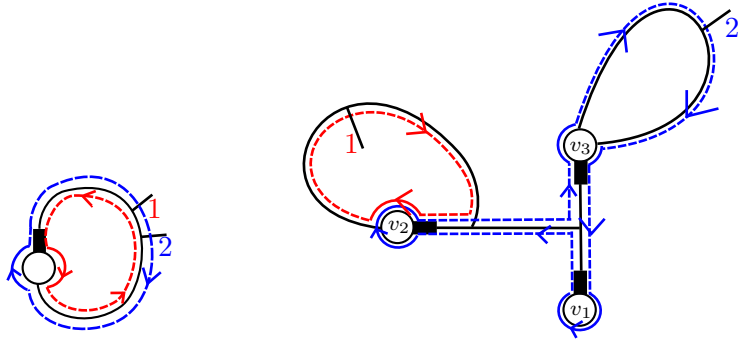

     \begin{center}
\vspace{-1cm} 
        \subfigure
        {%
            
            \input{pic_loopeddiagram1}
        }%
        \subfigure
        {%
           \input{pic_loopeddiagram2}
        }
    \end{center}
   \caption{
       Two looped diagrams
     }
   \label{fig:loopeddiagrams}
\end{figure}

The decomposition of the underlying commutative Sullivan diagram into connected components gives a decomposition of the looped diagram into connected components, since every loop has to stay in a connected component.

We want to make a complex out of these diagrams and thus need to define a differential. The differential is defined just as for Sullivan diagrams, where we blow up every possible pair of neighbored vertices at the white vertex with alternating sign (this is equivalent to the sign given by orientations, cf. \cite[Section 2.10]{wahl11}). 

For a $\oc{p}{(m,n)}$--looped diagram $(\Gamma, \gamma_1, \cdots, \gamma_n)$ the $i$-th blow up at the white vertex $v_k$ is the $\oc{p}{(m,n)}$--looped diagram $(\Gamma', \gamma'_1, \cdots, \gamma'_n)$ where $\Gamma'$ is the blow up of $\Gamma$ (as explained in the end of Section \ref{sec:graphs}) and the $\gamma'_l$ are given by the $\gamma_l$ after forgetting the segment $s^k_i$ whenever it was part of the loop and relabeling all $s^k_j$ for $j>i$.
We define the differential $d$ to be the sum of all blow ups with the sign at each vertex $v_i$ alternating and starting with $(-1)^{1+\sum_{j<i} |v_j|}$. Examples are given in Figure \ref{fig:diff} and Figure \ref{fig:multidiff}.
\begin{figure}[!ht]
 \center
\input{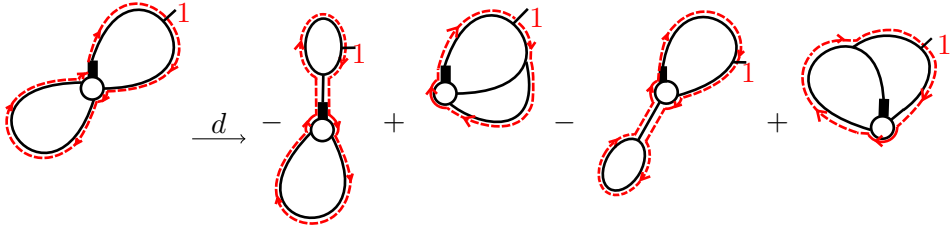}
	\caption{The differential of a looped diagram (with one loop)}
		\label{fig:diff}
\end{figure}

\begin{Lemma}\label{le:chcplx}
The map $d$ defines a differential on $\oc{p}{(m,n)}$--looped diagrams. The chain complex of these diagrams is called $\oc{p}{(m,n)}$--lD.
\end{Lemma}
\begin{proof}
One checks that the complex is the reduced chain complex of a $p$--multisimplicial set obtained by allowing unlabeled leaves at arbitrary positions. At each white vertex $v_k$ we have a simplicial set with boundary maps $d_i$ blowing up the $(i-1)$-st and $i$-th incoming leaves (and doing the induced procedure to the loops) and degeneracies adding in a unit leaf in between the $i$-th and $(i+1)$-st vertex and replacing $s^k_i$ by the two pieces (and relabeling the other segments at that vertex). The simplicial identities hold.
\end{proof}

\begin{Remark} \label{Re:addleaves}
In the previous proof we added a leaf to a vertex, replaced the corresponding boundary segment by the two new pieces and relabeled the others. Morally, we did nothing to our loop (cf. Figure \ref{fig:addleaf}), but being precise, given a looped diagram $(\Gamma, \gamma)$ after adding a leaf to the white vertex $v_k$ in between the old edges $j$ and $j+1$ we get the diagram $(\Gamma', \gamma')$ as follows: The commutative Sullivan diagram $\Gamma'$ is the diagram $\Gamma$ with the extra half-edge. The new loop $\gamma'$ is obtained from $\gamma$ by replacing all $s^k_i$ for $i>j$ by $s^k_{i+1}$ and if $s^k_j$ was in $\gamma$ with positive orientation, it is replaced by $s^k_j$ followed by $s^k_{j+1}$ and if it appeared with negative orientation by $-s^k_{j+1}$ and $-s^k_j$.
%
%
%
\end{Remark}
\begin{figure}[!ht]

 \center
\input{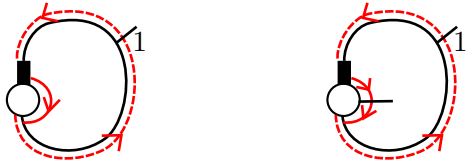}
	\caption{Adding a leaf}
		\label{fig:addleaf}
\end{figure}

\begin{Remark}
The set of  $\oc{p}{(m,n)}$--looped diagrams with all loops positively oriented is a subcomplex of the whole complex. We denote this chain complex by $\oc{p}{(m,n)}$--plD. 
\end{Remark}

We now construct a dg-category out of $\oc{p}{(m,n)}$--lD, extending the category of Sullivan diagrams.

\begin{Def}\label{def:lD} Let $\lDa$ be the dg-category of looped diagrams with objects pairs of natural numbers $\oc{n}{m}$ and morphism $\lDa(\oc{n_1}{m_1}, \oc{n_2}{m_2})$ given by the chain complex $\oc{n_2}{(m_1+m_2, n_1)}$--lD and composition defined by
\begin{figure}[!ht]

 \center
\input{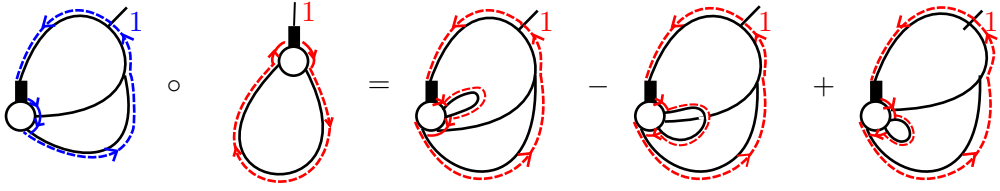}
	\caption{Composition $\circ$}
		\label{fig:comp2}
\end{figure}
taking elements $x=(\Gamma, \gamma_1, \cdots, \gamma_{n_1}) \in \lDa(\oc{n_1}{m_1}, \oc{n_2}{m_2})$ and $y=(\Gamma', \gamma'_1, \cdots, \gamma'_{n_2}) \in \lDa(\oc{n_2}{m_2}, \oc{n_3}{m_3})$ to the element
$y \circ x=(\widetilde{\Gamma}, \widetilde{\gamma}_1, \cdots,  \widetilde{\gamma}_{n_1}) \in \lDa(\oc{n_1}{m_1}, \oc{n_3}{m_3})$ 
which is zero if there is an $i$ such that the loop $\gamma'_i$ is constant and the vertex $v_i$ in $\Gamma$ has more than one half-edge attached to it and else is computed as the sum of all possible $([ G ], g_1, \cdots,  g_{n_1})$ (with a sign) obtained as follows
\begin{enumerate} \item The commutative Sullivan diagrams $[ G ]$ are the equivalence classes of black and white graphs obtained from $x$ and $y$  by
\begin{enumerate}
\item removing the $n_2$ white vertices from $\Gamma$,
\item For each $i=1, \cdots, n_2$ identifying the start half-edge of the $i$--th white vertex $v_i$ of $\Gamma$ with the $i$--th labeled leaf of $\Gamma'$,
\item starting with $i=1$ and continuing inductively
\begin{enumerate}
\item attaching the remaining half-edges from the vertex $v_i$ to the white vertices of $\Gamma'$ along the loop $\gamma'_i$ following their cyclic ordering (in general we here have several possibilities),
\item replacing the $\gamma'_j$ by the induced ones as describe in Remark \ref{Re:addleaves} (which morally does not do anything),
\end{enumerate}
\item attaching the last $m_2$ labeled leaves of $\Gamma$ to the leaves of $\Gamma'$ labeled $n_2+1, \cdots n_2+m_2$ respecting the order.
\end{enumerate}
\item The gluing defines a map from the boundary segments around the vertex $v_i$ in $\Gamma$ to ordered subsets of the loop $\gamma'_i$ which are sets of boundary segments in $[G]$. All the subsets are disjoint and putting them together following the order of the boundary segments of $v_i$ reproduces the loop $\gamma'_i$. Since we have such a map for each $i$, we get a map from $\{\text{boundary segments in} \ \Gamma \}$ to $\{\text{boundary segments in} \ [G]\}$. We define $g_j$ to be the image of $\gamma_j$ under this map.
\end{enumerate}
The fact that the $g_j$ are again loops follows directly from the construction. 

The orientation (and thus the sign) is obtained by juxtaposition of the orientations of $\Gamma$ and $\Gamma'$ as explained in \cite[Section 2.8]{wahl11}. However, we give a more explicit (but equivalent) way to compute the sign. It is computed by putting all non-start half-edges of $\Gamma$ to the right of the start half-edge of the first white vertex in $\Gamma'$ in their cyclic order and the order of the white vertices in $\Gamma$ (i.e. the first half-edge of the first white vertex of $\Gamma$ is next to the start half-edge of $\Gamma'$) and then computing the parity of the number of half-edges they have to pass to move to their final position. If they are glued left of the start half-edge of a white vertex then they can move to the right of the start half-edge of the next white vertex in $\Gamma'$ which does not change the sign.
\end{Def}
\begin{figure}[ht!]
\vspace{-1cm}  \begin{center}
\subfigure
       [The first building block of the identity]
        {\hspace{1.5cm}
          
            \input{pic_id}\hspace{1cm}
            \label{fig:id1}
            \hspace{0.5cm}
        }%
        \hspace{2.5cm}
        \subfigure
       [The second building block of the identity]
        {\hspace{1.5cm}       
            \input{pic_id2}\hspace{1cm}
            \label{fig:id2}
            \hspace{0.5cm}
        }%
    \end{center}
   \caption[The identity]{
       The element $\id_{\oc{1}{0}} \in \lD_+(\oc{1}{0}, \oc{1}{0})$ and the element $\id \in \lDa(\oc{0}{1}, \oc{0}{1})$}
     
   \label{fig:idgen}
\end{figure}
The identity element $\id_{\oc{n_1}{m_1}}$ in $\lDa(\oc{n_1}{m_1}, \oc{n_1}{m_1})$ is given by the element $\id_{\oc{1}{0}}\amalg \cdots \amalg \id_{\oc{1}{0}} \amalg \id \amalg \cdots \amalg \id$ with $\id_{\oc{1}{0}} \in \lDa(\oc{1}{0}, \oc{1}{0})$ the $\oc{1}{(0,1)}$--looped diagram shown in Figure \ref{fig:id1} and $\id \in \lDa(\oc{0}{1}, \oc{0}{1})$ the $\oc{0}{(2,0)}$--looped diagram (without white vertices and loops) shown in Figure \ref{fig:id2}.

The composition defined above is associative. Examples are given in Figure \ref{fig:comp2} and Figure \ref{fig:comp2.2}.

\begin{Def}\label{def:posbdry}
Let $\lDa_+$ be the dg-category of \emph{looped diagrams with positive boundary condition} with the same objects as $\lDa$ and morphisms $\lDa_+(\oc{n_1}{m_1}, \oc{n_2}{m_2})$ those looped diagrams in $\oc{n_2}{(m_1+m_2, n_1)}$--lD where every connected component contains at least one white vertex or one of the $m_2$ last labeled leaves, i.e. a leaf labeled by a number in $\{n_1+m_1+1, \ldots, n_1+m_1+m_2\}$.
\end{Def}
\begin{figure}[ht!]

 \center
\input{pic_comp2}
	\caption{Composition $\circ$}
		\label{fig:comp2.2}
\end{figure}
\begin{Def}\label{def:plD}
The dg-category $\plDa$ (and $\plDa_+$) of \emph{positively oriented looped diagrams (with positive boundary condition)} has the same objects as $\lDa$ and and morphism $\plDa(\oc{n_1}{m_1}, \oc{n_2}{m_2})$ given by the chain complex $\oc{n_2}{(m_1+m_2, n_1)}$--plD (lying in $\lDa_+(\oc{n_1}{m_1}, \oc{n_2}{m_2})$, respectively).
\end{Def}

%

\begin{Prop}
The composition $\circ$ of looped diagrams defined above is a chain map
\[\lDa(\oc{n_1}{m_1}, \oc{n_2}{m_2}) \otimes \lDa(\oc{n_2}{m_2}, \oc{n_3}{m_3}) \to \lDa(\oc{n_1}{m_1}, \oc{n_3}{m_3}).\]
\end{Prop}
\begin{proof}
Take $x=(\Gamma, \gamma_1, \cdots, \gamma_{n_1}) \in \lDa(\oc{n_1}{m_1}, \oc{n_2}{m_2})$ and $y=(\Gamma', \gamma'_1, \cdots, \gamma'_{n_2}) \in \lDa(\oc{n_2}{m_2}, \oc{n_3}{m_3})$ 
We need to show that $(-1)^{|x|}dy \circ x+y \circ dx = d(y \circ x)$.

We refer to edges at the white vertex of the composition as coming from $\Gamma$ if they where attached in the gluing process and as coming from $\Gamma'$ if they were belonging to $\Gamma'$ before.
 The differential in $y \circ x$ comes from four different kinds of boundaries, which is illustrated on the example in Figure \ref{fig:multidiff} where we compute the differential of the second to last summand of the composition shown in Figure \ref{fig:comp2.2}.
\begin{figure}[ht!]
     \begin{center}
     \input{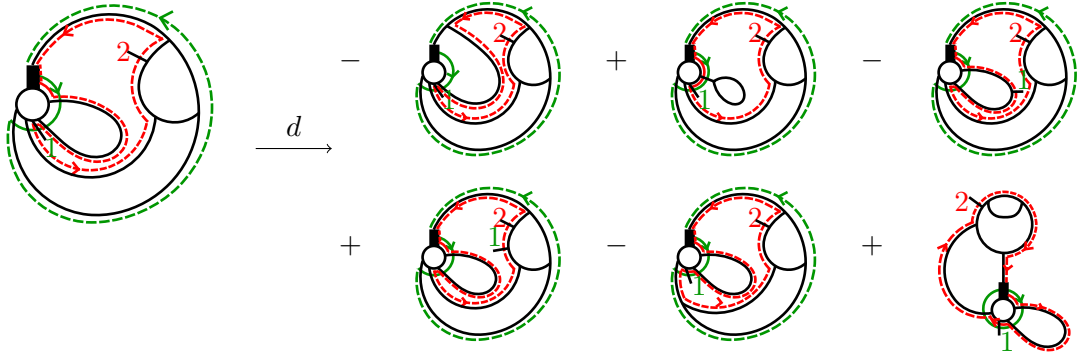}
     \end{center}
   \caption{
       The differential of the second to last term of the composition in Figure \ref{fig:comp2.2}
     }
   \label{fig:multidiff}
\end{figure}
The four kinds of boundaries are described as follows:
\begin{enumerate}
\item The boundaries coming from the multiplication of edges belonging to $\Gamma'$ together (cf. first, second and last summand in Figure \ref{fig:multidiff}).
\item The boundaries coming from the multiplication of edges originally belonging to the same white vertex $v_j$ in $\Gamma$ or from two special kinds of boundaries:
Those arising from contracting the first segment of a loop $g_i$ if this segment starts at the arc component which the old start half-edge of $v_j$ was attached to and ends at the old second half-edge of $v_j$.
Similarly, we additionally take the terms obtained from contracting the  last segment of a loop $g_i$ if this segment starts at the old last half-edge of $v_j$ and ends at the arc component which the old start half-edge of $v_j$ was attached to (cf. second last summand in Figure \ref{fig:multidiff}).
\item The boundaries arising from the multiplication of edges of $\Gamma$ and $\Gamma'$ together (except for the two cases mentioned in the step before) (cf. third summand in Figure \ref{fig:multidiff}).
\item The boundaries obtained from the multiplication of edges of $\Gamma$ coming from two different white vertices (cf. fourth summand in Figure \ref{fig:multidiff}).
\end{enumerate}

We now show that $dy \circ x$ gives exactly all summands appearing in 1., $y \circ dx$ all those appearing in 2. and that the sums of the diagrams in 3. and in 4. are zero (i.e. that every diagram shows up twice with opposite sign).

\begin{enumerate}
\item In $dy$ we have two kinds of boundaries, the ones coming from contracting boundary segments which are not contained in any loop and those contracting segments appearing in loops. 

In the first case the two neighbored edges will also be neighbored in $y \circ x$ and multiplying them first and then composing with $x$ or first composing and then multiplying is the same.

In the second case we again have to distinguish two cases. First, if we consider contracting a segment $s^k_j$ which appears in a loop but none of the loops only consists of this segment (i.e. $\gamma'_i \neq \{ s^k_j \})$. Then in the composition there are summands where we did not glue any edges into this segment. Contracting the segment in these summands of $y \circ x$ agrees with the composition of $x$ with those boundaries of $y$ where we contracted that segment in $\Gamma'$.
Second, if we have a loop $\gamma'_i=\{ s^k_j \}$ but $s^k_j$ is not the whole boundary segment (in which case there is no contraction on either of the sides) then the contraction of this segment is a constant loop, so it only gives a term in the composition if there is only one edge attached to $v_i$ in $\Gamma$ and the composition then makes the part of the loop going around $v_i$ (if there was one) constant. But in $y \circ x$ the according segment can only be empty in exactly this case and thus the terms agree.

So up to sign we have shown that $dy \circ x$ agrees with the terms described in 1.

For the sign we divide the edges glued onto $\Gamma'$ into two pairs, those left of the pair of edges of $\Gamma'$ we multiply to get the considered element and those right (if we have several white vertices in $\Gamma'$ and the edges we multiply are attached to $v_i$, edges attached to a white vertex $v_j$ with $j<i$ count as being left). Denote the numbers by $e_{left}$ and $e_{right}$. We have $|x|=e_{left}+e_{right}$. The sign of the boundary that multiplies the two vertices in $d(y \circ x)$ changed by $(-1)^{e_{left}}$ against the sign of the boundary of multiplying these edges in $y$. On the other hand, if we first multiply the  two edges in $\Gamma'$ and then glue $\Gamma$ on, all the edges right of this boundary have to move over one edge less to get to their position, so the sign of the composition changes by $(-1)^{e_{right}}$. Thus the total difference in sign is $(-1)^{e_{left}+e_{right}}=(-1)^{|x|}$, so
the terms of $(-1)^{|x|} dy \circ x$ show up with the same sign as those in $d(y \circ x)$.

\item It is not hard to see that first multiplying neighbored edges around a vertex in $x$ (and contracting the loop accordingly) and then gluing the result onto $y$ or taking those summands of $d(y \circ x)$  were we multiplied neighbored vertices coming from $x$ (and contracted the piece of the loop) agrees. The special cases described in $2$ come from multiplying the start edge with the second or last edge and then gluing it onto $\Gamma'$.
Similar considerations as in the first case show that the two terms show up with the same sign.

\item Write $\gamma'_i=(s_1, \ldots, s_l)$ with the $s_j$ boundary segments in $\Gamma'$. If in $x \circ y$ an edge $e$ of $\Gamma$ was glued as the last edge in a segment $s_j$ for $j<l$, then there is another summand in $x \circ y$ where all edges of $\Gamma$ different to $e$ was glued to the same segment as before, but $e$ was glued as the first edge to $s_{j+1}$. The element obtained from the first element by multiplying the edge $e$ onto the arc component following $s_j$ agrees with the one obtained from the second by multiplying $e$ onto the arc component before $s_{j+1}$ which is the same as the arc component following $s_j$ (by the definition of a loop). An example is given by the red edge in the first and fifth summand in the composition in Figure \ref{fig:comp2.2}. The boundary multiplying it onto the little loop (in the first case from the left in the second from the right) is the same.

Hence, all terms in the differential where we multiplied an edge of $\Gamma$ onto one of $\Gamma'$ show up twice by the above argumentation.

Assume that the first of these two elements form a term in $y \circ x$ with sign $\omega$ and its differential had sign $\sigma$, so in $d(y \circ x)$ the element has sign $\omega \cdot \sigma$. To get the second the specified edge of $\Gamma$ is moved by $r$ steps (to the right of the boundary component). This means that it shows up with sign $(-1)^r \omega$ in $y \circ x$. Moving over the neighbored edge does not change the sign of the differential and moving over more edges changes it by $-1$ each time. Thus the boundary shows up with sign $(-1)^{r-1} \sigma$, so in $d(y \circ x)$ the element has sign $(-1)^r \omega \cdot (-1)^{r-1} \sigma=- \omega \sigma$, i.e. both terms have opposite signs and cancel.
\item We are left to show that terms where we multiplied two edges of $\Gamma$ belonging to two different white vertices together show up twice, too. However, since we glued the edges inductively, the argumentation from the previous step works if we view the edges of the first vertex as fixed and glue the edges of the second vertex onto that graph.

\end{enumerate}
\end{proof}
It is not hard to check that $\lDa_+$, $\plDa$ and $\plDa_+$ are subcategories.

The way we defined our category $\lDa$ we obtain a functor $K:\SD \to \lDa$ which takes a Sullivan diagram $\Gamma \in \SD(\oc{n_1}{m_1},\oc{n_2}{m_2})$  and sends it to the looped diagram $([\Gamma], \gamma_1, \cdots, \gamma_{n_1})$ with $\gamma_i$ the loop starting from the $i$--th labeled leaf and following the boundary cycle the leaf was in before (see Figure \ref{fig:funcK} for an example).
\begin{figure}[ht!]
     \begin{center}
     \input{pic_SDCom}
     \end{center}
   \caption[The functor $K$]{The functor $K$ applied to a diagram in $\SD(\oc{1}{0}, \oc{1}{0})$}
   \label{fig:funcK}
\end{figure}

Before moving on we want to imitate one more construction done to Sullivan diagrams: To make them fit for string topology, in \cite[Section 6.3+6.5]{wahl11} a shifted version was considered. A Sullivan diagram $S$ was shifted by $-d \cdot \chi(S, \partial_{out})$, where $\chi(S, \partial_{out})$ is the Euler characteristic of a representative of $S$ as a CW-complex relative to its outgoing boundary (the $n_2$ white vertices and the $m_2$ labeled outgoing leaves). 
Similarly, we define:
\begin{Def}\label{def:shift}
Let $\lDa_{d}$ to be the shifted version of $\lDa$ where a looped diagram $(\Gamma, \gamma_1, \cdots, \gamma_{n_1})$ gets shifted by $-d \cdot \chi(\Gamma, \partial_{out})$.
\end{Def}
In particular, the functor $K$ also gives a functor $K:\SD_d \to \lDa_d$. For more details on this construction we refer to \cite[Section 6.5]{wahl11}.

\subsection{The split subcomplex of non-constant diagrams}\label{sec:nonconst}

For later purpose we want to split off those diagrams which have constant loops. They clearly form a subcomplex and as we will see below this subcomplex is split. In some situations it will be more natural to work with the non-constant diagrams only.

\begin{Def}\label{def:cst}
A $\oc{p}{(m,n)}$--looped diagram $(\Gamma, \gamma_1, \cdots, \gamma_n)$ is called \emph{partly constant} if one of the $\gamma_j$ is a constant loop. The subcomplex of $\lDa(\oc{n_1}{m_1}, \oc{n_2}{m_2})$ spanned by these diagrams is denoted by $\lDc(\oc{n_1}{m_1}, \oc{n_2}{m_2})$.
\end{Def}
We define the map $p_j:\lDa(\oc{n_1}{m_1}, \oc{n_2}{m_2}) \to \lDc(\oc{n_1}{m_1}, \oc{n_2}{m_2}) \subseteq \lDa(\oc{n_1}{m_1}, \oc{n_2}{m_2})$ by
\[p_j((\Gamma, \gamma_1, \ldots, \gamma_{n_1}))=(\Gamma, \gamma_1, \ldots, \gamma_{j-1}, cst, \gamma_{j+1}, \ldots, \gamma_{n_1})\]
where $cst$ is the constant loop starting at the leaf $j$.

\begin{Lemma}
The map $p_j$ is a chain map.
\end{Lemma}
\begin{proof}
Contracting a boundary segment in $\Gamma$ which was part of the loop $\gamma_j$ and then forgetting the rest of the loop commutes with first forgetting the whole loop and then contracting the boundary segment, thus $p_j$ commutes with the differential.
\end{proof}

For a set $T=\{t_1, \cdots, t_k\} \subseteq \{1, \cdots, n_1\}$ we define $p_T=p_{t_k} \circ \cdots \circ p_{t_1}$, i.e. the map making the loops corresponding to $T$ constant.

Moreover, we define $p_{cst}:\lDa(\oc{n_1}{m_1}, \oc{n_2}{m_2}) \to \lDc(\oc{n_1}{m_1}, \oc{n_2}{m_2}) $ by
\[p_{cst}=\sum_{k=1}^{n_1} \sum_{\substack{T \subseteq  \{1, \cdots, n_1\}\\ |T|=k}}(-1)^{k+1} p_T\]
\begin{Prop}
The map $p_{cst}$ is a splitting of the inclusion of the subcomplex $i:\lDc(\oc{n_1}{m_1}, \oc{n_2}{m_2}) \hookrightarrow \lDa(\oc{n_1}{m_1}, \oc{n_2}{m_2})$, i.e.
\[p_{cst} \circ i:\lDc(\oc{n_1}{m_1}, \oc{n_2}{m_2})  \to \lDc(\oc{n_1}{m_1}, \oc{n_2}{m_2}) \]
is the identity.
\end{Prop}
\begin{proof}
Let $x=(\Gamma, \gamma_1, \ldots, \gamma_{n_1}) \in \lDc(\oc{n_1}{m_1}, \oc{n_2}{m_2})$. For simplicity of notation we assume that $\gamma_1$ is constant. For a set $T$ with $1 \notin T$ we have that $p_{T}=p_{\{1\} \cup T}$. Since $|\{1\} \cup T|=|T|+1$ these terms show up with opposite signs and thus cancel. The family of all non-empty subsets $T$ with $1 \notin T$ together with $\{1\} \cup T$ are all non-empty subsets of $\{1, \cdots, n_1\}$ except for the set $\{1\}$. Thus the only non-trivial term in $p_{cst}(x)$ is $p_{\{1\}}(x)$ which is $x$ since the first loop was already constant.
\end{proof}

\begin{Def}\label{def:non-cst}
For a looped diagram $(\Gamma, \gamma_1, \ldots, \gamma_{n_1}) \in \lDa(\oc{n_1}{m_1}, \oc{n_2}{m_2})$ we define
\[ (\Gamma, \langle \gamma_1\rangle, \ldots, \langle\gamma_{n_1}\rangle)=(\Gamma, \gamma_1, \ldots, \gamma_{n_1}) - p_{cst}((\Gamma, \gamma_1, \ldots, \gamma_{n_1})).\]
The complex spanned by these diagrams is denoted by $\lD(\oc{n_1}{m_1}, \oc{n_2}{m_2})$.
\end{Def}
Since $p_\emptyset(x)=\id$ we can rewrite
\[ (\Gamma, \langle \gamma_1\rangle, \ldots, \langle\gamma_{n_1}\rangle)=\sum_{T \subseteq \{1, \cdots, n_1\} }(-1)^{|T|}p_T(\Gamma, \gamma_1, \ldots, \gamma_{n_1}) .\]
\begin{figure}[ht!]
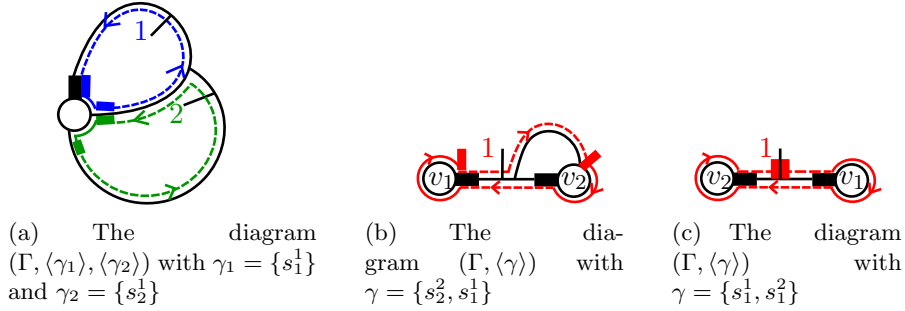

\vspace{-2cm}
     \begin{center}
     \subfigure[c]
        [The diagram $(\Gamma, \langle \gamma_1 \rangle, \langle \gamma_2 \rangle)$ with $\gamma_1=\{s_1^1\}$ and $\gamma_2=\{s_2^1\}$]
        {
            \label{fig:nc1}
            \input{pic_nc11}
        }%
        \hspace{0.5cm}
             \subfigure
        [The diagram $(\Gamma, \langle \gamma \rangle)$ with $\gamma= \{ s_2^2, s_1^1\}$]
        {%
            \label{fig:nc3}
            \input{pic_nc3}
        }%
         \hspace{0.5cm}
                \subfigure[The diagram $(\Gamma, \langle \gamma \rangle)$ with $\gamma= \{ s_1^1, s_1^2\}$]
        {%
            \label{fig:nc4}
            \input{pic_nc4}
        }%
    \end{center}
   \caption[Three non-constant diagrams in the new notation]{
       Three non-constant diagrams in the new notation
     }
   \label{fig:non-cst}
\end{figure}

In pictures we mark the loops $\langle\gamma\rangle$ by a bar at the start of the first boundary segment contained in the loop and a bar at the end of the last boundary segment (so it is the picture of $(\Gamma, \gamma_1, \ldots, \gamma_{n_1})$ with extra bars in there). Examples are given in Figure \ref{fig:non-cst}.

\begin{Cor}
We have a splitting
\[ \lDa(\oc{n_1}{m_1}, \oc{n_2}{m_2}) \cong \lDc(\oc{n_1}{m_1}, \oc{n_2}{m_2}) \oplus \lD(\oc{n_1}{m_1}, \oc{n_2}{m_2}).\]
\end{Cor}

The part of the differential on $\lD(\oc{n_1}{m_1}, \oc{n_2}{m_2})$ coming from a boundary map contracting a loop to a constant loop is trivial. 

The above splitting induces an isomorphism of chain complexes $\lD(\oc{n_1}{m_1}, \oc{n_2}{m_2}) \cong \lDa(\oc{n_1}{m_1}, \oc{n_2}{m_2}) / \lDc(\oc{n_1}{m_1}, \oc{n_2}{m_2})$ under which the class of $(\Gamma, \langle \gamma_1\rangle, \ldots, \langle\gamma_{n_1}\rangle)$ is equivalent to the class of $(\Gamma, \gamma_1, \ldots, \gamma_{n_1})$. One might want to use the second one to compute the composition of two elements in $\lD(\oc{n_1}{m_1}, \oc{n_2}{m_2})$. Unfortunately, the partly constant terms which get subtracted in the definition of $(\Gamma, \langle \gamma_1\rangle, \ldots, \langle\gamma_{n_1}\rangle)$ can contribute non-trivially to the composition. In the next part of the section we provide conditions under which this phenomenon cannot occur.

To do so, we introduce a bit of notation:

We call a white vertex $v_i$ in a  commutative Sullivan diagram $\Gamma$ \emph{singular} if it is of degree zero, i.e. there is only one boundary segment.

For a looped diagram $x=(\Gamma, \gamma_1, \cdots, \gamma_{n_1})$ and a singular vertex $v_i$ we write $x \backslash s^i_1:=(\Gamma, \gamma_1\backslash s^i_1, \cdots, \gamma_{n_1}\backslash s^i_1)$, where $\gamma_j \backslash s^i_j$ is the loop without the boundary segment $s^i_j$. An example is given in Figure \ref{fig:sing}. 
\begin{figure}[ht!]
     \begin{center}
     \input{pic_singremove}
     \end{center}
   \caption{
       $v_1$ singular and $v_1\backslash \{s_1^1\}$ 
     }
   \label{fig:sing}
\end{figure}

For $x=(\Gamma, \gamma_1, \ldots, \gamma_{n_1}) \in  \lDa(\oc{n_1}{m_1}, \oc{n_2}{m_2})$ denote the set of singular vertices of $\Gamma$ by $S_x \subseteq \{ 1, \cdots, n_2\}$ and define $s_T(x)=(x \backslash T)$ for $T \subseteq S_x$. 
%
One checks that for $T \subseteq S_x$ we have $p_{cst}(s_T(x))=s(p_{cst}(x_T))$ and thus $s_T (x-p_{cst}(x))=(\id-p_{cst})(s_T(x))$ hence $s_T:\lD(\oc{n_1}{m_1}, \oc{n_2}{m_2}) \to \lD(\oc{n_1}{m_1}, \oc{n_2}{m_2})$.

For a looped diagram $x=(\Gamma, \langle\gamma_1\rangle, \ldots, \langle\gamma_{n}\rangle) \in \lD(\oc{n_1}{m_1}, \oc{n_2}{m_2})$ we call a subset $T$ of the singular vertices $S_x$ \emph{loop-covering} if there is at least one loop $\gamma_i$ that only consists of boundary segments belonging to the white vertices in $T$ (equivalently, $\gamma \backslash T=cst$). A singular white vertex $v_i$ is called loop-covering, if $T=\{v_i\}$ is-loop covering. Note that if all vertices are loop-covering, also all subsets $T \subseteq S_x$ are loop-covering. The white vertex in Figure \ref{fig:nc1} is not singular. In Figure \ref{fig:nc3} the vertex $v_1$ is singular, but not loop-covering. In Figure \ref{fig:nc4} both vertices $v_1$ and $v_2$ are singular, but none of them is loop-covering. However, the set $T=\{v_1, v_2\}$ is loop-covering.

We define \[\overline{s}(x)=\sum_{\substack{T \subseteq {S}_x\\ T \text{ not loop-covering}}} (-1)^{|T|} s_T(x).\] 

Note that if for $x \in \lD(\oc{n_1}{m_1}, \oc{n_2}{m_2})$ all singular vertices are loop-covering, then
$\overline{s}(x)=x$.

For elements $x=(\Gamma, \langle\gamma_1\rangle, \ldots, \langle\gamma_{n_1}\rangle) \in \lD(\oc{n_1}{m_1}, \oc{n_2}{m_2})$ and $y=(\Gamma', \langle\gamma'_1\rangle, \cdots, \langle\gamma'_{n_2}\rangle) \in \lD(\oc{n_2}{m_2}, \oc{n_3}{m_3})$
 we denote their lifts by $\widehat{x}=(\Gamma, \gamma_1, \cdots, \gamma_{n_1}) \in \lDa(\oc{n_1}{m_1}, \oc{n_2}{m_2})$ and $\widehat{y}=(\Gamma', \gamma'_1, \cdots, \gamma'_{n_2}) \in \lDa(\oc{n_2}{m_2}, \oc{n_3}{m_3})$. Suppose their composition in $\lDa$ is given by
$\widehat{y} \circ \widehat{x}= \sum ([G], g_1, \cdots, g_{n_1})$. Then we define 
\[y \; \tilde{\circ}\;  x :=\widehat{y} \circ \widehat{x}-p_{cst}(\widehat{y} \circ \widehat{x})= \sum ([G], \langle g_1\rangle, \cdots, \langle g_{n_1}\rangle).\]
This is not a chain map, but it is not far from being one as in most cases it agrees with the actual composition $y \circ x$. More precisely, we get:
\begin{Prop}
For elements $x=(\Gamma, \langle\gamma_1\rangle, \ldots, \langle\gamma_{n_1}\rangle) \in \lD(\oc{n_1}{m_1}, \oc{n_2}{m_2})$ and $y=(\Gamma', \langle\gamma'_1\rangle, \cdots, \langle\gamma'_{n_2}\rangle) \in \lD(\oc{n_2}{m_2}, \oc{n_3}{m_3})$ their composition in terms of the above notation is given by
\[ y \circ x= 
y \; \tilde{\circ} \; \overline{s}(x).\]
In particular, $y \circ x \in \lD(\oc{n_1}{m_1}, \oc{n_3}{m_3})$.
\end{Prop}
\begin{proof}
Using the definition of $x$ and $y$, we need to show that
\[ (\widehat{y}-p_{cst}(\widehat{y})) \circ (\widehat{x} -p_{cst}(\widehat{x}))=\widehat{y} \circ \widehat{\overline{s}(x)}-p_{cst}(\widehat{y} \circ \widehat{\overline{s}(x)}) \]

Let $T \subseteq S_x$ be a subset that is not loop-covering. We claim that $\widehat{s_T(x)}=s_T(\widehat{x})$. To see so, recall that $\widehat{x}$ takes a representation of $x$ as a linear combination of looped diagrams and throws away the partly constant ones. Moreover, $x = \widehat{x} - p_{cst}(\widehat{x})$. We need to see that the partly constant terms of $x \backslash T= \widehat{x} \backslash T - p_{cst}(\widehat{x}) \backslash T$ are exactly given by $p_{cst}(\widehat{x}) \backslash T$. It is clear that all terms in $p_{cst}(\widehat{x}) \backslash T$ are still partly constant. Moreover, by the assumption that no loop of $x$ (and thus no loop of $\widehat{x}$) is completely covered by $T$, the diagram $\widehat{x} \backslash T$ cannot be partly constant, which shows the claim.

Now, the above formula is equivalent to showing that
\[\sum_{T \subseteq \{1, \cdots, n_2\}} \sum_{U \subseteq \{1, \cdots, n_1\}}  p_T(\widehat{y}) \circ p_U(\widehat{x}) =\sum_{\substack{T \subseteq {S}_x\\ T \text{ not loop-covering}}} \sum_{U \subseteq \{1, \cdots, n_1\}}  p_{U}(\widehat{y} \circ s_{T}(\widehat{x})).\]
The proposition follows via the following steps which hold for general elements $a=(\Lambda, \lambda_1, \cdots, \lambda_{n_1}) \in \lDa(\oc{n_1}{m_1}, \oc{n_2}{m_2})$ and $b=(\Lambda', \lambda'_1, \cdots, \lambda'_{n_2}) \in \lDa(\oc{n_2}{m_2}, \oc{n_3}{m_3})$:
\begin{enumerate}
\item For $U \subseteq \{1, \cdots, n_1\}$ we have $p_{U}(a \circ b)=a \circ p_{U}(b)$. 
\item The singular vertices of $a$ and $p_T(a)$ agree.
For any sets $T \subseteq S_a$ and $U \subseteq \{1, \cdots, n_1\}$ the equality $s_T(p_U(a))=p_U(s_T(a))$ holds, since removing first part of a loop via a singular vertex and then the whole loop commutes with first removing the whole loop and then everything else at the white vertex.
\item We have $p_T(b) \circ a= 0$ for $T \nsubseteq S_{a}$, since then we have a vertex with more than one edge glued to a constant loop.
\item We have $b \circ s_T(a)=p_T(b) \circ a$ for $T \subseteq S_a$ since in $p_T(b) \circ a$ all the loops around the singular vertices in $T$ become constant (because we removed them in $b$) and this is the same as first removing them and then gluing them onto $b$.
\item For $T \subseteq S_a$ such that $T$ is loop-covering, we have $\sum_{U \subseteq \{1, \cdots,  n_1\} }(-1)^{|U|}s_T(p_U(a))=0$. To see so, assume that $T$ covers the loop $\gamma_j$ of $a$, i.e. $\gamma_j$ only consists of boundary segments of white vertices belonging to $T$. For $U \subseteq \{1, \cdots, n_1\}$ with $j \notin U$, we get $s_T(p_U(x))=s_T(p_{U \cup \{j\}}(x))$, which implies the above claim.
\end{enumerate}
Plugging this in, we obtain 
\begin{align*}
\sum_{T \subseteq \{1, \cdots, n_2\}} \sum_{U \subseteq \{1, \cdots, n_1\}}  p_T(\widehat{y}) \circ p_U(\widehat{x}) &\stackrel{(3)}{=} \sum_{T \subseteq S_{x}} \sum_{U \subseteq \{1, \cdots, n_1\}}  p_T(\widehat{y}) \circ p_U(\widehat{x})\\
&\stackrel{(4)}{=} \sum_{T \subseteq S_{x}} \sum_{U \subseteq \{1, \cdots, n_1\}}  \widehat{y} \circ s_T(p_U(\widehat{x}))\\
&\stackrel{(5)}{=} \sum_{\substack{T \subseteq {S}_x\\ T \text{ not loop-covering}}} \sum_{U \subseteq \{1, \cdots, n_1\}}  \widehat{y} \circ s_T(p_U(\widehat{x}))\\
&\stackrel{(2)}{=} \sum_{\substack{T \subseteq {S}_x\\ T \text{ not loop-covering}}} \sum_{U \subseteq \{1, \cdots, n_1\}}  \widehat{y} \circ p_U(s_T(\widehat{x}))\\
&\stackrel{(1)}{=} \sum_{\substack{T \subseteq {S}_x\\ T \text{ not loop-covering}}} \sum_{U \subseteq \{1, \cdots, n_1\}}  p_U(\widehat{y} \circ s_T(\widehat{x}))
\end{align*}
and thus the proposition is proven.

\end{proof}
\begin{Cor}\label{cor:loopcov}

Let $x\in \lD(\oc{n_1}{m_1}, \oc{n_2}{m_2})$ and $y \in \lD(\oc{n_2}{m_2}, \oc{n_3}{m_3})$ and assume that all singular vertices in $x$ are loop-covering. 
Then their composition is given by
\[ y \circ x= y \; \tilde{\circ} \; x.\]
This holds in particular if $x$ has no singular vertices.
\end{Cor}

\subsection{The type of a diagram and products of diagrams}\label{sec:type}
In this section we provide a finer decomposition of $\lDa(\oc{n_1}{m_1}, \oc{n_2}{m_2})$ on the level of vector spaces. This part is particularly technical and we invite the reader to skip it and come back later if needed.

As already described above, we can concatenate subloops in a commutative Sullivan diagram which start at the same labeled leaf. We denote the concatenation of two loops $\gamma$ and $\gamma'$ by $\gamma \ast \gamma'$. For a commutative Sullivan diagram $\Gamma$ with $t_j$ loops $\gamma^i_j$ for $1 \leq i \leq t_j$ starting at the $j$-th labeled leaf for $1 \leq j \leq n$, we define the element $(\Gamma, \langle \gamma^1_1, \ldots, \gamma^{t_1}_1\rangle, \ldots, \langle\gamma^1_n, \ldots, \gamma^{t_n}_n\rangle) \in \lDa(\oc{n_1}{m_1}, \oc{n_2}{m_2})$ as
\begin{align*}
(\Gamma, \langle \gamma^1_1, \ldots, \gamma^{t_1}_1\rangle, \ldots, \langle\gamma^1_n, \ldots, \gamma^{t_n}_n\rangle) := \hspace{-0.25cm}\sum_{U_1 \subseteq \{1, \dots, t_1\}}\hspace{-0.15cm} \cdots\hspace{-0.15cm} \sum_{U_{n} \subseteq \{1, \ldots, t_{n}\}} \hspace{-0.5cm} (-1)^{\sum_j (t_j-|U_j|)} (\Gamma, \gamma_1^{\ast U_1} , \ldots,  \gamma_n^{\ast U_n})\end{align*}
where we define $\gamma_j^{\ast U_j}= \gamma_j^{u_j^1} \ast \cdots \ast  \gamma_j^{u_j^{|U_j|}}$ if $U_j=\{u_j^1, \cdots, u_j^{|U_j|}\}$ is non-empty and $\gamma_j^{\ast \emptyset}=cst$ the constant loop if $U_j$ is empty.

Note that we can assume that all the $\gamma_j$'s are not constant since otherwise the element $(\Gamma, \langle \gamma^1_1, \ldots, \gamma^{t_1}_1\rangle, \ldots, \langle\gamma^1_n, \ldots, \gamma^{t_n}_n\rangle)$ is zero. Moreover if all $t_j$ were $1$ we recover the definition of the non-constant diagrams $\lD$ used in the previous section. Furthermore, if at least one of the $t_i$ is zero, then $(\Gamma, \langle \gamma^1_1, \ldots, \gamma^{t_1}_1\rangle, \ldots, \langle\gamma^1_n, \ldots, \gamma^{t_n}_n\rangle)$ is a partially constant diagram.

In the pictures we draw the diagram $(\Gamma, \langle \gamma^1_1, \ldots, \gamma^{t_1}_1\rangle, \ldots, \langle\gamma^1_n, \ldots, \gamma^{t_n}_n\rangle)$ like the diagram $(\Gamma, \gamma_1^1 \ast \cdots \ast \gamma_1^{t_1}, \cdots, \gamma_n^1 \ast \cdots \ast \gamma_n^{t_n})$ but adding bars at the end of each of the loops $\gamma_i^j$ (for an example see Figure \ref{fig:irredtog}). 

\begin{figure}[ht!]
     \begin{center}
     \subfigure[c]
        [The diagram $(\Gamma, \langle \gamma^1, \gamma^2\rangle)$ with $\gamma^1=\{s_1^1, s_2^1\}$ and $\gamma^2=\{s_3^1\}$]
        {\hspace{0.7cm}
            \label{fig:irred}
            \input{pic_irred} \hspace{0.2cm}
        }%
        \hspace{1.4cm}
             \subfigure
        [The diagram $(\Gamma, \langle \lambda^1, \lambda^2, \lambda^3\rangle)$ with $\lambda^1=\{s_1^1\}$, $\lambda^2=\{s_2^1\}$ and $\lambda^3=\{s_3^1\}$]
        {%
        \hspace{0.7cm}
            \label{fig:irred2}
            \input{pic_irred2}
            \hspace{0.7cm}
        }%
    \end{center}
   \caption[ Two looped diagrams depicted in the new notation]{
       Two elements of $lD(\oc{1}{0},\oc{1}{0})$ depicted in the new notation
     }
   \label{fig:irredtog}
\end{figure}

Now we restrict to those diagrams where all the subloops $\gamma_i^j$ are irreducible, i.e. cannot be written as concatenation of non-constant loops.
\begin{Def}\label{def:type}
An \emph{irreducible $\oc{p}{(m,n)}$-looped diagram} is given by an element of the form   $(\Gamma, \langle\gamma^1_1, \ldots, \gamma^{t_1}_1\rangle, \ldots, \langle\gamma^1_n, \ldots, \gamma^{t_n}_n\rangle)$ such that all the loops $\gamma_i^j$ are irreducible. 
The \emph{type} of an irreducible looped diagram $(\Gamma, \langle\gamma^1_1, \ldots, \gamma^{t_1}_1\rangle, \ldots, \langle\gamma^1_n, \ldots, \gamma^{t_n}_n\rangle)$ is defined to be the tuple $(t_1, \cdots, t_n)$.

The space spanned by irreducible looped diagrams of type $(t_1, \cdots, t_n)$ is denoted by $\lDa^{t_1, \ldots, t_{n_1}}(\oc{n_1}{m_1}, \oc{n_2}{m_2})$.
\end{Def}
The type of a looped diagram is not preserved by the differential, hence the spaces $ \lDa^{t_1, \ldots, t_{n_1}}(\oc{n_1}{m_1}, \oc{n_2}{m_2})$ are in general not chain complexes.

One checks that every $\oc{p}{(m,n)}$-looped diagram can be written as the linear combination of irreducible $\oc{p}{(m,n)}$-looped diagrams and vice versa. Hence we can rewrite the complex of looped diagrams as
\[\lDa(\oc{n_1}{m_1}, \oc{n_2}{m_2}) \cong \bigoplus_{t_1, \cdots, t_{n_1}}  \lDa^{t_1, \ldots, t_{n_1}}(\oc{n_1}{m_1}, \oc{n_2}{m_2})\]
with
\[\lDc(\oc{n_1}{m_1}, \oc{n_2}{m_2}) \cong \bigoplus_{\substack{t_1, \cdots, t_{n_1}\\\exists \ j \text{ s.t. } t_j=0}}  \lDa^{t_1, \ldots, t_{n_1}}(\oc{n_1}{m_1}, \oc{n_2}{m_2})\]
and
\[\lD(\oc{n_1}{m_1}, \oc{n_2}{m_2}) \cong \bigoplus_{\substack{t_1, \cdots, t_{n_1}\\t_j >0  \text{ for all } j}}  \lDa^{t_1, \ldots, t_{n_1}}(\oc{n_1}{m_1}, \oc{n_2}{m_2}).\]

Similarly to working out the composition for positive diagrams explicitly, we want to say a few words about the composition of irreducible looped diagrams. 
For $x$ and $y$ two irreducible looped diagrams, in $y \circ x$ all old loops of $y$ are not allowed to be empty. This means that we either have to glue an edge in there or they afterward have to be covered by a loop again. Moreover, if they were covered but no edge was glued, the diagram where this subloop is omitted has to be subtracted. For an example see Figure \ref{fig:compirred1}.
Instead of taking the direct sum over all types of irreducible looped diagrams (which as explained just is the complex of looped diagrams) we want to take the product. Unfortunately, in the product complex over all types composition is not always well-defined. Nevertheless, we will use this complex and later on deal with composition.

\begin{figure}[ht!]
     \begin{center}
%
%

            \input{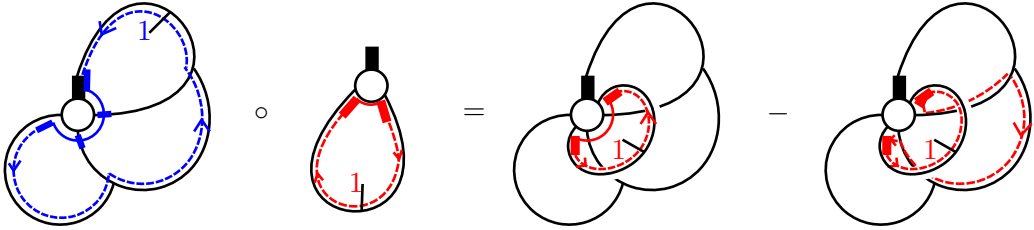}
    \end{center}
   \caption{
       composition of irreducible looped diagrams
     }
 \label{fig:compirred1}
\end{figure}

\begin{Def}\label{def:ilD}
For $n_1>0$ the chain complex of \emph{products of irreducible looped diagrams} $\ilDa(\oc{n_1}{m_1},\oc{n_2}{m_2})$ is defined as
\[ \ilDa(\oc{n_1}{m_1},\oc{n_2}{m_2})=\prod_{t_1, \cdots, t_{n_1}}  \lDa^{t_1, \ldots, t_{n_1}}(\oc{n_1}{m_1}, \oc{n_2}{m_2})\]
and
\[ \ilD(\oc{n_1}{m_1},\oc{n_2}{m_2})=\prod_{t_i>0}  \lDa^{t_1, \ldots, t_{n_1}}(\oc{n_1}{m_1}, \oc{n_2}{m_2})\]

Similarly, we define $ \ilDa_+(\oc{n_1}{m_1},\oc{n_2}{m_2})$, $ \iplDa(\oc{n_1}{m_1},\oc{n_2}{m_2})$ and $ \iplDa_+(\oc{n_1}{m_1},\oc{n_2}{m_2})$ as the products over all types restricted to these subcomplexes.
%
\end{Def}
We see that
\[\ilDa(\oc{n_1}{m_1},\oc{n_2}{m_2}) \cong \lDc(\oc{n_1}{m_1},\oc{n_2}{m_2}) \oplus \ilD(\oc{n_1}{m_1},\oc{n_2}{m_2}).\]
If $n_1=0$, we obtain
\[ \ilDa(\oc{0}{m_1},\oc{n_2}{m_2})=\iplDa(\oc{0}{m_1},\oc{n_2}{m_2})=\lDa(\oc{0}{m_1},\oc{n_2}{m_2}).\]

In order to see that the differential on $\ilDa$ and the other product complexes is well-defined, we need to check that for a fixed type $(t_1, \ldots, t_n)$ there can only be finitely many types $(t'_1, \ldots, t'_n)$ such that the differential has non-trivial elements of type $(t_1, \ldots, t_n)$.
However, one checks similarly to the computations done for non-constant diagrams that the differential is zero if an irreducible loop gets contracted and hence in the resulting summands of the differential there are at least as many irreducible loops as before. Therefore, the differential of a looped diagram of type $(t'_1, \ldots, t'_n)$ has summands of type $(t_1, \ldots, t_n)$ only if $t'_i \leq t_i$. Thus, on the product over all types, the differential is still welldefined.

\begin{Def}
For two elements $a=\sum_{t_1=1}^\infty \cdots \sum_{t_{n_1}=1}^\infty a_{t_1, \ldots, t_{n_1}} \in \ilDa(\oc{n_1}{m_1},\oc{n_2}{m_2})$ and $b=\sum_{t_1=1}^\infty \cdots \sum_{t_{n_2}=1}^\infty b_{t_1, \ldots, t_{n_2}} \in \ilDa(\oc{n_2}{m_2},\oc{n_3}{m_3})$ the pair $(a,b)$ is called \emph{composable} if
\[\sum_{t_1=1}^\infty \cdots \sum_{t_{n_1}=1}^\infty \sum_{t'_1=1}^\infty \cdots \sum_{t'_{n_2}=1}^\infty  b_{t'_1, \ldots, t'_{n_2}} \circ a_{t_1, \ldots, t_{n_1}} \]
only contains finitely many summands of type $(u_1, \ldots, u_{n_1})$ for arbitrary $u_i \in \N$.
\end{Def}

\begin{Def}\label{def:start}
Let $\plDa_{start}$ consist of those graphs in $\plDa$, where all loops consist of exactly one boundary segment of a white vertex which is the first boundary segment of that white vertex.
\end{Def}
\begin{Prop}\label{prop:composable}
Let $a \in \ilDa(\oc{n_1}{m_1},\oc{n_2}{m_2})$ and $b \in \ilDa(\oc{n_2}{m_2},\oc{n_3}{m_3})$. If one of the following conditions holds, the pair $(a,b)$ is composable:
\begin{enumerate}
\item \label{item:n1} $n_1=0$, i.e. $a \in \lDa(\oc{0}{m_1},\oc{n_2}{m_2})$ and $b$ arbitrary,
\item $b$ lies in the direct sum complex, i.e. $b \in \lDa(\oc{n_2}{m_2},\oc{n_3}{m_3})$ and $a$ is arbitrary, 
\item we have $a \in  \iplDa_{start}(\oc{n_1}{m_1},\oc{n_2}{m_2})$ and $b$ arbitrary.
\end{enumerate}
\end{Prop}
Note that for $a \in \lDa(\oc{n_1}{m_1},\oc{n_2}{m_2})$ and $b \in \lDa(\oc{n_2}{m_2},\oc{n_3}{m_3})$ their composition is just the composition $b \circ a$ by definition.
\begin{proof}
If $n_1=0$ by definition $a$ is contained in $\lDa(\oc{0}{m_1},\oc{n_2}{m_2})$ and thus it is a finite sum of diagrams in this complex. It is sufficient to show that for $x$ a looped diagram in $ \lDa(\oc{0}{m_1},\oc{n_2}{m_2})$ and $y \in \lDa(\oc{n_2}{m_2},\oc{n_3}{m_3})$ there are only finitely many types $(t'_1, \ldots, t'_{n_2})$ the diagram $y$ can have such that the composition $y \circ x$ is non-trivial. We have $(\sum_{i=1}^{n_2} t'_i)$ irreducible loops in $y$. By the observations we made earlier, we know, that we either have to glue an edge into each of these loops or cover it with a loop of $x$. Since $x$ does not have any loops, we have to glue at least one edge of $x$ in into each irreducible loop for the composition to be non-trivial. There are only $|x|$ edges which get glued to $y$, thus the composition is trivial whenever $(\sum_{i=1}^{n_2} t'_i)>|x|$.

If $n_1 \neq 0$ we show that for a given type $(t_1, \ldots, t_{n_1})$ there are only finitely many types $(t_1, \ldots, t_{n_1})$ and $(t'_1, \ldots, t'_{n_2})$ such that the compositions $( b_{t'_1, \ldots, t'_{n_2}} \circ a_{t_1, \ldots, t_{n_1}})$ have summands of type $(u_1, \ldots, u_{n_1})$. In the second case by assumption only finitely many types occur in $b$, i.e. we only need to show that for an arbitrary looped diagram $y$ there are only finitely many types a looped diagram $x$ can have, such that $y \circ x$ has type $(u_1, \ldots, u_{n_1})$. However, for a looped diagram $x$ of type $(t_1, \ldots, t_{n_1})$ the type of the composition $y \circ x$ is bounded below by $(t_1, \ldots, t_n)$. Therefore, in the $(u_1, \ldots, u_{n_1})$ component of the composition, we can only have elements resulting from the composition of $a_{t_1, \ldots, t_{n_1}}$ with $t_i \leq u_i$.

The proof of the last case is particularly technical and difficult to explain. Since the fact is not used later on, we omit the proof.
\end{proof}

The later proposition and the associativity of $\circ$ imply that the union of morphism spaces $\coprod_{n_1, m_1}\ilDa(\oc{n_1}{m_1},\oc{n_2}{m_2})$ is a left $\iplDa_{start}$-module and $\coprod_{n_2, m_2} \ilDa(\oc{n_1}{m_1},\oc{n_2}{m_2})$ a right $\lDa$-module.

\section{The natural operations for commutative Frobenius algebras}\label{sec:natcom}
\subsection{The category of commutative Frobenius algebras}\label{sec:comfrob}
This paper deals with commutative Frobenius algebras, so we start with recalling definitions on the subject.

\begin{Def}
A\emph{ Frobenius algebra} $A$ is given by a finite-dimensional vector space equipped with the following data:
\begin{itemize}
\item a multiplication $m: A \otimes A \to A$ and a unit $\1_A: \k \to A$ such that $m$ and $\1_A$ define an algebra structure on $A$
\item a comultiplication $\Delta: A \to A \otimes A$ and a counit $\eta: A \to \k$ such that they define a coalgebra structure on $A$
\end{itemize}
satisfying the so called Frobenius relation $\Delta \circ m= (m \otimes \id) \circ (\id \otimes \Delta)= (\id \otimes m) \circ (\Delta \otimes \id)$.
If $A$ is a chain complex, we obtain \emph{Frobenius dg-algebras}. 

We denote the twist map $A \otimes A \to A \otimes A$ by $
\tau$. 

A Frobenius algebra is called \emph{symmetric} if $\eta \circ m \circ \tau=\eta \circ m$ and it is \emph{commutative} if $ m \circ \tau= m$. A commutative Frobenius algebra is cocommutative, i.e. $\tau \circ \Delta= \Delta$.

An \emph{open Frobenius algebra} is a Frobenius algebra without a counit. It is cocommutative if $\tau \circ \Delta= \Delta$. In this case commutativity does not imply cocommutativity (but we will only work with commutative cocommutative open Frobenius algebras).

\end{Def}

The open cobordism category $\O$ was defined in \cite[section 2.6]{wahl11} to be the dg-category with objects the natural numbers and morphism $\O(n,m)$ the chain complex of oriented fat graphs with $n+m$ labeled leaves (for the definition see Section \ref{sec:graphs}), i.e. $\O(n,m)=\oc{0}{m+n}$--Graphs. The category $sFr=H_0(\O)$ is the category with the same objects but with cobordisms $sFr(n,m)=H_0(\O)(n,m):=H_0(\O(n,m))$. This chain complex consists of trivalent graphs module the sliding relation (cf. Figure \ref{fig:equiv}), i.e. $sFr(n,m)=\oc{0}{m+n}$--Sullivan diagrams, so these are Sullivan diagrams without white vertices. A split symmetric monoidal functor $\Phi:H_0(\O) \to \Ch$, i.e. an $sFr$-algebra is an open TQFT and by \cite[Cor 4.5]{Laud2008} these algebras are precisely the symmetric Frobenius dg-algebras. Usually one would pass to the closed cobordism category to deal with commutative Frobenius algebras, but we instead want to continue working with graphs (as it is for example done in \cite[Chapter 3]{kock2004}). Adding the commutativity relation is equivalent to forgetting the ordering of the edges at the vertices. Thus the PROP $cFr$ of commutative Frobenius algebras can be defined to have objects the natural numbers and morphisms $cFr(m_1,m_2)=\lDa(\oc{0}{m_1}, \oc{0}{m_2})=H_0(\O)(m_1,m_2)/\sim \ = \oc{0}{m_1+m_2}$--commutative Sullivan diagrams. So a commutative Frobenius dg-algebra is a strict symmetric monoidal functor from $cFr$ to chain complexes.
Forgetting the counit is equivalent to restricting to diagrams with the positive boundary condition (i.e. forcing every connected component to have an output). Thus we can define the PROP $cFr_+$ of commutative cocommutative open Frobenius algebras to have morphism spaces $cFr_+(m_1,m_2)=\lDa_+(\oc{0}{m_1}, \oc{0}{m_2})$. Hence a commutative cocommutative open dg-Frobenius algebra is a strict symmetric monoidal functor $cFr_+ \to \Ch$.

Moreover, we also have a graded version of commutative Frobenius algebras where the comultiplication has degree $d$ and the counit has degree $-d$. Analogously to \cite[Section 6.3]{wahl11} for symmetric Frobenius algebras, the shifted PROP $cFr_d$ then agrees with the PROP where we shifted the commutative Sullivan diagrams $\Gamma$ by $-d \cdot \chi(\Gamma, \partial_{out})$ (as defined in the end of Section \ref{sec:comsul}), i.e. $cFr_d(m_1,m_2)=\lDa_d(\oc{0}{m_1}, \oc{0}{m_2})$.

\subsection{Formal operations}\label{sec:formal}
\subsubsection{Definitions of the Hochschild complex and formal operations}
Let $\e$ be a PROP with a multiplication, i.e. a dg-PROP with a functor $Ass \to \e$ which is the identity on objects. We define $m^k_{i,j} \in \e(k, k-1)$ to be the image of the map in $Ass(k,k-1)$ which multiplies the $i$--th and $j$--th input and is the identity on all other elements.

We recall the definitions of the Hochschild and coHochschild constructions of functors from \cite[Section 1]{wahl12}.

For $\Phi: \e \to \Ch$ a dg-functor the \emph{Hochschild complex }of $\Phi$ is the functor $C(\Phi): \e \to \Ch$ defined by
\[ C(\Phi)(n)=\bigoplus_{k \geq 1} \Phi(k+n)[k-1]. \]
The differential is the total differential of the differential on $\Phi$ and the differential coming from the simplicial abelian group structure with boundary maps $d_i=\Phi(m^k_{i+1, i+2} + \id_n)$ where we set $m^k_{k,k+1}=m^k_{k,1}$ and degeneracy maps induced by the map inserting a unit at the $i+1$--st position.%

The \emph{reduced Hochschild complex} $\bC(\Phi)(n)$ is the reduced chain complex associated to this simplicial abelian group, i.e. the quotient by the image of the degeneracies.%

Iterating this construction, the functors $C^{(n,m)}(\Phi)$ and $\bC^{(n,m)}(\Phi)$ are given by 
\begin{align*}C^{(n,m)}(\Phi):=C^n(\Phi)(m) &&\text{and} &&\bC^{(n,m)}(\Phi):=\bC^n(\Phi)(m).\end{align*}
Working out the definitions explicitly, one sees that
\[C^{(n,m)}(\Phi) \cong \bigoplus_{j_1 \geq 1, \cdots, j_n \geq 1} \Phi(j_1+ \cdots+ j_n + m)[j_1+ \cdots + j_n -n].\]

The category of $\e$--algebras is equivalent to strict symmetric monoidal functors $\Phi: \e \to \Ch$, sending an algebra $A$ to the functor $A^{\otimes -}$. For an algebra $A$, the Hochschild complex $C(A^{\otimes -})(0)$ is the ordinary Hochschild complex $C_*(A,A)$ (and similarly for the reduced complexes). Furthermore, we have an isomorphism
\begin{equation} \label{eq:mon} C^{(n,m)}(A^{\otimes -}) \cong C_*(A,A)^{\otimes n} \otimes A^{\otimes m}\end{equation}
natural in all $\e$--algebras $A$.

Dually, given a dg-functor $\Psi: \e^{op} \to \Ch$ its \emph{CoHochschild complex} is defined as
\[ D(\Psi(n))= \prod_{k \geq 1} \Psi(k+n)[1-k]\] 
with the differential coming from the cosimplicial structure induced by the multiplications and the inner differential on $\Psi$. Again, we can take the reduced cochain complex $\bD(\Psi)(n)$. By \cite[Prop. 1.7 + 1.8]{wahl12}, the inclusion $\bD(\Psi) \to D(\Psi)$ and the projection $C(\Phi) \to \bC(\Psi)$ are quasi-isomorphisms. 

We can also spell out the iterated construction explicitly, i.e. for a functor $\Psi: \e^{op} \to \Ch$ we get
\[D^n(\Psi)(m) \cong \prod_{j_1, \cdots, j_n} \Psi(j_1 + \cdots + j_n +m)[n-(j_1+ \cdots + j_n)].\]

The complex of formal operations $\Nat_\e(\oc{n_1}{m_1},\oc{n_2}{m_2})$ is defined as all maps
\[C^{(n_1,m_1)}(\Phi) \to C^{(n_2,m_2)}(\Phi)\]
%
natural in all functors $\Phi: \e \to \Ch$.

In \cite[Theorem 2.1]{wahl12} it is shown that  
\[\Nat_\e(\oc{n_1}{m_1},\oc{n_2}{m_2}) \cong D^{n_1}C^{n_2}(\e(-,-))(m_2)(m_1),\]
which is used to compute the complex of formal operations explicitly.

Instead of testing on all functors $\Phi: \e \to \Ch$ we could test on strict symmetric monoidal functors only and denote the operations obtained this way by $\Nat_\e^\otimes(\oc{n_1}{m_1},\oc{n_2}{m_2})$. Via the isomorphism in equation \eqref{eq:mon} a transformation in $\Nat_\e^\otimes(\oc{n_1}{m_1},\oc{n_2}{m_2})$ corresponds to an operation
\[C_*(A,A)^{\otimes n_1} \otimes A^{\otimes m_1} \to C_*(A,A)^{\otimes n_2} \otimes A^{\otimes m_2}\]
natural in all $\e$--algebras $A$, so in other words it is a natural transformation of the Hochschild complex. Since every transformation in $\Nat_\e(\oc{n_1}{m_1},\oc{n_2}{m_2})$ is in particular natural in all strict symmetric monoidal functors, we have a restriction map $\rho: \Nat_\e(\oc{n_1}{m_1},\oc{n_2}{m_2}) \to \Nat_\e^\otimes(\oc{n_1}{m_1},\oc{n_2}{m_2})$, so every formal operation gives us a natural operations of the Hochschild complex of $\e$--algebras. In general we do not know whether this map is injective or surjective (for more details on this matter see \cite[Section 2.2]{wahl12}).

\subsubsection{Formal operations for commutative Frobenius algebras}
We now focus on the case where $\e=cFr$. Using the definitions of the previous section, we can describe the complexes $\bC^n(cFr(m_1,-))(m_2)$ and $\bC^n(cFr_+(m_1,-))(m_2)$ as follows:
\begin{Lemma}\label{le:natlooped}
There are isomorphisms
\[\bC^n(cFr(m_1,-))(m_2) \cong \lDa(\oc{0}{m_1}, \oc{n}{m_2}) = \oc{n}{m_1+m_2}-cSD\]
and
\[\bC^n(cFr_+(m_1,-))(m_2) \cong \lDa_+(\oc{0}{m_1}, \oc{n}{m_2}).\]
\end{Lemma} 
This is a direct analog of \cite[Lemma 6.1]{wahl11} in the commutative setting and the proof works completely similar.

Applying the coHochschild construction $n_1$ times, we can describe the formal operations for $cFr$ and $cFr_+$ via
\begin{align*}\bNat_{cFr}(\oc{n_1}{m_1},\oc{n_2}{m_2}) &\simeq D^{n_1}(\bC^{n_2}(cFr(-,-)(m_2))(m_1)\\&\cong \prod_{j_1, \cdots, j_{n_1}} \lDa(\oc{0}{j_1+ \cdots + j_{n_1}+m_1}, \oc{n_2}{m_2})[n_1-\Sigma j_i]\end{align*}
and
\[\bNat_{cFr_+}(\oc{n_1}{m_1},\oc{n_2}{m_2}) \simeq \prod_{j_1, \cdots, j_{n_1}} \lDa_+(\oc{0}{j_1+ \cdots + j_{n_1}+m_1}, \oc{n_2}{m_2})[n_1-\Sigma j_i].\]

Since every commutative Frobenius algebra is in particular a commutative cocommutative open Frobenius algebra, we have an induced inclusion $\bNat_{cFr_+}(\oc{n_1}{m_1},\oc{n_2}{m_2}) \hookrightarrow \bNat_{cFr}(\oc{n_1}{m_1},\oc{n_2}{m_2})$. Under the above equivalences, this inclusion corresponds to the inclusions of the subcomplexes $\lDa_+(\oc{0}{j_1+ \cdots + j_{n_1}+m_1}, \oc{n_2}{m_2}) \hookrightarrow \lDa(\oc{0}{j_1+ \cdots + j_{n_1}+m_1}, \oc{n_2}{m_2})$.

The composition in $\bNat_{cFr}$ (and thus also in $\bNat_{cFr_+}$) is described  in terms of the right hand side as follows:

For $\Gamma \in \prod_{j_1, \cdots, j_{n_1}} \oc{n_2}{j_1+ \cdots + j_{n_1}+m_1+m_2}-cSD$ and $\Gamma' \in \prod_{j_1, \cdots, j_{n_2}} \oc{n_3}{j_1+ \cdots + j_{n_2}+m_2+m_3}-cSD$ we get
$(\Gamma' \circ \Gamma )_{j_1, \ldots, j_{n_1}}$ by attaching a summand $G$ in $(\Gamma)_{j_1, \ldots, j_{n_1}}$ which has $n_2$ white vertices with each $k_1, \ldots, k_{n_2}$ half-edges, to the element $(\Gamma')_{k_1, \ldots, k_{n_2}}$. This is done by taking away the white vertices from $G$ and gluing the $k_1+ \cdots + k_{n_2}$ half-edges onto the according labeled leaves of $(\Gamma')_{k_1, \ldots, k_{n_2}}$.

Before we move on, we want to say a few words about how to view an element in $x \in \prod_{j_1, \cdots, j_{n_1}} \lDa(\oc{0}{j_1+ \cdots + j_{n_1}+m_1}, \oc{n_2}{m_2})$ as an operation on commutative Frobenius dg-algebras, i.e. how to extract an operation
\[CC_*(A,A)^{\otimes n_1} \otimes A^{m_1} \to CC_*(A,A)^{\otimes n_2} \otimes A^{m_2}.\]
We fix a tuple $(j_1, \cdots, j_{n_1})$ and an element 
\[(a_0^1 \otimes \cdots \otimes a_{j_1}^1) \otimes \cdots \otimes (a_0^{n_1} \otimes \cdots \otimes a_{j_{n_1}}^{n_1}) \otimes b_1 \otimes \cdots \otimes b_{m_1}\hspace{-0.1cm} \in \hspace{-0.1cm} CC_{j_1}(A,A) \otimes \cdots \otimes  CC_{j_n}(A,A) \otimes A^{m_1}.\]
To get the resulting element in  $CC_*(A,A)^{\otimes n_2} \otimes A^{m_2}$, we need to consider $x_{j_1+1, \cdots, j_{n_1}+1} \in \lDa(\oc{0}{j_1+1+ \cdots + j_{n_1}+1+m_1}, \oc{n_2}{m_2})$ which is a linear combination of $\oc{n_2}{j_1+1+ \cdots + j_{n_1}+1+m_1+m_2}$--looped diagrams. We start with writing $a_0^1, \cdots, a_{j_1}^1$ on the first $j_1+1$--leaves and continue by putting the other $a_i^k$ following their order. Then we write the $b_k$ on the leaves labeled $j_1+1+ \cdots + j_{n_1}+1+1$ to $j_1+1+ \cdots + j_{n_1}+1+m_1$. We put units on all unlabeled leaves. Now we have $m_2$ labeled leaves where we have not written an element of $A$ on. We view these and the half-edges attached to the white vertices as ends of the graph for a moment. Reading the black vertices of the diagram as multiplications and comultiplications, we obtain a linear combination of the diagram where we assigned values in $A$ to all ends of the graph (i.e. to the remaining $m_2$ leaves and all half-edges attached to the white vertices).

We read off our element in $CC_*(A,A)^{\otimes n_2} \otimes A^{m_2}$ by starting with the white vertices: Each white vertex corresponds to one copy of $CC_*(A,A)$. In the procedure described above, an element of $A$ is assigned to each half-edge attached to a white vertex. If the white vertex has degree $k$, i.e. $k+1$ edges attached to with labels $c_0, \cdots c_k$, the resulting element lies in $CC_k(A,A)$ and is given by $c_0 \otimes c_k$. The elements assigned to the $m_2$ leaves give the resulting elements in $A^{\otimes m_2}$.
\begin{figure}[!ht]

 \center
\input{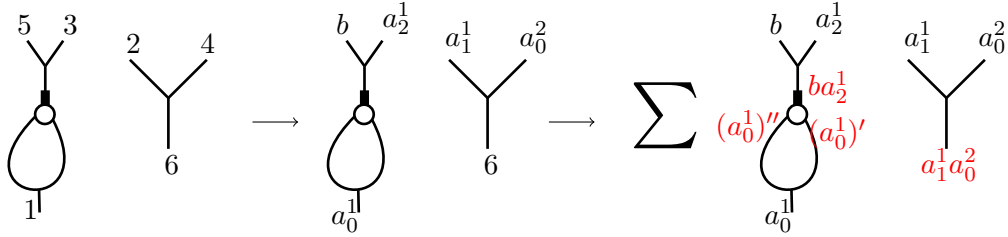}
\caption{How to read off operations}
	\label{fig:getelement}
\end{figure}

In Figure \ref{fig:getelement} we have illustrated how to evaluate an element of $\lDa(\oc{0}{3+1+1}, \oc{1}{1})$, as part of an operation in $\Nat_{cFr}(\oc{2}{1},\oc{1}{1})$, on an element $(a_0^1 \otimes a_1^1 \otimes a_2^1) \otimes (a_0^2) \otimes b \in CC_2(A,A) \otimes CC_0(A,A) \otimes A$. The sum in the last picture is the sum coming from the comultiplication of $a_0^1$ using Sweedler's notation $\Delta(a_0^1)= \sum (a_0^1)'\otimes (a_0^1)''$. The element we read off is $(\sum ba_2^1 \otimes (a_0^1)'\otimes (a_0^1)'') \otimes a_1^1 a_0^2 \in CC_2(A,A) \otimes A$.

\subsubsection{Splitting off zero chains}
It is well-known that for a commutative algebra $A$ the Hochschild chains $C_*(A,A)$ split as $C_0(A,A) \oplus C_{>0}(A,A)$. This relies on the fact that $d(a_0 \otimes a_1)=a_0 \cdot a_1- (-1)^{|a_0||a_1|} a_1 \cdot a_0=0$ by the commutativity of $A$. It generalizes to the Hochschild complex of functors, since we already get $d_0=- d_1 \in \Com(2,1)$ and thus $d=0 \in \Com(2,1)$. Hence for $\Phi: \Com \to \Ch$ the differential on degree one of the Hochschild complex $C_*(\Phi)$ is trivial and we get a splitting $C_*(\Phi) \cong C_0(\Phi) \oplus C_{>0}(\Phi)$. This generalizes to the iterated complex $C^{(n,m)}(\Phi)\cong \bigoplus_{j_1 \geq 1, \cdots, j_n \geq 1} \Phi(j_1+ \cdots+ j_n + m)$ which we therefore can rewrite as:
\[C^{(n,m)}(\Phi)=\bigoplus_{S \subseteq \{1, \cdots, n\}}\bigoplus_{j_i \geq 2, i\notin S} \Phi(j_1+ \cdots+ j_n + m)\]
with $j_i=1$ if $i \in S$.

For each $S \subseteq \{1, \cdots, n\}$ with $|S|=k$ the sum $\bigoplus_{j_i \geq 2, i\notin S} \Phi(j_1+ \cdots+ j_n + m)$ with $j_i=1$ for $i \in S$ is isomorphic to $\bigoplus_{j_{r_i} \geq 2} \Phi(j_{r_1}+ \cdots+ j_{r_{n-k}} +k+ m)$ given by relabeling those $j_i$ with $i \notin S$ to $j_{r_l}$ and moving the $j_i$ with $i =1$ to the end (with a sign involved). Defining $C^{>0,(n,m)}_\e:=\bigoplus_{j_1 \geq 2, \cdots, j_n \geq 2} \Phi(j_1+ \cdots+ j_n + m)$, we see that \[\bigoplus_{j_{r_i} \geq 2} \Phi(j_{r_1}+ \cdots+ j_{r_{n-k}} +k+ m) \cong C^{>0,(n-k,m+k)}(\Phi)\] and hence we get an isomorphism%
\[C^{(n,m)}(\Phi)\cong \bigoplus_{S \subseteq \{1, \cdots, n\}} C^{>0,(n-|S|,m+|S|)}(\Phi).\] 

Defining $\Nat^{>0}_\e(\oc{n_1}{m_1},\oc{n_2}{m_2}):=\hom_\e(C^{>0,(n_1,m_1)}(-), C^{(n_2,m_2)}(-))$ we conclude
\begin{align*}
\Nat_\e(\oc{n_1}{m_1},\oc{n_2}{m_2})&=\hom_\e(C^{(n_1,m_1)}(-), C^{(n_2,m_2)}(-))\\
&=\bigoplus_{S \subseteq \{1, \cdots, n_1\}}\hom_\e(C^{>0,(n_1-|S|,m_1+|S|)}(-), C^{(n_2,m_2)}(-))\\
&=\bigoplus_{S \subseteq \{1, \cdots, n_1\}} \Nat^{>0}_\e(\oc{n_1-|S|}{m_1+|S|},\oc{n_2}{m_2}).
\end{align*}
The same works for the reduced Hochschild construction and reduced natural transformations, i.e.
\[\bNat_\e(\oc{n_1}{m_1},\oc{n_2}{m_2})=\bigoplus_{S \subseteq \{1, \cdots, n_1\}} \bNat^{>0}_\e(\oc{n_1-|S|}{m_1+|S|},\oc{n_2}{m_2}).\]

Similarly, for a functor $\Psi: \e^{op} \to \Ch$ we define $D^{>0,n}(\Psi)(m):= \prod_{j_i>2} \Psi(j_1+ \cdots + j_{n}+m)$. Going through the proof of \cite[Theorem 2.1]{wahl12} one sees that
$\bNat^{>0}_{\e}(\oc{n_1}{m_1},\oc{n_2}{m_2}) \cong \bD^{>0,n_1}(\bC^{n_2}(\e(-,-)(m_2))(m_1) \simeq D^{>0,n_1}(\bC^{n_2}(\e(-,-)(m_2))(m_1)$.

Using the positive coHochschild construction, we can identify the subcomplex $\Nat_{cFr}^{>0}$ via the following:
\[\bNat_{cFr}^{>0}(\oc{n_1}{m_1},\oc{n_2}{m_2}) \simeq \prod_{j_1>1, \cdots, j_{n_1}>1} \lDa(\oc{0}{j_1+ \cdots + j_{n_1}+m_1}, \oc{n_2}{m_2})[n_1-\Sigma j_i].\]


\subsection{Building operations out of looped diagrams}\label{sec:build}

We describe a dg-functor from $\lDa$ to $\bNat_{cFr}$ which is the identity on objects. Thus  we assign an operation on commutative Frobenius dg-algebras to every looped diagram. In the two sections afterward we will show that this actually covers interesting operations.

Recall the graph $l_j \in \lDa(\oc{0}{j}, \oc{1}{0})$ defined in Section \ref{sec:graphs} given by a single white vertex with $j$ leaves attached to it and let $\id_{m_1} \in \lDa(\oc{0}{m_1}, \oc{0}{m_1})$ be the identity element, so we have $(l_{j_1} \amalg \cdots \amalg l_{j_{n_1}} \amalg \id_{m_1})\} \in \lDa(\oc{0}{j_1+ \cdots + j_{n_1}+m_1}, \oc{n_1}{m_1})$ of degree $\sum j_i-n_1$.

\begin{Theorem} \label{Th:buildop}
There is a functor of dg-categories
\[J:\lDa \to \bNat_{cFr}\]
which is the identity on objects and sends a looped diagram $G \in \lDa(\oc{n_1}{m_1}, \oc{n_2}{m_2})$ to $(G \circ (l_{j_1} \amalg \cdots \amalg l_{j_{n_1}} \amalg \id_{m_1}))_{j_1, \cdots, j_{n_1}} \in \prod_{j_1, \cdots, j_{n_1}} \lDa(\oc{0}{j_1+ \cdots + j_{n_1}+m_1}, \oc{n_2}{m_2})[n_1-\Sigma j_i]$.
The functor restricts to functors
\begin{align*} \lDa_+ \to \bNat_{cFr_+}, &&\lD \to \bNat^{>0}_{cFr}&&\text{and}&&\lD_+ \to \bNat^{>0}_{cFr_+}.\end{align*}
\end{Theorem}
\begin{proof}
We need to prove that the above map preserves the identity and composition:

The former follows from the fact that the element $l_{j_1} \amalg \cdots \amalg l_{j_{n_1}} \amalg \id_{m_1}$ is the identity in $\bNat^{>0}_{cFr}(\oc{n_1}{m_1}, \oc{n_1}{m_1})$. This is also illustrated in Example \ref{ex:id}.

Let $x \in \lDa(\oc{n_1}{m_1}, \oc{n_2}{m_2})$ and $y \in \lDa(\oc{n_2}{m_2}, \oc{n_3}{m_3})$, so $J(x) \in \prod_{j_1, \cdots, j_{n_1}} \lDa(\oc{0}{j_1+ \cdots + j_{n_1}+m_1}, \oc{n_2}{m_2})$ and $J(y) \in \prod_{j_1, \cdots, j_{n_2}} \lDa(\oc{0}{j_1+ \cdots + j_{n_2}+m_2}, \oc{n_3}{m_3})$ (with shifted degrees). More precisely, we get $J(x)$ from $x$ by gluing the $j_i$ leaves along the loop $\gamma_i$ in all possible order preserving ways and similar for $J(y)$.
In the composition $J(y) \circ J(x)$  we glue a term in $J(x)$ which at the white vertex $v_i$ has degree $d_i$ onto $J(y)_{(d_1+1, 
\cdots, d_{n_2}+1)}$. This is the same as first gluing all edges of $x$ in all possible ways onto the loops in $y$ and then the leaves onto the new loops, which is exactly what $J(y \circ x)$ does.
\end{proof} 
The above theorem provides us with a big family of operations. Unfortunately, these do not cover all operations we know, in particular not all operations coming from operations on commutative algebras (which will be investigated in Section \ref{sec:com}). The next theorem provides a bigger set of operations, but we have to be more careful with composition, since as seen earlier, $\ilDa$ is not a category anymore.

\begin{Theorem}\label{Th:buildop2}
We have dg-maps
\[J_{cFr}:\ilDa(\oc{n_1}{m_1}, \oc{n_2}{m_2}) \to \bNat_{cFr}(\oc{n_1}{m_1}, \oc{n_2}{m_2})\]
and
\[J_{cFr_+}:\ilDa_+(\oc{n_1}{m_1}, \oc{n_2}{m_2}) \to \bNat_{cFr_+}(\oc{n_1}{m_1}, \oc{n_2}{m_2})\]
preserving the composition of composable objects.
\end{Theorem}
\begin{proof}
We need to show that the map is well-defined, i.e. that an infinite sum $a=\sum_{t_1=1}^\infty \cdots \sum_{t_{n_1}=1}^\infty a_{t_1, \ldots, t_{n_1}} \in \ilD(\oc{n_1}{m_1},\oc{n_2}{m_2})$ is taken to a well-defined element in the complex $ \prod_{j_1, \cdots, j_{n_1}} \lDa(\oc{0}{j_1+ \cdots + j_{n_1}+m_1}, \oc{n_2}{m_2})$. So we show that for a fixed tuple $(j_1, \cdots, j_{n_1})$ only finitely many of the $J(a_{t_1, \ldots, t_{n_1}})_{(j_1, \cdots, j_{n_1})}$ are non-zero. In the composition with a diagram of type $(t_1, \cdots, t_{n_1})$ none of the $\sum t_i$ loops can be empty and by composing with $l_i$'s only leaves are glued on. Hence for the composition to be non-zero we need that $j_i> t_i$. Thus the claim is shown.
We only need to see that composition of composable elements is preserved. In order to do so we use that all operations in $\Nat_{cFr}$ are composable, so if two elements $a=\sum_{t_1=1}^\infty \cdots \sum_{t_{n_1}=1}^\infty a_{t_1, \ldots, t_{n_1}} \in \ilDa(\oc{n_1}{m_1},\oc{n_2}{m_2})$ and $b=\sum_{t'_1=1}^\infty \cdots \sum_{t'_{n_2}=1}^\infty b_{t'_1, \ldots, t'_{n_2}} \in \ilDa(\oc{n_2}{m_2},\oc{n_3}{m_3})$ are composable then 
\begin{align*} 
J(b \circ a) &= J\left(\sum_{(t_1, \cdots, t_{n_1})} \sum_{(t'_1, \cdots, t'_{n_2})}b_{t'_1, \ldots, t'_{n_2}} \circ a_{t_1, \ldots, t_{n_1}}\right)\\
&=\sum_{(t_1, \cdots, t_{n_1})} \sum_{(t'_1, \cdots, t'_{n_2})}J\left(b_{t'_1, \ldots, t'_{n_2}} \circ a_{t_1, \ldots, t_{n_1}}\right)\\
&=\sum_{(t_1, \cdots, t_{n_1})} \sum_{(t'_1, \cdots, t'_{n_2})}J\left(b_{t'_1, \ldots, t'_{n_2}}\right) \circ J\left(a_{t_1, \ldots, t_{n_1}}\right)
\end{align*}
so the composition agrees.
\end{proof}

\begin{Remark}[Operations of type $(t_1, \cdots, t_n)$]\label{rem:optype}
At this point we want to explain how the map $J_\Com$ actually acts on an element of type $(t_1, \cdots, t_{n_1})$. 
There is an easy way to read off the operation of such an element without going back to the original definition of the type. 
For a general element $x=(\Gamma, \langle\gamma^1_1, \ldots, \gamma^{t_1}_1\rangle, \ldots, \langle\gamma^1_{n_1}, \ldots, \gamma^{t_{n_1}}_{n_1}\rangle) \in \ilDa(\oc{n_1}{m_1}, \oc{n_2}{m_2})$ of type $(t_1, \cdots, t_{n_1})$ the composition with $(l_{j_1} \amalg \cdots \amalg l_{j_{n_1}} \amalg \id_{m_1})$ is trivial if there is a $j_i < t_i$ and is given by all possible ways of (for each $i$) gluing the $j_i$ labeled leaves along the $\gamma_i^r$ (respecting the order of the leaves and the loops) such that we glued at least one leaf to each $\gamma_i^r$ (and gluing the $m_1$ extra leaves as usual). In particular the image $J_{cFr}(x) \in \Nat(\oc{n_1}{m_1}, \oc{n_2}{m_2})$ acts trivial on all Hochschild degrees $(j_1, \ldots, j_{n_1})$ with $j_i< t_i$ for some $i$.
\end{Remark}

\subsection{Connection to non-commutative operations}\label{sec:con}
The analog of the functor $J$ has been defined in the context of symmetric Frobenius algebras in \cite[Section 3]{wahl12}. There, it was even shown to be a split quasi-isomorphism of complexes:
\begin{Theorem}[{\cite[Theorem 3.8]{wahl12}}]
The functor \[J_{H_0}: \SD \to \Nat_{sFr}\] defined by sending a graph $G$ to $(G \circ (l_{j_1} \amalg \cdots \amalg l_{j_{n_1}} \amalg \id_{m_1}))_{j_1, \cdots, j_{n_1}}$ is a split quasi-isomorphism.
\end{Theorem}
Using the functor $sFr \to cFr$ which induces a functor $\Nat_{sFr} \to \Nat_{cFr}$ and recalling the functor $K: \SD \to \lDa$ defined at the end of Section \ref{sec:comsul} (cf. Figure \ref{fig:funcK}), one checks that the definitions were made to make the following proposition true:
\begin{Prop}\label{prop:commsym}
The diagram
\[ \xymatrix{ \SD \ar[d]^-{K} \ar[r]^-{J_{H_0}}& \Nat_{sFr} \ar[d] \\ \lDa \ar[r]^-{J}& \Nat_{cFr} } \]
commutes.
\end{Prop}

\begin{Remark}
Everything done so far works also in the shifted setup, in particular we get a functor
\[J:\lDa_d \to \Nat_{cFr_d}\]
and a commutative diagram
\[ \xymatrix{ \SD_d \ar[d]^-{K} \ar[r]^-{J_{H_0}}& \Nat_{sFr_d} \ar[d] \\ \lDa_d \ar[r]^-{J}& \Nat_{cFr_d}. } \]
\end{Remark}
\subsection{First examples of operations and relations}\label{sec:exop}
Before we investigate some subcomplexes of $\lDa$ more systematically, we want to represent some of the known operations on the Hochschild homology of commutative Frobenius algebras by looped diagrams and apply the new tools to easily prove a relation between some of them.
\begin{figure}[ht!]  \begin{center}
\subfigure
       [The identity]
        {\hspace{1cm}
          
            \input{pic_id}\hspace{1cm}
            \label{fig:id}
        }%
        \subfigure
       [The image of the identity under $J$]
        {\hspace{1cm}          
            \input{pic_idcomp}\hspace{1cm}           
           \label{fig:idcomp}
        }%
    \end{center}
   \caption[The identity element]{
       The element $\id_{\oc{1}{0}} \in \lD_+(\oc{1}{0}, \oc{1}{0})$ and its image in $\Nat_{cFr_+}(\oc{1}{0}, \oc{1}{0})$
     }
   \label{fig:idall}
\end{figure}
\begin{Example}[Identity]\label{ex:id}
As seen before, the diagram $\id_{\oc{1}{0}} \in \lD_+(\oc{1}{0}, \oc{1}{0})$ shown in Figure \ref{fig:id} corresponds to the identity in $\Nat_{cFr_+}(\oc{1}{0}, \oc{1}{0})$. In order to see this in pictures, we spell out the map $J:\lD_+(\oc{1}{0}, \oc{1}{0}) \to \Nat_{cFr_+}(\oc{1}{0}, \oc{1}{0})$ explicitly. By definition in degree $j$ we have $J(\id_{\oc{1}{0}})_{j}= \id_{\oc{1}{0}} \circ l_j$. There is only one way to glue the edges of $l_j$ onto $\id_{\oc{1}{0}} $, so the resulting diagram in $\lDa_+(\oc{0}{j}, \oc{1}{0})$ is shown in Figure \ref{fig:idcomp}, where we labeled the leaves of $l_j$ by $a_1, \cdots, a_j$. Given a Hochschild chain $a_1 \otimes \cdots \otimes a_j$ the image of the operation is now given by reading off around the white vertex. Thus we get $a_0 \otimes \cdots \otimes a_j$ back.
\end{Example}

\begin{figure}[ht!]  \begin{center}
\subfigure
       [The shuffle product $pr$]
        {\hspace{1cm}
          
            \input{pic_pr}\hspace{1cm}
            \label{fig:pr}
        }%
        \subfigure
       [$J(pr)_{2,3}=pr \circ (l_2 \amalg l_3)$]
        {\hspace{1cm}          
            \input{pic_prcomp}\hspace{1cm}           
           \label{fig:prcomp}
        }%
    \end{center}
   \caption[The shuffle product]{
       The element $pr \in \lD_+(\oc{2}{0}, \oc{1}{0})$ and the degree $(2,3)$ part of its image in $\Nat_{cFr_+}(\oc{2}{0}, \oc{1}{0})$
     }
   \label{fig:prall}
\end{figure}
\begin{Example}[Shuffle product]
The operation $pr \in \lD_+(\oc{2}{0}, \oc{1}{0})$ shown in Figure \ref{fig:pr} is the shuffle product on the Hochschild homology. The composition  $pr \circ (l_{j_1} \amalg l_{j_2})$ glues the first labeled leaves of $l_{j_1}$ and $l_{j_2}$ onto the start half-edge of $pr$ and all other edges around the white vertex keeping the cyclic ordering of the edges coming from $l_{j_1}$ and the cyclic ordering of the edges coming from $l_{j_2}$. Thus it produces all shuffles of these edges and hence corresponds to the shuffle product on the Hochschild chains. The example is illustrated in Figure \ref{fig:prcomp} for $j_1=2$ and $j_2=3$, where we again already labeled the leaves by $a_i$ and $b_i$ to give a clearer understanding of the final operation.
\end{Example}

\begin{figure}[ht!]
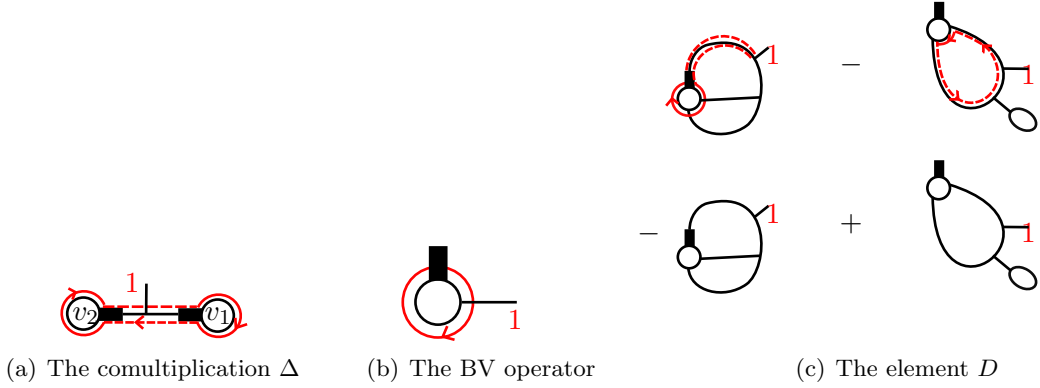
  \begin{center}
\subfigure
       [The  comultiplication $\Delta$]
        {\hspace{1cm}
          
            \input{pic_coprod}\hspace{1cm}
            \label{fig:coprod}
        }%
                \subfigure
       [The  BV operator]
        {\hspace{0.5cm}          
            \input{pic_BV}\hspace{1cm}           
           \label{fig:BV}
        }%
        \subfigure
       [The  element $D$]
        {       
            \input{pic_D1}\hspace{1cm}           
           \label{fig:D}
        }%
    \end{center}
   \caption{
       The comultiplication, the BV-operator and the boundary element $D$
     }
   \label{fig:elD}
\end{figure}
Recall from the introduction that for a $1$--connected closed oriented manifold, we have an isomorphism $HH_*(C^{-*}(M), C^{-*}(M)) \cong H^{-*}(LM)$. On $H^{-*}(LM)$ we have a coproduct (the dual of the Chas-Sullivan product) and a BV-operator. On the other hand, working with coefficients in $\Q$, by \cite{lamb07} there is a commutative Frobenius algebra $A$ of degree $d$ for $d=\dim M$ such that we have a weak equivalence $C^*(M) \simeq A$ and hence $HH_*(C^{-*}(M), C^{-*}(M)) \cong HH_*(A^{-*}, A^{-*})$. Since $A^{-*}$ is a commutative Frobenius algebra of degree $-d$, we have an action of $\lDa_{-d}$ on $HH_*(A^{-*}, A^{-*})$ and can show:
\begin{Prop}[BV-structure on $H^*(LM, \Q)$]
Working with coefficients in $\Q$, the coBV structure on $HH_*(C^{-*}(M), C^{-*}(M)) \cong HH_*(A^{-*}, A^{-*})$, induced via the above isomorphism by the dual of the Chas-Sullivan product and the BV-operator on $H^{-*}(LM, \Q)$, is generated by the operations $J(\Delta)\in \Nat_{cFr_{-d}}(\oc{1}{0},\oc{2}{0})$ and $J(B) \in \Nat_{cFr_{-d}}(\oc{1}{0},\oc{1}{0})$ with $\Delta \in \lDa_{-d}(\oc{1}{0},\oc{2}{0})$ and $B \in \lDa_{-d}(\oc{1}{0},\oc{1}{0})$ the shifted versions of the diagrams illustrated in Figure \ref{fig:coprod} and Figure \ref{fig:BV}.
\end{Prop}
\begin{proof}
The diagrams $\Delta$ and $B$ are the images of the diagrams illustrated in \cite[Figure 13]{wahl11} under the functor $K: \SD \to \lD$. Thus, the result is a direct consequence of \cite[Prop. 6.10]{wahl11}.
\end{proof}
\begin{figure}[!ht]

 \center
\input{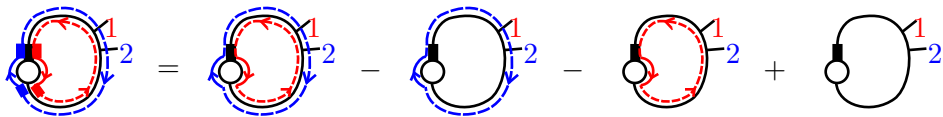}
\caption{The product $\mu$}
	\label{fig:prod}
\end{figure}
\begin{Example}[Product of suspended BV-structure on $HH_*$]
The image of the diagram $\mu \in \lD_+(\oc{2}{0},\oc{1}{0})$ shown in Figure \ref{fig:prod} (in both ways of visualizing it as an element in $\lD$ and $\lDa$) defines a product on the Hochschild chains of commutative cocommutative open Frobenius dg-algebras. This product was introduced in \cite[Section 7]{abba13} and \cite[Section 6]{abba13b}, where it was also shown that together with the BV-operator it induces a BV-structure on the Hochschild homology. In Section \ref{sec:cacti} we will show that it is part of a desuspended Cacti operad.
\end{Example}
The product $\mu$ shows similar behavior as the Goresky-Hingston product on $H^*(LM, M)$ (cf. \cite{gore09}). For example, the composition of the Goresky-Hingston coproduct with the Chas-Sullivan product is zero. We can show a similar observation for $\mu \circ \Delta$, namely:
\begin{Prop}
The composition $\mu \circ \Delta$ is a boundary in $\lDa(\oc{1}{0}, \oc{1}{0})$, i.e. it gives a trivial operation on Hochschild homology of commutative cocommutative open Frobenius dg-algebras. 
\end{Prop}
\begin{proof}
In Figure \ref{fig:rel} we have computed $\mu \circ \Delta$. This is equal to the boundary of the element $D$ defined in Figure \ref{fig:D}.
\end{proof}

\begin{figure}[!ht]

 \center
\input{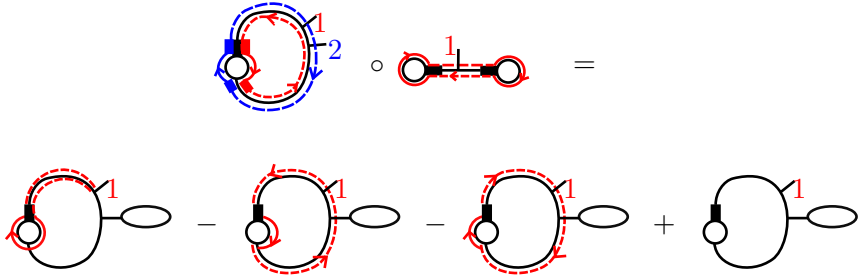}
\caption{The composition $\mu \circ \Delta=d(D)$}
	\label{fig:rel}
\end{figure}
Hence simultaneously with \cite{abba13b} we conjecture:
\begin{Conjecture}\label{conj:GH}
Under the isomorphism $H^{-*}(LM, \Q) \cong HH_*(C^{-*}(M), C^{-*}(M)) \cong HH_*(A^{-*}, A^{-*})$ the the Goresky-Hingston product corresponds to the (shifted) operation induced by $\mu \in  \lD_{-d}(\oc{2}{0},\oc{1}{0})$ shown in Figure \ref{fig:prod}.
\end{Conjecture}

\section{The operations coming from commutative algebras}\label{sec:com}

We have a map of PROPs $\Com \to cFr$ which is the identity on objects and an inclusion on morphism spaces (since the structure of commutative Frobenius algebras includes the structure of commutative algebras). Therefore, we get an inclusion $\Nat_\Com \to \Nat_{cFr}$. This inclusion is split and factors through  $\Nat_{cFr_+}$.

In \cite{kla13}, working over a field $\k$, we recalled the shuffle operations defined in \cite{loda89} and computed the homology of $\Nat_\Com$ in terms of infinite sums of shuffles of these. In this section we define a split subcomplex of $\iplDa(\oc{n_1}{m_1}, \oc{n_2}{m_2})$ whose image under the map $J_{cFr}:\iplDa(\oc{n_1}{m_1}, \oc{n_2}{m_2}) \to \Nat_{cFr_+}(\oc{n_1}{m_1}, \oc{n_2}{m_2})$ defined in Theorem \ref{Th:buildop2} is quasi-isomorphic to $\Nat_\Com(\oc{n_1}{m_1}, \oc{n_2}{m_2})$. On the level of complexes we give an even smaller subcomplex of $\iplDa(\oc{n_1}{m_1}, \oc{n_2}{m_2})$ which has trivial differential such that the map to $\Nat_\Com(\oc{n_1}{m_1}, \oc{n_2}{m_2})$ is a quasi-isomorphism, too. 

\begin{Def}\label{def:com}
The subcomplex $\plDa_\Com(\oc{n_1}{m_1}, \oc{n_2}{m_2})$ of $\plDa_+(\oc{n_1}{m_1}, \oc{n_2}{m_2})$ is spanned by all looped diagrams $(\Gamma, \gamma_1, \cdots, \gamma_{n_1})$ such that $\Gamma$ is a disjoint union of $n_2$ white vertices with trees attached to it (i.e. leaves multiplied together and attached to the white vertex and thus each irreducible loop goes once around the whole vertex) and $m_2$ labeled outgoing leaves with trees attached to them.
The complex $\iplDa_\Com(\oc{n_1}{m_1}, \oc{n_2}{m_2})$ is given as above by taking products over the type of diagrams as defined in Section \ref{sec:type}.
\end{Def}
Note that an element of type $(t_1, \cdots, t_{n_1})$ corresponds to the $i$--th loop given by $\gamma_i=\langle\underbrace{\sigma, \cdots, \sigma}_{t_i}\rangle$ with $\sigma$ the loop going once around the vertex starting at the leaf $i$. An example of an element in $\plDa_\Com(\oc{4}{2}, \oc{2}{1})$ of type $(2,1,1,1)$ is given in Figure \ref{fig:com1} (where we label the leaves corresponding to $m_1$ and $m_2$ with subindices).
\begin{figure}[!ht]

 \center
\input{pic_com1}

\caption[An element in $\plDa_\Com$]{An element in $\plDa_\Com({\oc{4}{2}, \oc{2}{1}})$ of type $(2,1,1,1)$}	
	\label{fig:com1}
\end{figure}

\begin{Example}
We can describe the generators of $\plDa_\Com(\oc{1}{0}, \oc{1}{0})$ explicitly:

In degree zero it is generated by a family of elements $sh^n$.
The element $sh^n$ of type $n$ is defined as $(\Gamma, \langle \sigma, \cdots, \sigma\rangle)$ where $\Gamma$ is the diagram with one white vertex and the leaf attached to the start half-edge and $\sigma$ the loop going once around the vertex (i.e. we have $n$ irreducible loops around the vertex). For $n=0$ we only have the underlying diagram, representing the inclusion of the algebra. For an example see Figure \ref{fig:sh}.

The above elements have been defined using the type decomposition from Section \ref{sec:type} which we need to use if we want to take products over the elements. However, we saw that we can rewrite them by ordinary elements in $\plDa(\oc{1}{0}, \oc{1}{0})$. Thus for $n >0$, we define the element $\lambda^n$ to be the non-constant diagram in $\plD(\oc{1}{0}, \oc{1}{0})$ looping around the white vertex $n$ times (i.e. the loop $\sigma^{\ast n}$ which is not irreducible) and define $\lambda^0=sh^0$. These elements give us another generating family of $\plD(\oc{1}{0}, \oc{1}{0})$ and we will come back to them when explaining the operations corresponding to these diagrams. An example is given in Figure \ref{fig:lambda}.

For the degree one part of $\plDa_\Com(\oc{1}{0}, \oc{1}{0})$ we only give the generating family in terms of the type decomposition:
Recall from Figure \ref{fig:BV} that the BV-operator is the looped diagram with a unit at the start half-edge, one labeled incoming leaf attached to the white vertex and a loop going once around from that leaf.

The elements $B^n$ of type $n$ are defined as $B \circ sh^n$, so they are given by the same graphs as the BV-operator but have $n$ irreducible loops going around (cf. Figure \ref{fig:b}).

\begin{figure}[ht!]  \begin{center}
\subfigure
       [The  element $sh^2$]
        {\hspace{1cm}
          
            \input{pic_sh}\hspace{1cm}
            \label{fig:sh}
        }%
        \subfigure
       [The  element $\lambda^2$]
        {\hspace{1cm}
          
            \input{pic_lambda}\hspace{1cm}
            \label{fig:lambda}
        }%
        \subfigure
       [The  element $B^2$]
        {\hspace{1cm}
          
            \input{pic_b2}\hspace{1cm}
            \label{fig:b}
        }%
 
    \end{center}
   \caption[General elements]{
       General elements in $\plDa(\oc{1}{0}, \oc{1}{0})$}
   \label{fig:exgen}
\end{figure}
\end{Example}

\begin{Lemma}\label{le:comcomp}
All morphisms in $\iplDa_\Com$ are composable and thus $\iplDa_\Com$ is a dg-category.
\end{Lemma}
\begin{proof}
It is enough to check the claim on generators.
Given $a \in \plDa_\Com(\oc{n_1}{m_1}, \oc{n_2}{m_2})$ of type $(t_1, \ldots, t_{n_1})$ and $b \in \plDa_\Com(\oc{n_2}{m_2}, \oc{n_3}{m_3})$ of type $(t'_1, \ldots, t'_{n_2})$ we give conditions on the type of the composition being non-zero. First, every irreducible loop of $a$ becomes an irreducible loop in the composition, i.e. the type of non-trivial elements in $a \circ b$ is bounded below by $(t_1, \ldots, t_{n_1})$.

In $\plDa_\Com(\oc{n_1}{m_1}, \oc{n_2}{m_2})$ we can rewrite every diagram as a diagram without loops together with its type. In particular, in $a \circ b$ the underlying diagrams are of the form $a \circ \Gamma'$ for $\Gamma'$ the underlying diagram of $b$. 
Assuming that the minimal non-trivial type of the composition $a \circ b$ is $(u_1, \ldots, u_{n_1})$, we obtain that composition with $(l_{u_1+1} \amalg \cdots \amalg l_{u_{n_1}+1} \amalg \id_{m_2})$ is non-trivial (different underlying diagrams of the same type have different images under this composition). Thus $J(b \circ a)_{(u_1+1, \cdots, u_{n_1}+1)}$ is non-zero. However, we know that in the composition $(J(b) \circ J(a))_{l_1, \cdots, l_{n_1}}$ one glues the summands of $J(a)_{l_1, \cdots, l_{n_1}}$ onto terms $J(b)_{j_1, \cdots, j_{n_2}}$ with $\sum {j_i}=\deg(a) + \sum (l_i-1)$. For this to be non-zero, we need $J(b)_{j_1, \cdots, j_{n_2}}$ to be non-zero, which is only true if $j_i > t'_i$ for all $i$. Thus we conclude that $\sum t'_i < \sum j_i = deg(a) + \sum u_i$ and hence for any type $(u_1 \cdots, u_{n_1})$ occurring in $b \circ a$ we have $\sum t'_i-\deg(a) <  \sum u_i$ and $t_i \leq u_i$ for all $i$. This proves the lemma.
\end{proof}
There is an analog of Lemma \ref{le:natlooped} for commutative algebras:
\begin{Lemma}\label{le:comnat}
We have an isomorphism
\[\bC^n(\Com(m_1,-))(m_2) \cong \plDa_\Com(\oc{0}{m_1}, \oc{n}{m_2}) \]
and hence a weak equivalence 
\[\prod_{j_1, \cdots, j_{n_1}} \plDa_\Com(\oc{0}{j_1+ \cdots + j_{n_1}+m_1}, \oc{n_2}{m_2})[n_1-\Sigma j_i] \to \bNat_\Com(\oc{n_1}{m_1}, \oc{n_2}{m_2}).\]
\end{Lemma}
It follows that the diagram
\[\xymatrix{ \prod_{j_1, \cdots, j_{n_1}} \plDa_\Com(\oc{0}{j_1+ \cdots + j_{n_1}+m_1}, \oc{n_2}{m_2})[n_1-\Sigma j_i] \ar@{^{(}->}[d] \ar[rr]^-{\simeq}&& \bNat_\Com(\oc{n_1}{m_1}, \oc{n_2}{m_2})\ar@{^{(}->}[d]\\ \prod_{j_1, \cdots, j_{n_1}} \plDa(\oc{0}{j_1+ \cdots + j_{n_1}+m_1}, \oc{n_2}{m_2})[n_1-\Sigma j_i] \ar[rr]^-{\simeq}& &\bNat_{cFr}(\oc{n_1}{m_1}, \oc{n_2}{m_2})}\]
commutes. 

Hence the dg-map $J_{cFr}:\iplDa(\oc{n_1}{m_1}, \oc{n_2}{m_2}) \to \bNat_{cFr}(\oc{n_1}{m_1}, \oc{n_2}{m_2})$ restricts to a dg-map
$J_{\Com}:\iplDa_\Com(\oc{n_1}{m_1}, \oc{n_2}{m_2}) \to \bNat_{\Com}(\oc{n_1}{m_1}, \oc{n_2}{m_2})$.

By Lemma \ref{le:comcomp} all morphisms in $\iplDa_\Com$  are composable and since $J_{cFr}$ and hence also $J_\Com$ preserve the composition of composable objects, the map $J_\Com$ is a natural transformation of categories $J_\Com: \iplDa_\Com \to \Nat_\Com$.

Let $m_{r_1, \cdots, r_n}$ be the $\oc{1}{(0,n)}$--looped diagram with the tree with $n$ leaves labeled $r_1$ to $r_n$ attached to the start half-edge, no other half-edges attached to the white vertex and a loop going around the white vertex for each $i$. 
For $n=0$ the diagram $m_\emptyset$ only has a unit at the start half-edge. 
Denote by $\overline{m}_{r_1, \cdots, r_n}$ the $\oc{0}{(n+1,0)}$--looped diagram consisting of the tree with incoming leaves labeled by $r_1, \cdots r_n$ and one outgoing leaf. For $m=0$ this diagram only has one unlabeled half-edge which is the outgoing leaf (i.e. it is a unit).
\begin{figure}[ht!]  \begin{center}
\subfigure
       [The  element $sh^2$]
        {\hspace{1cm}
          
            \input{pic_sh}\hspace{1cm}
           
        }%
        \subfigure
       [The  element $B^2$]
        {\hspace{1cm}
          
            \input{pic_b2}\hspace{1cm}
            
        }%
 \subfigure
        [The  element $m_{1,4,6,7}$]
        {
          
            \input{pic_m}\hspace{0.5cm}
        }\hspace{0.8cm}
\subfigure
        [The  element $\overline{m}_{1,4,6,7}$]
        {
            \hspace{1cm}
            \input{pic_mbar}
            \hspace{0.5cm}

        }
    \end{center}
   \caption{
       Building blocks of $\tplDa$
     }
   \label{fig:build}
\end{figure}

\begin{Def}\label{def:tcom}
The subcomplex $\tplDa(\oc{n_1}{m_1}, \oc{n_2}{m_2})$ of $\plDa_\Com(\oc{n_1}{m_1}, \oc{n_2}{m_2})$ spanned by elements $x$ obtained as follows:
Given 
\begin{itemize}

\item a function $f:\{1, \cdots, n_1+m_1\} \to \{1, \cdots, n_2+m_2\}$; 

\item a tuple of integers $(t_1, \cdots, t_{n_1})$ (the type), with $t_i=0$ if $f(i)>n_2$,
\item a function $s: f^{-1}(\{1, \cdots, n_2\}) \to \{0,1\}$
\end{itemize}
we define $x= x_2 \circ x_1$ as the composition of two looped diagrams $x_1$ and $x_2$ with  $x_1 \in \plDa_\Com(\oc{n_1}{m_1}, \oc{c}{m_1+n_1-c})$ and $x_2 \in \plDa_\Com(\oc{c}{m_1+n_1-c}, \oc{n_2}{m_2})$ for $c:=|f^{-1}(\{1, \cdots, n_2\})|$. 
\begin{itemize}
\item The element $x_1=x_1^1 \amalg \ldots \amalg x_1^{n_1+m_1}$ is defined as the disjoint union of $n_1+m_1$ elements with each one incoming leaf (such that we relabeled the leaf of the $i$th element by $i$ for $i \leq n_1$ and $j_{m_1}$ for $j=i-n_2$ for $n_1 \leq i \leq m_1+n_1$ after taking disjoint union).
These elements are defined as follows:
\begin{itemize}
\item If $1 \leq i \leq n_1$ and
\begin{itemize}
\item if $f(i) \leq n_2$ then $x_1^i \in \plDa_\Com(\oc{1}{0}, \oc{1}{0})$ is given by
\[ x_1^i = \begin{cases} sh^{t_i} &\text{if } s(i)=0\\ B^{t_i} & \text{if } s(i)=1.\end{cases}\]
\item if $f(i)>n_2$ and thus $t_i=0$, then $x_1^i=\id \in \plD_\Com(\oc{1}{0}, \oc{0}{1})$, the constant diagram with the incoming leaf glued to the outgoing leaf.
\end{itemize}
\item For $n_1+1 \leq i \leq n_1+m_1$ we have
\begin{itemize}
\item If $f(i) \leq n_2$ then $x_1^i \in \plDa_\Com(\oc{0}{1}, \oc{1}{0})$ given by
\[ x_1^i = \begin{cases} sh^0 &\text{if } s(i)=0\\ B^0 & \text{if } s(i)=1.\end{cases}\]
\item If $f(i) > n_2$ then $x_1^i=\id \in \plD_\Com(\oc{0}{1}, \oc{0}{1}).$
\end{itemize}
\end{itemize}
Thus $x_1=x_1^1 \amalg \ldots \amalg x_1^{n_1+m_1} \in \plDa(\oc{n_1}{m_1}, \oc{c}{m_1+n_1-c})$.
\item The element $x_2 \in \plDa(\oc{c}{m_1+n_1-c}, \oc{n_2}{m_2})$ multiplies the $i$--th incoming vertex onto the outgoing vertex $f(i)$ (and the incoming leaf onto $f(i)$ for $i>n_2$). More precisely, 
\[x_2=m_{\{f^{-1}(1)\}} \amalg \ldots \amalg  m_{\{f^{-1}(n_2)\}} \amalg \overline{m}_{\{f^{-1}(n_2+1)\}} \amalg \ldots \amalg \overline{m}_{\{f^{-1}(n_2+m_2)\}}. \]
\end{itemize}
We define $\tiplDa(\oc{n_1}{m_1}, \oc{n_2}{m_2})$ analogous to before by taking products over the type of these elements.
\end{Def} 
An example of such an element in $\tplDa(\oc{2}{2}, \oc{1}{1})$ with $t_1=1$, $t_2=2$, $f(1)=f(2)=f(4)=1$ and $f(3)=2$, $s(1)=s(3)=1$ and $s(2)=0$ is given in Figure \ref{fig:comx}. We have $x_1 \in \plDa_\Com(\oc{2}{2}, \oc{3}{1})$ and $x_2 \in \plDa_\Com(\oc{3}{1}, \oc{1}{1})$. 
\begin{figure}[ht!]  \begin{center}
\subfigure
       [The  element $x_1$]
        {
          
            \input{pic_comx1}
           
        }
     \hspace{1.5cm}
        \subfigure
       [The  element $x_2$]
        {
          
            \input{pic_comx2}\hspace{1cm}
            
        }%
        \vspace{0.5cm}
        
 \subfigure
        [The  element $x=x_2 \circ x_1$]
        {
          
            \input{pic_comxboth}
        }
%
    \end{center}
   \caption[An element in $\tplDa$]{An element in $\tplDa(\oc{2}{2}, \oc{1}{1})$}
   \label{fig:comx}
\end{figure}

Since both elements $x_1$ and $x_2$ have trivial differential, the differential on the complexes $\tplDa(\oc{n_1}{m_1}, \oc{n_2}{m_2})$ and $\tiplDa(\oc{n_1}{m_1}, \oc{n_2}{m_2})$ is also trivial.
Note moreover, that $\iplDa_\Com(\oc{1}{0}, \oc{1}{0}) \cong \tiplDa(\oc{1}{0}, \oc{1}{0})$.

We denote the further restriction of $J_\Com$ to $\tiplDa(\oc{n_1}{m_1}, \oc{n_2}{m_2})$ by $\widetilde{J}_\Com$. In \cite[Section 2.3]{kla13} we recalled Loday's lambda and shuffle operations $\lambda^k$ and $sh^k$ (cf. \cite{loda89}) and defined operations $B^k$ as the composition of the shuffle operations with Connes' boundary operator and used the families $sh^k$ and $B^k$ to build general operations in $\Nat_\Com(\oc{n_1}{m_1}, \oc{n_2}{m_2})$. Up to sign, the shuffle operations act on the Hochschild degree $n$ by taking all $(p_1, \cdots, p_k)$-shuffles in $\Sigma_{n}$ with $p_1+ \cdots p_k=n$ and all $p_j \geq 1$ and view these as elements in $\Com(n+1, n+1)$ leaving the first element fixed. Using this combinatorial description and recalling from Remark \ref{rem:optype} how to read off operations of type $k$, it is not hard to see that $\widetilde{J}_\Com$ sends the looped diagram $sh^k$ to the operation $sh^k \in \plDa(\oc{1}{0}, \oc{1}{0})$. Both, as diagrams and as operations we can rewrite the family $\lambda^k$ in terms of the $sh^k$ with the same coefficients occurring and hence see that $\widetilde{J}_\Com$ also sends the diagrams $\lambda^k$ to the corresponding operations $\lambda^k \in \plDa(\oc{1}{0}, \oc{1}{0})$. Hence this way we recover Loday's lambda operations. Since we have already seen that Connes' boundary operator is send to $B$ under $\widetilde{J}_\Com$, the same follows for the $B^k$.  Moreover, the looped diagrams $x_1$ and $x_2$ constructed in Definition \ref{def:tcom} are send to the operations $x_1$ and $x_2$ in Definition \cite[Def. 3.3]{kla13}. Therefore the complex spanned by the diagrams of type $(t_1, \ldots, t_{n_1})$ in $\tiplDa(\oc{n_1}{m_1}, \oc{n_2}{m_2})$ is mapped to the complex $A_{t_1, \cdots, t_{n_1}}$ as defined in \cite[Def. 3.3]{kla13} and thus $\widetilde{J}_\Com(\tiplDa(\oc{n_1}{m_1}, \oc{n_2}{m_2})) \cong \prod A_{t_1, \ldots, t_{n_1}}$. In \cite[Theorem 3.4]{kla13} we prove that the inclusion of this complex into $\Nat_\Com(\oc{n_1}{m_1}, \oc{n_2}{m_2})$ is a quasi-isomorphism, thus in terms of looped diagrams the theorem can be restated as:
\begin{Theorem}[{\cite[Theorem 3.4]{kla13}}]\label{th:comiso}
Let $\k$ be a field. Then  the map $\widetilde{J}_\Com$ is a quasi-isomorphism, i.e.
\[\widetilde{J}_\Com:\tiplDa(\oc{n_1}{m_1}, \oc{n_2}{m_2}) \xrightarrow{\simeq} \Nat_\Com(\oc{n_1}{m_1}, \oc{n_2}{m_2})\]
and the left complex has trivial differential, thus on homology
\[H_*(\widetilde{J}_\Com):\tiplDa(\oc{n_1}{m_1}, \oc{n_2}{m_2}) \xrightarrow{\cong} H_*(\Nat_\Com(\oc{n_1}{m_1}, \oc{n_2}{m_2})).\]
\end{Theorem}
\begin{Cor}\label{cor:comemb}
Working with coefficients in a field $\k$, the inclusion $\tiplDa(\oc{n_1}{m_1}, \oc{n_2}{m_2}) \to \iplDa_{\Com}(\oc{n_1}{m_1}, \oc{n_2}{m_2})$ is a quasi-isomorphism.
\end{Cor}
\begin{proof}
Since the differential does not contract a loop going around a whole vertex and all loops are irreducible and of this kind and there is an isomorphism of chain complexes
$\plDa_{\Com}^{t_1, \cdots, t_{n_1}}(\oc{n_1}{m_1}, \oc{n_2}{m_2}) \cong \plDa_{\Com} (\oc{0}{m_1+n_1}, \oc{n_2}{m_2})$ and similarly by the restriction of this isomorphism we have $\widetilde{pl\mathcal{D}}_{\Com}^{t_1, \cdots, t_{n_1}}(\oc{n_1}{m_1}, \oc{n_2}{m_2}) \cong \widetilde{pl\mathcal{D}}_{\Com} (\oc{0}{m_1+n_1}, \oc{n_2}{m_2})$. Moreover, by Lemma \ref{le:comnat} and the following constructions $\Nat_\Com(\oc{0}{m_1+n_1}, \oc{n_2}{m_2}) \cong  \plDa_{\Com} (\oc{0}{m_1+n_1}, \oc{n_2}{m_2})$ and the map $\widetilde{J}_\Com$ in this case is given by the embedding of $\widetilde{pl\mathcal{D}}_{\Com} (\oc{0}{m_1+n_1}, \oc{n_2}{m_2})$ into this complex. Since $\widetilde{J}_\Com$ is a quasi-isomorphism, the embedding 
\[\widetilde{pl\mathcal{D}}_{\Com} (\oc{0}{m_1+n_1}, \oc{n_2}{m_2}) \to \plDa_{\Com} (\oc{0}{m_1+n_1}, \oc{n_2}{m_2})\]
is a quasi-isomorphism and thus by the isomorphism of complexes stated above, so is
\[ \widetilde{pl\mathcal{D}}_{\Com}^{t_1, \cdots, t_{n_1}}(\oc{n_1}{m_1}, \oc{n_2}{m_2}) \to \plDa_{\Com}^{t_1, \cdots, t_{n_1}}(\oc{n_1}{m_1}, \oc{n_2}{m_2}).\] Since homology commutes with products, the map
\[\prod_{(t_1, \cdots, t_{n_1})} \widetilde{pl\mathcal{D}}_{\Com}^{t_1, \cdots, t_{n_1}}(\oc{n_1}{m_1}, \oc{n_2}{m_2})  \to \prod_{(t_1, \cdots, t_{n_1})}\plDa_{\Com}^{t_1, \cdots, t_{n_1}}(\oc{n_1}{m_1}, \oc{n_2}{m_2}) \]
is an isomorphism on homology and thus the corollary is proven.
\end{proof}

\begin{Cor}
The map $J_\Com: \iplDa_\Com \to \Nat_\Com$ is a quasi-isomorphism of dg-categories.
\end{Cor}
\begin{proof}
We have seen earlier that $J_\Com$ is a dg-functor. Thus we only need to show that it is a quasi-isomorphism. By Theorem \ref{th:comiso} and Corollary \ref{cor:comemb} the maps $\widetilde{J}_\Com$ and $i_\Com$ are quasi-isomorphisms. Furthermore, $J_\Com \circ i_\Com= \widetilde{J}_\Com$ and hence it follows that $J_\Com$ is a quasi-isomorphism for all morphism spaces.
\end{proof}

\section{The suspended cacti operad and its action}\label{sec:cacti}
In this section we define a subcategory $\plD_{cact}(n_1, n_2)$ of $\plD_+(\oc{n_1}{0}, \oc{n_2}{0})$ and show that it is quasi-isomorphic to a suspension of the cacti quasi-operad. We start with operadic constructions and definitions.

\subsection{Operadic constructions}
First we give some operadic tools. 
We only consider non-unital operads, i.e. operads indexed on the positive integral numbers. Furthermore, sometimes we will have to work with quasi-operads. A quasi-operad fulfills the same axioms as an operad beside associativity (for more details see \cite[Section 1]{kauf05}).

We work with (quasi-)operads in chain complexes, topological spaces and pointed topological spaces.

To switch between these, we need the following constructions:
\begin{Prop}
\begin{itemize}
\item For $\OP$ an operad in topological spaces with all structure maps proper, the level-wise one-point compactification $\OP^c$ is an operad in pointed spaces (cf. \cite[Prop. 4.1]{aron13}).
\item Let $\OP$ be an operad in topological spaces and $I$ an operadic ideal. Define $\OP/I$ in level $n$ to be the pointed space $\OP(n)/I(n)$ with basepoint $I(n)$. Then $\OP/I$ is an operad in pointed spaces.
\end{itemize}
\end{Prop}
Next we define some operads used later on:
\begin{Def}
The open simplex operad $\D$ is the topological operad with
$\D(n)=\mathring{\Delta}^{n-1}=\{(s_1, \cdots, s_n) \ | \ 0 < s_i < 1, \sum s_i=1\}$ the open standard $n-1$--simplex, with $\Sigma_n$--action given by permuting the coordinates and composition defined by
\begin{align*}
\circ_i: \mathring{\Delta}^{k-1} \times \mathring{\Delta}^{n-1} &\to \mathring{\Delta}^{n+k-2}\\ (t_1, \cdots t_k)
\circ_i (s_1, \cdots, s_n) &=(s_1, \cdots, s_{i-1}, s_i \cdot t_1, \cdots, s_i \cdot t_k, s_{i+1}, \cdots, s_n)
.\end{align*}

The sphere operad $\Sph$ is the one-point compactification of $\D$, i.e. $\Sph(n)=\D(n)^c$.

\end{Def}
Note that the operad $\D$ is homeomorphic to the scaling operad $\OR_{>0}$ defined in \cite[5.1.1]{kauf05} and given by $\OR_{>0}(n) =\R_{>0}^{n}$. The homeomorphism from $\OR_{>0}$ to $\D$  sends a tuple $(r_1, \ldots, r_n) \in \R_{>0}^{n}$ with $R=\sum r_i$ to $(\frac{r_1}{R}, \ldots, \frac{r_n}{R}) \in \mathring{\Delta}^{n-1}$. 

The operads $\D$ and $\Sph$ have also been defined in \cite{aron13} where they are denoted by $\Delta_1^{n-1}$ and $S$, respectively. There, it is also mentioned that one can use the sphere operad to define operadic suspension. To make this more precise, we recall:
\begin{Def}[{\cite[Section 7.2.2]{loda12}}]\label{def:susp}
The desuspension operad $\des$ is defined as the endomorphism operad of $s^{-1}\k$ with $s^{-1}$ the shift operator, thus
\[\des(n)=\hom((s^{-1}\k)^{\otimes n}, s^{-1}\k).\]

For a graded operad $\OO$ its operadic desuspension is defined as
\[s^{-1}\OO(n)=\des(n) \otimes \OO(n),\]
as the Hadamard tensor product of operads (cf. \cite[Section 5.3.3]{loda12}).
We define the twisted desuspension operad $\dest=\widetilde{H}_*(\Sph)$ the reduced homology of the sphere operad. This operad equals $\des(n)$ as a graded $\Sigma_n$--module, but the signs in the composition differ. We define the twisted operadic desuspension by
\[\widetilde{s}^{-1}\OO(n)=\dest(n) \otimes \OO(n).\]
\end{Def}
Defining the topological operadic desuspension of an operad $\OP$ in pointed spaces as $\Sph \wedge \OP$, on reduced homology one obtains $\widetilde{H}_*(\Sph \wedge \OP) \cong \widetilde{s}^{-1}\widetilde{H}_*(\OP)$.

\subsection{The cacti-like diagrams}%
In this section we define the different kinds of looped and Sullivan diagrams used later on. 

\begin{Def}\label{def:cactdiag}
An $\oc{n_2}{(0,n_1)}$--looped diagram $(\Gamma, \gamma_1, \cdots, \gamma_{n_1})$ belongs to $\plDa_{cact}(n_1,n_2)$ if it fulfills the following properties:
\begin{itemize}
\item The white vertices in $\Gamma$ are not connected, i.e. $\Gamma$ is the disjoint union of $n_2$ commutative diagrams.
\item The underlying commutative Sullivan diagram has a representation embeddable into the plane.
\item Every boundary segment of any white vertex in $\Gamma$ is part of exactly one loop $\gamma_i$ and all these loops $\gamma_i$ are irreducible and positively oriented.
\item If an arc component  (i.e. the connected components after removing the white vertex) has genus $g$ (as a graph) there are exactly $g$ constant loops attached to it.
\end{itemize} 
\end{Def}
This defines a complex, since by the irreducibility of the loops around the white vertex, genus in the graph can only be created if a loop is contracted completely and thus there is a constant loop belonging to the new genus.

Denote by $\plDc_{cact}(n_1, n_2)$ the subcomplex of partly constant diagrams (i.e. at least one of the loops is constant) which in particular contains all diagrams with genus in the arc components by the last condition in the definition. Let $\plD_{cact}(n_1, n_2)$ be the subcategory of $\plD(\oc{n_1}{0}, \oc{n_2}{0})$ given by the non-constant diagrams as introduced in Section \ref{sec:nonconst}. Since the white vertices are not connected, all singular vertices must be loop covering and hence by Corollary \ref{cor:loopcov} we have an isomorphism of dg-categories
\[\plD_{cact}(n_1, n_2)\cong \plDa_{cact}(n_1, n_2) / \plDc_{cact}(n_1, n_2).\]

%

%
 \begin{figure}[ht!]
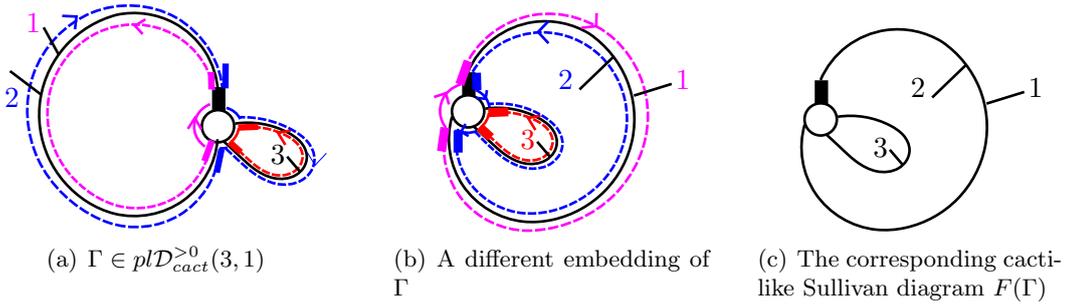

     \begin{center}
\subfigure
       [$\Gamma \in \plD_{cact}(3, 1)$]
        {
          
            \input{pic_cacti1}
	\label{fig:cacti}
           
        }
     \hspace{0.5cm}
\subfigure
       [A different embedding of $\Gamma$]
        {
          
            \input{pic_cacti2}
	\label{fig:cacti2}
           
        }
             \hspace{0.5cm}   
 \subfigure
       [The corresponding cacti-like Sullivan diagram $F(\Gamma)$]
        {
          
            \input{pic_cacti3}
	\label{fig:cacti3}
           
        }
     \hspace{0.5cm}
    \end{center}
   \caption{A looped diagram in $\plD_{cact}(3, 1)$, the embedding used to obtain a Sullivan diagram and the corresponding cacti-like Sullivan diagram}
   \label{fig:cactall}
\end{figure}

By the description every diagram is a disjoint union of looped diagrams with at least one incoming leaf and hence
\[\plDa_{cact}(n_1, n_2)= \coprod_{\substack{f: \{1, \ldots, n_1\} \to \{1, \cdots, n_2\}\\f \text { surj, } t_i:=|f^{-1}(i)|}} \plDa_{cact}(t_1, 1)  \times \cdots \times \plDa_{cact}(t_{n_2}, 1)\]
and similarly for $\plDc_{cact}(n_1, n_2)$  and $\plD_{cact}(n_1, n_2)$. 
Furthermore, the chain complexes $\plDa_{cact}(n, 1)$, $\plDc_{cact}(n, 1)$  and $\plD_{cact}(n, 1)$  define non-unital dg-operads. 
An example of an element in  $\plD_{cact}(3, 1)$ is given in Figure \ref{fig:cacti}.

Before we give an actual definition of the cacti operad, we define the complex of cacti-like $\oc{1}{n}$--Sullivan diagrams:
\begin{Def}
A Sullivan diagram $\Gamma \in \oc{1}{n}-\text{Sullivan diagrams}$ is \emph{cacti-like} if
\begin{itemize}
\item it is embeddable into the plane,
\item the arc components have genus zero as graphs (i.e. viewed as CW-complexes they are contractible),
\item the diagram has exactly $n$ boundary cycles, each labeled by a leaf.
\end{itemize}
\end{Def}

In a first step, we show that on the level of diagrams we have a bijection:
\begin{Lemma}\label{le:bijsul}
We have a bijection
\[K:\{\text{Cacti-like } \ \oc{1}{n}-\text{Sullivan diagrams}\} \longrightarrow 
\left\{ \begin{array}{l}
\text{non-constant looped diagrams} \\(\Gamma, \gamma_1, \cdots, \gamma_n) \in \plDa_{cact}(n,1)
\end{array}\right\}
\]
with $K$ the map described at the end of Section \ref{sec:comsul} forgetting about the cyclic ordering at the black vertices though remembering a loop for each boundary cycle.
\end{Lemma}
\begin{proof} The map $K$ has an inverse $F$ which can be described as follows: After forgetting the labeled leaves, for any looped diagram $(\Gamma, \gamma_1, \cdots, \gamma_{n}) \in \plDa_{cact}(n, 1)$ with no constant loops the commutative Sullivan diagram $\Gamma$ can be uniquely (up to the equivalence relation on Sullivan diagrams) embedded into the plane such that the last segment of the white vertex is on the outside of the graph. To see so, one verifies that the data of a commutative Sullivan diagram of degree $d$ without genus in the arc components is equivalent to dividing $\{0, \cdots, d\}$ into a disjoint union of subsets and connecting all elements of one subset by an arc component. This becomes unique if the diagram was embeddable into the plane before and if we glue onto an interval instead of a circle (which we do by the condition that the last boundary segment lies outside the graph). To define the inverse map $F$, for a  looped diagram $(\Gamma, \gamma_1, \cdots, \gamma_{n}) \in \plDa_{cact}(n, 1)$ we then choose the labeled leaves  to be inside the loop which starts at this leaf (this is possible since the loops do not overlap). It is not hard to see that after forgetting the loops this embedding into the plane gives a cacti-like $\oc{1}{n}$--Sullivan diagram. An example of the chosen embedding of the diagram in Figure \ref{fig:cacti} and the corresponding cacti-like Sullivan diagram are shown in Figure \ref{fig:cactall}.
By definition, the composition $F \circ K$ is the identity. In order to see that $K \circ F$ is the identity too, one checks that every  cacti-like $\oc{1}{n}$--Sullivan diagram can be embedded into the plane such that the last segment of the white vertex is on the outside of the diagram. Choosing this embedding and using the fact that $F$ is well-defined, i.e. independent of the embedding, it follows directly that $K \circ F$ is the identity.
\end{proof}

\subsection{The cacti operad}

In this section we recall the definition of the (normalized) Cacti operad with spines. The original definition goes back to Voronov in \cite{voro05}. For an overview over different definitions of cacti see \cite{kauf05}. We use the definition given in \cite[Section 2.2]{cohe05}.

An element in  $Cacti(n)$ is given by a treelike configuration of $n$ labeled circles with positive circumferences $c_i$ such that $\sum c_i =1$ (usually one uses the radii, but for our setup working with the circumferences immediately is easier) together with the following data: (1) A cyclic ordering at each intersection point, (2) the choice of a marked point on each circle and (3) the choice of a global marked point on the whole configuration together with a choice of a circle this point lies on.  Treelike means that the dual graph of this configuration, whose vertices correspond to the lobes and which has an edge whenever two circles intersect, is a tree. In the normalized cacti $Cacti^1(n)$ all circles have the same radius/circumference. This is only a pseudo-operad, since associativity fails (cf. \cite[Remark 2.9.19]{kauf05}).

\begin{figure}[!ht]

 \center
\input{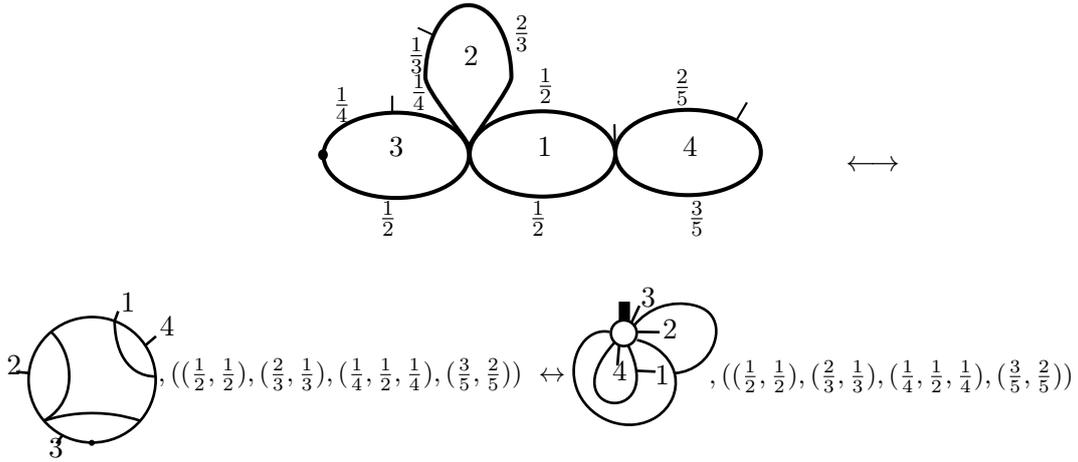}

	\caption{Three different representations of the same element in $Cacti^1(4)$}

		\label{fig:cactireal}
\end{figure}

In \cite{lupe08} it is shown that $Cact$ can be equivalently defined using the space of chord diagrams. To avoid even more notation, we apply the equivalence of Sullivan chord diagrams and Sullivan diagrams as defined in Section \ref{sec:graphs} and  give a description of $Cacti^1$ and $Cacti$ in these terms:
$Cact^1(n)$ is a finite CW complex with cells $\Delta^{f_\Gamma(1)-1} \times \cdots \times \Delta^{f_\Gamma(n)-1}$ for each $\Gamma$ a cacti like $\oc{1}{n}$--Sullivan diagram whose boundary cycle belonging to the $i$--th labeled incoming leaf consists of $f_\Gamma(i)$ pieces. Thus the $t_1^j, \ldots, t_{f_\Gamma(j)}^j$ give the lengths of the pieces of a lobe starting at the spine (marked point) of the lobe. 
Hence we have an isomorphism 
\[Cacti^1(n) \cong \bigcup_{\Gamma \in \oc{1}{n}-\text{Sullivan diagrams, cacti- like}} \Delta^{f_\Gamma(1)-1} \times \cdots \times  \Delta^{f_\Gamma(n)-1}\]
where $\Delta^m$ is the standard simplex.
The correspondence is shown in Figure \ref{fig:cactireal}, where we have drawn an element in $Cacti^1(4)$ first as an ordinary cactus, then as a Sullivan chord diagram with the corresponding arc lengths in $\Delta^{1} \times \Delta^1 \times \Delta^2 \times \Delta^1$ (i.e. the description corresponding to \cite{lupe08}) and then in terms of the definition given in the lines above (Sullivan diagram with arc lengths in $\Delta^{1} \times \Delta^1 \times \Delta^2 \times \Delta^1$).

The attaching maps identify an element $(\Gamma, ((t_1^1, \ldots, t_{f_\Gamma(1)}^1), \ldots, (t_1^n, \ldots, t_{f_\Gamma(n)}^n)))$ with one of the $t_i^j=0$ with the pair $(\Gamma', ((t_1^1, \ldots, t_{f_\Gamma(1)}^1), \ldots, t_{i-1}^j,t_{i+1}^j, \ldots, t_{f_\Gamma(n)}^n)))$ where $t_i^j$ is omitted and $\Gamma'$ is the Sullivan diagram where the $i$--th boundary segment belonging to the boundary cycle labeled by the $j$--th leaf is contracted. In Figure \ref{fig:equivcacti} we have given an example, where in the left representative the length $t_3^3$ (so the last length belonging to the boundary cycle labeled by the leaf $3$) equals zero and in the right representative the last segment of the boundary cycle of the diagram (which in this case is the first boundary segment belonging to white vertex) got contracted.

\begin{figure}[!ht]

 \center
\input{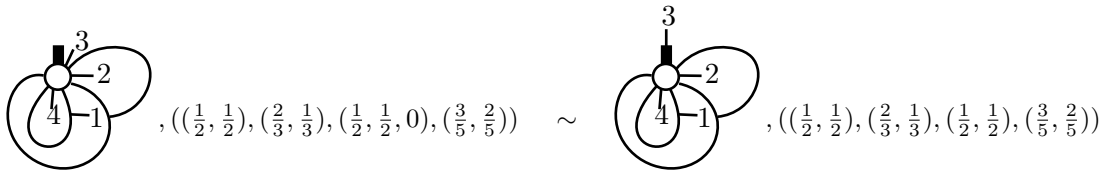}

	\caption{The gluing relation in $Cacti^1(4)$}
		\label{fig:equivcacti}
\end{figure}

As a space, we have $Cacti(n) \cong Cacti^1(n) \times \mathring{\Delta}^{n-1}$ where the extra parameters specify the lengths of the loops. So we get
\[Cacti(n) \cong \bigcup_{\Gamma \in \oc{1}{n}-\text{Sullivan diagrams,  cacti- like}}\Delta^{f_\Gamma(1)-1} \times \cdots \times  \Delta^{f_\Gamma(n)-1} \times \mathring{\Delta}^{n-1}\]
and the equivalence relation coming from the one in $Cacti^1(n)$.

By \cite[Cor. 5.2.3]{kauf05} the homology of the quasi-operad $Cacti^1$ is equivalent to the homology of $Cacti$ as a graded operad. As announced in \cite[Theorem 2.3] {voro05} (see \cite[Theorem 5.3.6]{kauf05} for a complete proof), the cacti operad is homotopy equivalent to the framed little disc operad, and by \cite{getz94} the singular homology of the framed little disc operad is isomorphic as an algebraic operad to the BV-operad. Thus,
\[H_*(Cacti^1) \cong BV\]
as an equivalence of graded operads.

\subsection{Result and proof}
Now we have given enough definitions to state the main theorem of the section:
\begin{Theorem}\label{Th:shiftbv}
The complex $\plD_{cact}(n, 1)$ is a chain model for the twisted operadic desuspended BV-operad,
i.e. \[H_*(\plD_{cact}(-,1)) \cong \widetilde{s}^{-1}BV\]
as graded operads.
\end{Theorem}
Before we prove the theorem, we first want to point out the consequence for the operations on Hochschild homology:
\begin{Cor}\label{Cor:action}
There is a desuspended BV-algebra structure on the Hochschild homology of a commutative cocommutative open Frobenius dg-algebra (in particular on the Hochschild homology of a commutative Frobenius dg-algebra) which comes from an action of a chain model of the suspended Cacti operad on the Hochschild chains. 
\end{Cor}
More precisely, the product is the product given by the action of the looped diagram shown in Figure \ref{fig:prod}. This BV-structure agrees with the one given in \cite[Section 7]{abba13} and \cite[Section 6]{abba13b}.

To avoid introducing more notation, we write $\plDa_{cact}(n,1)_\bullet$ and $\plDc_{cact}(n,1)_\bullet$ for the associated simplicial sets of the chain complexes (where we allow unlabeled incoming leaves at any point of the white vertex, cf. the proof of Lemma \ref{le:chcplx}) and take their geometric realization. We define an operad structure on this space and show that on homology we have an isomorphism of operads $H_*(\plDa_{cact}(-,1) / \plDc_{cact}(-,1))\cong H_*(\left|\plDa_{cact}(-,1)_\bullet\right| ,  \left|\plDc_{cact}(-,1))_\bullet \right|)$.

Write $X_\bullet=\plDa_{cact}(-,1)_\bullet$. We first explain why we do not need to care about degenerate simplices (i.e. diagrams with unlabeled leaves attached to the white vertex at other positions than the start half-edge). Denote the non-degenerate part by $X^{non-deg}_\bullet$. Since the boundary of a non-degenerate element in $X_\bullet$ is non-degenerate, the simplicial realization is the realization of $X^{non-deg}_\bullet$ as a semi-simplicial set, i.e.
\[|X_\bullet| \cong \coprod_k (X^{non-deg}_k \times \Delta^k) / \sim\]
with $(x, \delta^i y) \sim (d_i x, y)$.

So an element in the geometric realization is an equivalence class of a looped diagram together with an assignment of lengths to each piece of the white vertex. Thus a point in the geometric realization is a diagram where we consider the white vertex as a circle of length one up to the equivalence relation of rescaling and sliding the arc components. For composing $y$ onto the $i$--th loop of $x$ we assume that the boundary segment belonging to the $i$--th loop of $x$ have lengths $(s_1, \ldots, s_r)$. Then composition onto the $i$--th loop rescales the white vertex of $y$ to lengths $\sum s_i$ and glues the diagram into these pieces gluing the start half-edge onto the labeled leaf and considering the arc components to have lengths zero. This is similar to the gluing in the arc complex and fat graph considered in \cite{penn87}, \cite{kauf03} and others. An example of such a composition is shown in Figure \ref{fig:cacticomp}, where we glue an element in $|\plDa_{cact}(2,1)_\bullet|$ onto the second loop of an element in $|\plDa_{cact}(3,1)_\bullet|$ and thus obtain an element in $|\plDa_{cact}(4,1)_\bullet|$. We have written the lengths of the boundary segments of the white vertices next to the vertex in the picture. To stay with the notation we use in operads we write $y \circ_i x$ if we glue $y$ onto the $i$--th loop of $x$ (i.e. we switch the order against the ordinary gluing in the chain complex of looped diagrams).

\begin{figure}[!ht]

 \center
\input{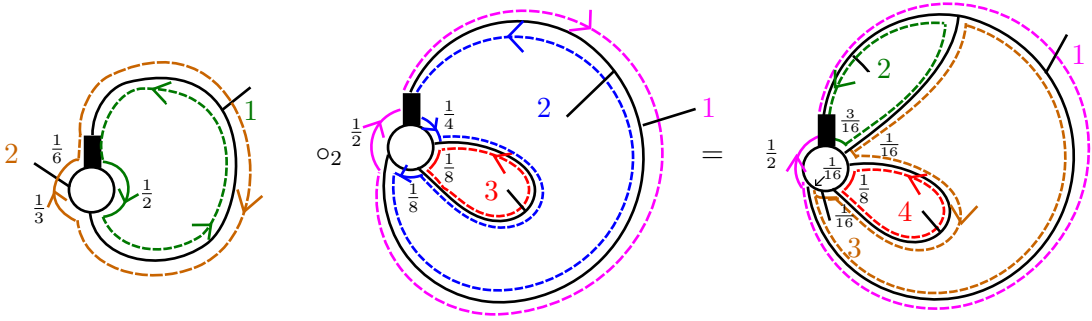}
	\caption{An example of the composition in $|\plDa_{cact}(-,1)_\bullet|$, gluing the left element onto the second (blue) loop of the right one}
		\label{fig:cacticomp}
\end{figure}

\begin{Lemma}
With the above operad structure on $\plDa_{cact}(-,1)$, we obtain an isomorphism of operads
\[H_*(\plDa_{cact}(-,1) / \plDc_{cact}(-,1))\cong H_*(\left|\plDa_{cact}(-,1)_\bullet\right| ,  \left|\plDc_{cact}(-,1))_\bullet \right|) .\]
\end{Lemma}
\begin{proof}
Given a simplicial set $X_\bullet$, we have an isomorphism of chain complexes $\overline{C}_*(X_\bullet) \cong C^{CW}_*(|X_\bullet|)$ where $\overline{C}_*$ denotes the normalized chains whereas $C^{CW}_*$ stands for the cellular chains of a complex. In our situation this induces isomorphisms $\plDa_{cact}(n, 1) \cong C^{CW}_*(|\plDa_{cact}(n, 1)_\bullet|)$ and $\plDc_{cact}(n, 1) \cong C^{CW}_*(|\plDc_{cact}(n, 1)_\bullet|)$. Since the operadic composition on the geometric realization is a cellular map we get an induced composition on the cellular complexes. Hence we only have to check that on homology the composition commutes with the isomorphisms.

A diagram $x \in \plDa_{cact}(n, 1)$ of degree $k$ is mapped to the cell given by the $x \times \Delta^k$. Given a cell $x \times \Delta^k$ and a cell $y \times \Delta^l$ with the $i$--th loop of $y$ non-constant, gluing $x \times \Delta^k$ onto the $i$--th loop of $y \times \Delta^l$ we obtain the union of the $\Delta^{k+l}$ cells which one gets by gluing the diagram $y$ onto the diagram $x$ in all possible ways. The signs come from the orientation of the cells.
By the definition the complex $x \in \plDc_{cact}(n, 1)$ gets mapped to the cells belonging to the operad $|\plDc_{cact}(-,1))|$. Since both are operadic ideals, we get the wished isomorphism on homology.
\end{proof}

%
%

\begin{Lemma}\label{le:cact}
We have a homeomorphism of operads in pointed spaces
\[ \left| \plDa_{cact}(-,1)_\bullet\right| / \left| \plDc_{cact}(-,1)_\bullet \right| \cong Cacti^c\]
where $Cacti^c$ is the one-point compactification of the cacti operad.
\end{Lemma}
We do not see an easy proof of the fact that the structure maps in $Cacti$ are proper and hence that the structure maps in $Cacti^c$ are continuous. However, proving the homeomorphism in the above statement and checking that it preserves the obvious structure maps then implies the continuity of these maps.
\begin{proof}
Note that $\left| \plDc_{cact}(n,1)_\bullet \right|$ is a closed subspace of the compact space $\left| \plDa_{cact}(n,1)_\bullet\right|$ (which is the realization of a finite simplicial set). It is a point set topological exercise that when given a compact space $X$ with a closed subspace $X'$, the one-point compactification of $X \backslash X'$ (the complement of $X'$ in $X$) is homeomorphic to $X/X'$. Hence proving a homeomorphism 
\[\left| \plDa_{cact}(n,1)_\bullet\right| \backslash \left| \plDc_{cact}(n,1)_\bullet\right| \cong Cacti\]
induces a homeomorphism 
\[ \left| \plDa_{cact}(-,1)_\bullet\right| / \left| \plDc_{cact}(-,1)_\bullet \right| \cong Cacti^c\] 
which sends $ \left| \plDc_{cact}(-,1)_\bullet \right|$ to the point at infinity of $Cacti^c$.

Take $(x, y)$ with $x=(\Gamma, \gamma_1, \cdots, \gamma_{n}) \in \plDa_{cact}(-,1)_k$ and $y=(s_0, \cdots, s_k) \in \Delta^n$ such that $[(x,y)] \in \left| \plDa_{cact}(n,1)_\bullet\right|\backslash \left| \plDc_{cact}(n,1)_\bullet\right|$. This is equivalent to assuming that $x$ is non-constant and that in $y$ not all of the $s_i$ belonging to the same loop in $x$ are zero. More precisely, for $F(x)$ the corresponding cacti-like Sullivan diagram defined in Lemma \ref{le:bijsul} and $f_\Gamma$ the map which counts how many of the boundary components of the white vertex belong to each boundary cycle, we reorder and relabel $(s_0, \cdots, s_k)$ to $(t^1_1, \ldots, t^1_{f_{F(\Gamma)}(1)}, t_1^2, \ldots, t^n_{f_{F(\Gamma)}(n)})$. Then the assumption on $y$ is equivalent to $R_i=\sum_j t_j^i \neq 0$ for all $i$. 

We send the pair $(x,y)$ to 
\[ ((\frac{t^1_1}{R_1}, \ldots, \frac{t^1_{f_{F(\Gamma)}(1)}}{R_1}), \ldots, (\frac{t^n_1}{R_n}, \ldots, \frac{t^n_{f_{F(\Gamma)}(n)}}{R_n}), (R_1, \ldots, R_n)) \in \Delta^{f_\Gamma(1)-1} \times \cdots \times \Delta^{f_\Gamma(n)-1} \times \mathring{\Delta}^{n-1}\]
lying in $Cacti(n)$ in the ``cell'' belonging to $F(\Gamma)$. Away from the boundary the coordinates of the simplex (i.e. the $s_i$) just give us the lengths of the pieces of the arc of a cacti.

The map is well-defined, since the equivalence relation given by the geometric realization (contracting a piece of the boundary is equivalent to setting the corresponding $t_i^j$ zero) is the same equivalence relation as we have on cacti.

We next construct an inverse.
Given 
\[( (t^1_1, \cdots, t^1_{f_\Gamma(1)} , t^2_1, \ldots, t^n_{f_\Gamma(n)}), (R_1, \ldots, R_n)) \in  \Delta^{f_\Gamma(1)-1} \times \cdots \times \Delta^{f_\Gamma(n)-1} \times \mathring{\Delta}^{n-1}\]
in the ``cell'' corresponding to a cacti-like Sullivan diagram $\Gamma$, 
it is mapped to
\[[(K(\Gamma), (R_1 \cdot t^1_1, \cdots, R_1 \cdot t^1_{f_\Gamma(1)}, R_2 \cdot t^2_1, \ldots,  R_n \cdot t^n_{f_\Gamma(n)}))] \in \left| \plDa_{cact}(n,1)_\bullet\right|. \]
It is not hard to check that this map is an inverse to the first one.
Moreover, both maps are continuous, hence we have constructed the asked homeomorphism.

Because of the way we defined the maps, it is also clear that composition is preserved on 
$\left| \plDa_{cact}(n,1)_\bullet\right| \backslash \left| \plDc_{cact}(n,1)_\bullet\right|$. Composition on $\left| \plDa_{cact}(n,1)_\bullet\right| / \left| \plDc_{cact}(n,1)_\bullet\right| $ sends everything containing the basepoint to the basepoint and hence agrees with the composition of the one-point compactification.

\end{proof}
The next step of the proof goes along the lines of \cite[Section 5]{kauf05}.
We first need to define two further operads, one given by a semi-direct product and the analog given by the semi-direct smash product.

Recall the simplex operad $\D$ with $\D(n)=\mathring{\Delta}^{n-1}$. Let $\D \ltimes Cacti^1$ be the operad with $(\D \ltimes Cacti^1)(n)=\D(n) \times Cacti^1(n)$ with diagonal $\Sigma_n$--action and the composition which for $(d,c) \in \D(n) \times Cacti^1(n)$ and $(d',c') \in \D(k) \times Cacti^1(k)$ is defined by
\[(d,c) \circ_i (d', c')=(d \circ_i d', c \circ_i^{d'} c')\]
with $c \circ_i^{d'} c'$ computed via the following procedure: Write $d'=(d_1, \cdots, d_k)$ and rescale $c'$ by $d'$, i.e. scale the $j$--th lobe by $d_j$. Then we glue $c'$ into the $i$--th lobe of $c$ and rescale all lobes back to length $1$. A more general theory of the semi-direct product of operads can be found in \cite[Section 1.3]{kauf05}.

Similarly, we define the operad $\Sph\  \overset{\ltimes}{\wedge} \  Cacti^1_+$ with pointed spaces $(\Sph \ \overset{\ltimes}{\wedge} \ Cacti^1_+)(n)=\Sph(n) \wedge Cacti^1_+(n)$, where $Cacti^1_+(n)$ is the space $Cacti^1(n)$ with added disjoint basepoint. The composition is defined by
\[(d \wedge c) \circ_i (d'\wedge c')=(d \circ_i d')\wedge (c \ \widetilde{\circ_i^{d'}} \ c')\]
with
\[ c \  \widetilde{\circ_i^{d'}}  \ c'=\begin{cases} \ast & \text{if any of } c, c' \text{ or } d \text{ is the base point}\\c \circ_i^{d'} c' & \text{else.}\end{cases}\]

On the level of spaces we have
\[(\D(n) \times Cacti^1(n))^c \cong \D(n)^c \wedge Cacti^1(n)^c \cong \Sph(n) \wedge Cacti^1_+(n)\]
since $Cacti^1(n)$ is a finite CW-complex and hence compact, i.e. its one-point compactification just adds a disjoint basepoint. 
We defined the operad structures exactly such that
\[(\D \ltimes Cacti^1)^c \cong \Sph \ \overset{\ltimes}{\wedge} \ Cacti^1_+.\]

Rewriting the first part of \cite[Theorem 5.2.2]{kauf05} in terms of the simplex operad, we have:
\begin{Lemma}
There is a homeomorphism of operads 
\[Cacti \cong \D \ltimes Cacti^1\]
and hence
\[Cacti^c \cong (\D \ltimes Cacti^1)^c \cong \Sph \ \overset{\ltimes}{\wedge} \ Cacti^1_+.\]
\end{Lemma}
As the last step, we need the analog of the second part \cite[Theorem 5.2.2]{kauf05}, which is the analog of the comparison between loops and Moore loops:
\begin{Lemma}
There is a homotopy equivalence of pointed quasi-operads 
\[\Sph \ \overset{\ltimes}{\wedge} \ Cacti^1_+ \simeq \Sph \wedge Cacti^1_+.\]
\end{Lemma}
\begin{proof}
Completely similar to \cite{kauf05}, the perturbed and unperturbed multiplications are homotopic. Choosing a line segment from $d' \in \Delta^{k-1}$ to the midpoint of the simplex $(1/k, \ldots, 1/k)$ and denoting the corresponding path $d'_t$, the homotopy for the pointed quasi-operad composition is defined by
\[(d,c) \circ_i (d', c') = (d \circ_i d', c  \ \widetilde{\circ_i^{d'_t}} \  c')\]
where we rescale $c'$ by $d'_t$. 
\end{proof}

 Now we are able to prove the main proposition of the section:
\begin{proof}[Proof of Theorem \ref{Th:shiftbv}]
Collecting all the homeomorphisms and homotopy equivalences of (quasi-)operads shown in this section, we get the following isomorphism of operads,
\begin{align*}
H_*(\plD_{cact}(-, 1)) &\cong H_*(\plDa_{cact}(-,1) / \plDc_{cact}(-,1))\\
&\cong H_*(\left|\plDa_{cact}(-,1)_\bullet\right| ,  \left|\plDc_{cact}(-,1))_\bullet \right|) \\
&\cong \widetilde{H}_*(\left|\plDa_{cact}(-,1)_\bullet\right| / \left|\plDc_{cact}(-,1))_\bullet \right| )\\ 
&\cong \widetilde{H}_*( Cacti^c)\\ 
&\cong \widetilde{H}_* (\Sph \ \overset{\ltimes}{\wedge} \ Cacti^1_+)\\
&\cong \widetilde{H}_* (\Sph \wedge Cacti^1_+)\\
& \cong \dest \otimes H_*(Cacti)
\end{align*}
where the last step is the Kuenneth morphism $ \widetilde{H}_*(X) \otimes \widetilde{H}_*(Y) \to \widetilde{H}_*(X \wedge Y) $ which is an isomorphism since $\widetilde{H}_*(X)$ is free in the case considered here. Moreover, the Kuenneth morphism is a symmetric monoidal functor and thus preserves the operad structure.

Since $H_*(Cacti)$ is the BV operad, the isomorphism of operads 
\[H_*(\plD_{cact}(-, 1)) \cong \dest \otimes BV\]
follows.
\end{proof}
\begin{appendix}
\section{An overview over the complexes of looped diagrams}\label{sec:overview}

In the two tables below we give an overview over all the complexes defined in the paper. Table \ref{tab:defs} gives the basic definitions of the complexes and Table \ref{tab:defs2} applies constructions to complexes $\cC$ from the first table.
\begin{table}[ht]
\begin{center}
    \begin{tabular}{ | p{3cm} |p{2.5cm}| p{6cm}| p{1.5cm} |}
    \hline
    $\cC$ & Name (if given) &Description of generators $x=(\Gamma, \gamma_1, \cdots, \gamma_{n_1}) \in \cC$& Place defined in the paper \\ \midrule[2pt]
  \vspace{0.5pt}  $\lDa(\oc{n_1}{m_1},\oc{n_2}{m_2})$& \emph{looped diagrams} &$\Gamma$ a commutative Sullivan diagram with $n_1+m_1+m_2$ labeled leaves and $n_1$ loops starting at the first $n_1$--labeled leaves &Def. \ref{def:lD}\\ \hline \vspace{0.5pt}
     $\lDa_+(\oc{n_1}{m_1},\oc{n_2}{m_2})$& \emph{looped diagrams with positive boundary condition} & $x \in \lDa(\oc{n_1}{m_1},\oc{n_2}{m_2})$ such that every connected component contains at least one white vertex or one of the $m_2$ last labeled leaves&Def. \ref{def:posbdry}\\ \hline   
 \vspace{0.5pt}    $\plDa(\oc{n_1}{m_1},\oc{n_2}{m_2})$, $\plDa_+(\oc{n_1}{m_1},\oc{n_2}{m_2})$ & \emph{positively oriented looped diagrams (with positive boundary condition)} & $x \in \lDa(\oc{n_1}{m_1},\oc{n_2}{m_2})$ (or $\lDa_+(\oc{n_1}{m_1},\oc{n_2}{m_2})$, resp.)  such that all loops are positively oriented& Def. \ref{def:plD}\\ \hline      \vspace{0.5pt} $\plDa_{start}(\oc{n_1}{m_1},\oc{n_2}{m_2})$&  & $x \in \plDa(\oc{n_1}{m_1},\oc{n_2}{m_2})$ such that each loop $\gamma_i$ consists of exactly one boundary segment of a white vertex which is the first boundary segment of that white vertex& Def. \ref{def:start}\\ \hline  \vspace{0.5pt}
$\plDa_\Com(\oc{n_1}{m_1}, \oc{n_2}{m_2})$& &$x \in \plDa_+(\oc{n_1}{m_1},\oc{n_2}{m_2})$ such that $\Gamma$ is a disjoint union of $n_2$ white vertices with trees attached to it and $m_2$ labeled outgoing leaves with trees attached to them& Def. \ref{def:com}\\\hline
\vspace{0.5pt}$\tplDa(\oc{n_1}{m_1}, \oc{n_2}{m_2})$&& $x \in \plDa_\Com(\oc{n_1}{m_1}, \oc{n_2}{m_2})$ built via a specific procedure described in Def. \ref{def:tcom} & Def. \ref{def:tcom}\\\hline  \vspace{0.5pt}
    $\plDa_{cact}(n_1,n_2)$&  & $x \in \plDa_+(\oc{n_1}{0},\oc{n_2}{0})$, all white vertices of $\Gamma$ connected, $\Gamma$ is embeddable into the plane, all loops irreducible, every bound. segm. of white vert. is part of exactly one loop, one constant loop per genus& Def. \ref{def:cactdiag}\\ 
    \hline

    \end{tabular}
\end{center}
\caption{Definitions of the (sub)complexes of $\lDa$}
\label{tab:defs}
\end{table}

\begin{table}[ht]
\begin{center}
    \begin{tabular}{ | p{1.5cm} |p{1cm}|p{2.5cm}| p{6.5cm}| p{1cm} |}
    \hline
    $\cD$ & Com\allowbreak{-}plex?&Name &Definition of $x \in \cD$& Place defined in the paper \\ \midrule[2pt]
       $ \cC_d$&Yes& &the degree of a looped diagram $(\Gamma, \langle\gamma^1_1, \ldots, \gamma^{t_1}_1\rangle, \ldots, \langle\gamma^1_n, \ldots, \gamma^{t_n}_n\rangle)$ 
 is shifted by $-d \cdot \chi(\Gamma, \partial_{out})$ &Def. \ref{def:shift}\\ \hline
   $ \cC^{t_1, \cdots, t_{n_1}}$&No& \emph{irreducible looped diagrams of type $(t_1, \cdots, t_{n_1})$}&
$x=(\Gamma, \langle\gamma^1_1, \ldots, \gamma^{t_1}_1\rangle, \ldots, \langle\gamma^1_n, \ldots, \gamma^{t_n}_n\rangle)$ (for the notation cf. Section  \ref{sec:type}) such that all the loops $\gamma_i^j$ are irreducible &Def. \ref{def:type}\\ \hline
$\cC^{cst}$& Yes &p\emph{artly constant diagrams}& $x \in \cC^{t_1, \cdots, t_{n_1}}$ with at least one $t_i=0$ (spanned by those $x=(\Gamma, \gamma_1, \cdots, \gamma_{n_1})$ with one of the $\gamma_i$ constant)& Def. \ref{def:cst}\\ \hline
 
 $\cC^{>0}$& Yes &\emph{non-constant diagrams}& $x \in \cC^{t_1, \cdots, t_{n_1}}$ with all $t_i>0$ (split complement of $\cC^{cst}$)& Def. \ref{def:non-cst}\\ \hline
 $i\cC$& Yes &\emph{Products of irreducible looped diagrams}& $\cD= \prod_{t_1, \cdots, t_{n_1}} \cC^{t_1, \cdots, t_{n_1}}$, i.e. infinite sums of elements, in general composition is not well-defined& Def. \ref{def:ilD}\\ \hline
    \end{tabular}
\end{center}
\caption[Constructions applied to subcomplexes of $\lDa$]{Constructions applied to complexes $\cC \subseteq \lDa(\oc{n_1}{m_1},\oc{n_2}{m_2})$}
\label{tab:defs2}
\end{table}
\end{appendix}
\FloatBarrier
\bibliographystyle{alpha}

\begin{thebibliography}{Abb13b}

\bibitem[Abb13a]{abba13}
Hossein Abbaspour.
\newblock {On algebraic structures of the Hochschild complex}.
\newblock {\em arXiv preprint arXiv:1302.6534}, 2013.

\bibitem[Abb13b]{abba13b}
Hossein Abbaspour.
\newblock {On the Hochschild homology of open Frobenius algebras}.
\newblock {\em arXiv preprint arXiv:1309.3384}, 2013.

\bibitem[AK13]{aron13}
Gregory Arone and Marja Kankaanrinta.
\newblock The sphere operad.
\newblock {\em to appear in Bull. Lond. Math. Soc.}, 2013.

\bibitem[CS99]{chas99}
Moira Chas and Dennis Sullivan.
\newblock String topology.
\newblock {\em arXiv preprint math/9911159}, 1999.

\bibitem[CV05]{cohe05}
Ralph~L Cohen and Alexander~A Voronov.
\newblock Notes on string topology.
\newblock {\em arXiv preprint math/0503625}, 2005.

\bibitem[Get94]{getz94}
Ezra Getzler.
\newblock {Batalin-Vilkovisky algebras and two-dimensional topological field
  theories}.
\newblock {\em Communications in mathematical physics}, 159(2):265--285, 1994.

\bibitem[GH09]{gore09}
Mark Goresky and Nancy Hingston.
\newblock Loop products and closed geodesics.
\newblock {\em Duke Mathematical Journal}, 150(1):117--209, 2009.

\bibitem[God07]{godi07}
V{\'e}ronique Godin.
\newblock Higher string topology operations.
\newblock {\em Arxiv preprint arxiv:0711.4859}, 2007.

\bibitem[Kau05]{kauf05}
Ralph~M Kaufmann.
\newblock On several varieties of cacti and their relations.
\newblock {\em Algebraic \& Geometric Topology}, 5:237--300, 2005.

\bibitem[Kla13]{kla13}
Angela Klamt.
\newblock Computation of the formal operations on the {H}ochschild homology of
  commutative algebras.
\newblock {\em unpublished, available at
  \url{http://www.math.ku.dk/~angela/FormOpComAlg.pdf}}, 2013.

\bibitem[KLP03]{kauf03}
Ralph~M Kaufmann, Muriel Livernet, and Robert C~Penner.
\newblock Arc operads and arc algebras.
\newblock {\em Geometry and Topology}, 7(1):511--568, 2003.

\bibitem[Koc04]{kock2004}
Joachim Kock.
\newblock {\em {Frobenius algebras and 2D topological quantum field theories}}.
\newblock Cambridge Univ Pr, 2004.

\bibitem[Lod89]{loda89}
Jean-Louis Loday.
\newblock Op{\'e}rations sur l'homologie cyclique des alg{\`e}bres
  commutatives.
\newblock {\em Inventiones mathematicae}, 96(1):205--230, 1989.

\bibitem[LP08]{Laud2008}
Aaron~D Lauda and Hendryk Pfeiffer.
\newblock {Open-closed strings: Two-dimensional extended TQFTs and Frobenius
  algebras}.
\newblock {\em Topology and its Applications}, 155(7):623--666, 2008.

\bibitem[LS07]{lamb07}
Pascal Lambrechts and Don Stanley.
\newblock Poincar{\'e} duality and commutative differential graded algebras.
\newblock {\em arXiv preprint math/0701309}, 2007.

\bibitem[LUA08]{lupe08}
Ernesto Lupercio, Bernardo Uribe, and Miguel~A Axicotencatl.
\newblock Orbifold string topology.
\newblock {\em Geometry \& Topology}, 12:2203--2247, 2008.

\bibitem[LV12]{loda12}
Jean-Louis Loday and Bruno Vallette.
\newblock {\em Algebraic operads}, volume 346.
\newblock Springer, 2012.

\bibitem[Pen87]{penn87}
Robert~C Penner.
\newblock {The decorated Teichm{\"u}ller space of punctured surfaces}.
\newblock {\em Communications in Mathematical Physics}, 113(2):299--339, 1987.

\bibitem[TZ06]{trad06}
Thomas Tradler and Mahmoud Zeinalian.
\newblock {On the cyclic Deligne conjecture}.
\newblock {\em Journal of Pure and Applied Algebra}, 204(2):280--299, 2006.

\bibitem[Vor05]{voro05}
Alexander~A Voronov.
\newblock Notes on universal algebra.
\newblock In {\em Graphs and Patterns, Mathematics and Theoretical Physics,
  Amer. Math. Soc., Proc. Symp. Pure Math}, volume~73, pages 81--103, 2005.

\bibitem[Wah12]{wahl12}
Nathalie Wahl.
\newblock Universal operations on {H}ochschild homology.
\newblock {\em arXiv preprint arXiv:1212.6498}, 2012.

\bibitem[WW11]{wahl11}
Nathalie Wahl and Craig Westerland.
\newblock {Hochschild homology of structured algebras}.
\newblock {\em Arxiv preprint arXiv:1110.0651}, 2011.

\end{thebibliography}

\end{document}